\documentclass[11pt,fullpage,letterpaper]{article} 

\usepackage{natbib}

\setcitestyle{authoryear,round,citesep={;},aysep={,},yysep={;}}
\usepackage[english]{babel}

\usepackage{fancyhdr}
\usepackage[utf8]{inputenc} 
\usepackage[T1]{fontenc}    

\usepackage{titletoc}
\usepackage[toc, page, header]{appendix} 

\usepackage[colorlinks=true, linkcolor=blue, citecolor=blue,urlcolor=black]{hyperref}
\usepackage{url}            
\usepackage{booktabs}       
\usepackage{amsfonts}       
\usepackage{nicefrac}       
\usepackage{microtype}      
\usepackage{xcolor}         

\usepackage{indentfirst}

\title{Adversarial Network Optimization under Bandit Feedback: \\ Maximizing Utility in Non-Stationary Multi-Hop Networks}

\author{%
    Yan Dai~\thanks{ORC \& LIDS, MIT. Email: \texttt{yandai20@mit.edu}. Work done when Yan was an undergraduate student at IIIS, Tsinghua.}\and Longbo Huang~\thanks{IIIS, Tsinghua. Email: \texttt{longbohuang@tsinghua.edu.cn}.}
}
\date{}

\usepackage{amsmath}
\usepackage{amsthm}
\usepackage{amssymb}
\usepackage{mathtools}

\usepackage{cleveref} 
\newtheorem{theorem}{Theorem}[section]
\newtheorem{lemma}[theorem]{Lemma}
\newtheorem{definition}[theorem]{Definition}
\newtheorem{remark}{Remark}
\newtheorem{assumption}{Assumption}

\usepackage{algorithm}
\usepackage[noend]{algorithmic}

\usepackage{nicefrac}
\usepackage{bm,bbm}
\usepackage{comment}
\usepackage{enumitem}
\setitemize{leftmargin=*}
\setenumerate{leftmargin=*}

\usepackage{footnote}
\usepackage{multirow}
\usepackage{colortbl}

\allowdisplaybreaks
\renewcommand{\tilde}{\widetilde}
\renewcommand{\hat}{\widehat}

\renewcommand{\O}{\operatorname{\mathcal O}}
\newcommand{\Otil}{\operatorname{\tilde{\mathcal O}}}
\newcommand{\E}{\operatornamewithlimits{\mathbb{E}}}

\newcommand{\mN}{{\mathcal N}}
\newcommand{\mL}{{\mathcal L}}

\usepackage{xspace}
\newcommand{\SSAlg}{\texttt{NSO}\xspace}
\newcommand{\UMAlg}{\texttt{UMO}\textsuperscript{\texttt{2}}\xspace}
\newcommand{\OLOAlg}{\texttt{AdaPFOL}\xspace}
\newcommand{\BCOAlg}{\texttt{AdaBGD}\xspace}
\newcommand{\OLO}{\OLOAlg}
\newcommand{\BCO}{\BCOAlg}

\usepackage{cleveref}
\Crefname{ALC@line}{Line}{Lines}
\Crefformat{equation}{Eq. #2(#1)#3}
\Crefrangeformat{equation}{Eqs. #3(#1)#4 to #5(#2)#6}
\Crefmultiformat{equation}{Eqs. #2(#1)#3}{ and #2(#1)#3}{, #2(#1)#3}{ and #2(#1)#3}
\Crefrangemultiformat{equation}{Eqs. #3(#1)#4 to #5(#2)#6}{ and #3(#1)#4 to #5(#2)#6}{, #3(#1)#4 to #5(#2)#6}{ and #3(#1)#4 to #5(#2)#6}

\usepackage{tikz}
\usetikzlibrary{positioning}
\newcommand\scalemath[2]{\scalebox{#1}{\mbox{\ensuremath{\displaystyle #2}}}}

\begin{document}
\maketitle

\begin{abstract}
Stochastic Network Optimization (SNO) concerns scheduling in stochastic queueing systems. It has been widely studied in network theory.
Classical SNO algorithms require network conditions to be stationary with time, which fails to capture the non-stationary components in many real-world scenarios. Many existing algorithms also assume knowledge of network conditions before decision, which rules out applications where unpredictability presents.

Motivated by these issues, we consider Adversarial Network Optimization (ANO) under bandit feedback.
Specifically, we consider the task of \textit{i)} maximizing some unknown and time-varying utility function associated to scheduler's actions, where \textit{ii)} the underlying network is a non-stationary multi-hop one whose conditions change arbitrarily with time, and \textit{iii)} only bandit feedback (effect of actually deployed actions) is revealed after decisions.
Our proposed \UMAlg algorithm ensures network stability and also matches the utility maximization performance of any ``mildly varying'' reference policy up to a polynomially decaying gap.
To our knowledge, no previous ANO algorithm handled multi-hop networks or achieved utility guarantees under bandit feedback, whereas ours can do both.

Technically, our method builds upon a novel integration of online learning into Lyapunov analyses:
To handle complex inter-dependencies among queues in multi-hop networks, we propose meticulous techniques to balance online learning and Lyapunov arguments.
To tackle the learning obstacles due to potentially unbounded queue sizes, we design a new online linear optimization algorithm that automatically adapts to loss magnitudes.
To maximize utility, we propose a bandit convex optimization algorithm with novel queue-dependent learning rate scheduling that suites drastically varying queue lengths.
Our new insights in online learning can be of independent interest.
\end{abstract}

\section{Introduction}
Stochastic Network Optimization (SNO) studies the fundamental problem of resource allocation in a dynamic system to fulfill incoming demands, with extensive applications in real-world problems including communication networks \citep{srikant2013communication}, cloud computing \citep{maguluri2012stochastic}, and supply chains \citep{rahdar2018tri}.
There are many classical scheduling algorithms in this field enjoying performance guarantees in terms of throughput maximization \citep{tsibonis2003exploiting}, delay minimization \citep{neely2008order}, or utility maximization \citep{huang2011utility}.

Classical SNO models often assume that the {network conditions}, for example, the arrival and service rates to each queue or the capacities of data links, are stationary with respect to time. However,  many important network scenarios in practice face non-stationarity. For instance, in applications such as autonomous driving, parties in the communication networks can move rapidly \citep{ashjaei2021time}, causing the network conditions to vary from time to time. Even more, attacks such as Distributed Denial-of-Service (DDoS) or jamming can frequently happen in communication networks \citep{zou2016survey}, where arrival rates or link conditions are altered by some malicious adversary.

Moreover, we notice that existing works, even those allowing non-stationary network conditions, assumes {perfect knowledge} about the network conditions. For example, in the paper by \citet{liang2018minimizing}, the network condition is revealed at the beginning of each round, so the outcomes associated with each action can be accurately calculated, \textit{before} actually deciding and deploying the scheduler's action (see \Cref{sec:related work} for more discussions).
Nevertheless, this may again not be the case in practice.  In underwater wireless communication systems, for instance, the network conditions is unpredictable until the policy is actually executed \citep{khan2020channel}. In Internet of Things (IoT), device failures or sensor temperatures can change rapidly, resulting in highly unpredictable  traffic and channel patterns in the network \citep{gaddam2020detecting}.
Therefore, it is hard to estimate {counterfactual} outcomes of other actions (\textit{i.e.}, ``what will happen if we used a different action?'') even after deploying the action and obtaining more information about the network conditions, not to mention pre-decision evaluations.
In a nutshell, it is important and largely open to design network algorithms that are robust to time-dependent or adversarial conditions with post-decision feedback. 

Motivated by these two challenges, this paper considers optimizing an abstract utility function associated with the scheduler's action (the so-called utility maximization task \citep{neely2008fairness}) even when the network is non-stationary (which we call Adversarial Network Optimization, or ANO in short) and the feedback model is bandit style.
Specifically, in ANO, the network conditions and utility functions can unknowingly vary from time to time. Therefore, statistics of the past merely infers the current network condition, which breaks many traditional SNO techniques.
Moreover, under the bandit feedback model, the scheduler has no information about the current network condition before decision. Even worse, after making decisions, it also only receives feedback resulting from the chosen action -- not those ``counterfactual'' ones associated with other actions. 
More formally, if an action $a$ is associated with outcome $O_t(a)$ (for example, arrival and service rates) in round $t$. Then \textit{i)} the scheduler has to decide action $a_t$ without having any information about $O_t(a)$, and \textit{ii)} after playing $a_t$, it can only observe $O_t(a_t)$ but not those $O_t(a')$'s for $a'\neq a$.
Therefore, it is hard to evaluate the optimal action even in hindsight: Based on information collected in rounds $[1,t]$, one cannot accurately calculate which action had the most gain within rounds $[1,t]$.
Henceforth, in our model, one not only cannot predict the future, but also cannot fully interpolate the history.

In addition to the challenging ANO setup, bandit feedback model, and utility maximization task, we also allow the underlying network  to be \textit{multi-hop}, which means jobs can be forwarded between queues.
Despite these hardness, we succeeded in designing the utility maximization algorithm of \UMAlg, which achieves a strong performance guarantee in non-stationary multi-hop networks under bandit feedback. It not only ensures the network is stable over time, but also proves a polynomially decaying gap between our utility and any 
``mildly varying'' policy's (measured by the path length; see \Cref{thm:multi-hop utility main theorem}), similar to what can be achieved in perfect-knowledge SNO problems \citep{neely2008fairness}.

We now highlight several technical innovations in our algorithm. While our algorithm is based on the classical Lyapunov drift-plus-penalty (DPP) analysis, the adversarial network conditions breaks existing SNO arguments, which we tackle by designing online learning algorithms enjoying dynamic regret guarantees in adversarial environments.
However, due to the multi-hop topology, learning in different queues is correlated. Thus, it is highly non-trivial to decompose the problem into several online learning tasks. To this end, we develop meticulous analysis techniques that can jointly analyze the online learning algorithms and the utility maximization effects.
Moreover, the queue lengths can occasionally be large despite having a bounded expectation. Such a unique challenge is missing in online learning literature, which usually assumes the loss magnitudes are uniformly bounded by a constant.
Finally, yet another challenge is due to a combination of the multi-hop topology and the unbounded queue lengths, which makes the losses fed into online learning algorithms sometimes quite negative, a known issue for many online learning algorithms \citep{zheng2019equipping,dai2023refined}.
These two challenges make existing online learning algorithms unable to fulfill our purpose, and we propose a novel {Online Linear Optimization} algorithm (\OLOAlg; used for system stability) that adapts to drastically varying losses (proportional to queue lengths) and a new {Bandit Convex Optimization} algorithm (\BCOAlg; deployed for utility maximization) whose learning rates are carefully designed to take care of the time-dependent loss magnitudes and Lipschitzness.

\begin{table}[t]
\begin{minipage}{\textwidth}
    \caption{Comparison of Most Related Works}
    \label{table}
    \centering
    \begin{savenotes}
    \renewcommand{\arraystretch}{1.1}
    \resizebox{\textwidth}{!}{%
    \begin{tabular}{|c|c|c|c|c|c|}\hline
     & \textbf{Network} & \textbf{Arrival \&} & & & \\[-3pt]
     & \textbf{Conditions} & \textbf{Service} \footnotemark[1] \footnotetext{\footnotemark[1] \textbf{Arrival \& Service} and \textbf{Utility} columns stand for whether the arrival and service rates or the utility function associated with each feasible control action is known \textit{before decision-making}, respectively.} & \textbf{Topology} & \textbf{Objective} & \textbf{Utility} \footnotemark[1] \\\hline
    \citep{neely2008fairness} & {\color{red!70!black} Stochastic} & {\color{red!70!black} Known} & {\color{green!70!black} Multi-Hop} & {\color{green!70!black} Utility Maximization} & {\color{red!70!black} Known} \\\hline
    \citep{neely2010universal} & {\color{green!70!black} Adversarial} & {\color{red!70!black} Known} & {\color{green!70!black} Multi-Hop} & {\color{green!70!black} Utility Maximization} & {\color{red!70!black} Known} \\\hline
    \citep{liang2018minimizing} & {\color{green!70!black} Adversarial} & {\color{red!70!black} Known} & {\color{green!70!black} Multi-Hop} & {\color{red!70!black} Network Stability} & --- \\\hline
    \citep{liang2018network} & {\color{green!70!black} Adversarial} & {\color{red!70!black} Known} & {\color{green!70!black} Multi-Hop} & {\color{green!70!black} Utility Maximization} & {\color{red!70!black} Known} \\\hline
    \citep{yang2023learning} & {\color{green!70!black} Adversarial} & {\color{green!70!black} Unknown} & {\color{red!70!black} Single-Hop} & {\color{red!70!black} Network Stability} & --- \\\hline
    \citep{huang2023queue} & {\color{green!70!black} Adversarial} & {\color{green!70!black} Unknown} & {\color{red!70!black} Single-Hop} & {\color{red!70!black} Network Stability} & --- \\\hline
    \rowcolor{gray!40!white}\textbf{Ours} & {\color{green!70!black} Adversarial} & {\color{green!70!black} Unknown} & {\color{green!70!black} Multi-Hop} & {\color{green!70!black} Utility Maximization}  & {\color{green!70!black} Unknown}\\\hline
    \end{tabular}}
    \end{savenotes}
\end{minipage}
\end{table}

Finally, we mention some most related works including \citep{neely2010universal,liang2018minimizing,liang2018network,huang2023queue,yang2023learning} to help interpolate our position in the literature.
Among them, \citet{neely2010universal,liang2018minimizing,liang2018network} assumed perfect pre-decision knowledge on network conditions, which allows direct calculation of the arrival and service rates resulting from every action. In our case, we have to learn these outcomes in an online manner.
\citet{liang2018minimizing,huang2023queue,yang2023learning} focused on network stability, while our utility maximization task additionally requires maximizing an abstract, unknown, and time-varying utility function, thus adding difficulties in designing online learning algorithms.
\citet{neely2010universal,liang2018network} considered utility maximization, but their utility functions are fixed and non-adversarial. In our case, the utility functions are both time-varying and unknown, thus another online learning sub-routine is needed.
\citet{huang2023queue,yang2023learning} investigated single-hop networks where jobs leave the network upon being served, whereas our formulation considers the more general multi-hop networks.
We refer the readers to \Cref{table} and \Cref{sec:related work} for more information.

Our main contributions in this paper can be summarized as follows:
\begin{itemize}
\item We propose a novel algorithm \UMAlg (\Cref{alg:utility}) for adversarial multi-hop networks under bandit feedback, which gives rigorous utility optimization guarantee. To the best of our knowledge, no previous algorithm can handle multi-hop topology or achieve utility guarantees in adversarial networks under bandit feedback, whereas our \UMAlg algorithm is able to do both.
Moreover, as a by-product, we also derive a simpler algorithm \SSAlg (\Cref{alg:stability}) which ensures network stability for adversarial multi-hop networks under bandit feedback.
\item To handle the multi-hop topology which brings inter-queue correlations and to jointly handle online learning and network optimization, 
we develop a unified analysis that allows the integration of online learning techniques into the classical Lyapunov drift-plus-penalty arguments. 
Specifically, via the design of a new OLO algorithm to stabilize the network and a novel BCO algorithm to maximize the utility, \UMAlg algorithm enjoys a network stability guarantee together with a polynomially decaying gap between its utility and that of any policy that is ``mildly varying''
(in the sense that its path length is of order $o(\sqrt T)$; see \Cref{thm:multi-hop utility main theorem} for more details).
\item Due to the potentially unbounded queue lengths, existing online learning algorithms are unfortunately inapplicable. We design an OLO algorithm that can handle large losses and enjoys a performance guarantee adapted to the loss magnitudes (\OLOAlg; \Cref{thm:multi-hop stability decision making}). We also develop a new BCO method specially crafted for the drastically varying loss magnitudes and Lipschitzness (\BCOAlg; \Cref{thm:multi-hop utility decision making}). Both online learning algorithms can be of independent interest.
\end{itemize}

\subsection{Related Works}\label{sec:related work}
We discuss the most related works here. A more comprehensive literature review is in \Cref{sec:more related work}.

\textbf{Adversarial Network Control.}
Adversarial networks date back to the 1990s, when \citet{cruz1991calculus} gave the first adversarial dynamics network model and its scheduling algorithm. More efforts were made to allow more general arrival rates \citep{borodin2001adversarial,andrews2001universal}, link conditions \citep{andrews2004scheduling,andrews2007stability}, or both \citep{liang2018minimizing}. We also direct the readers to the references therein for more discussions.
The main focus of the aforementioned papers were usually system stability, whereas ours is utility maximization. As we are aware of, existing results on utility maximization \citep{neely2010universal,liang2018network} mostly assumed perfect knowledge on network conditions.

\textbf{Feedback Models.}
Most previous works considered perfect knowledge model which assumes pre-decision knowledge on network conditions \citep{liang2018minimizing,liang2018network}.
In contrast, our paper considers bandit feedback model which only reveals the consequence of our action. A small number of previous works \citep{fu2022joint,yang2023learning,huang2023queue} also assumed similar feedback models albeit under different names.
Another feedback model whose difficulty lies in between is full-information feedback model, which requires network conditions to be revealed after decision and thus counterfactual evaluations of all actions (\textit{i.e.}, not only the deployed one) are allowed in hindsight. See \citep{neely2012max} for an example.

\textbf{Adversarial Networks under Bandit Feedback.}
Prior to our work, \citet{huang2023queue,yang2023learning} also studied adversarial networks under bandit feedback. However, they both assumed single-hop networks and focused on network stability. In contrast, our paper allows a general multi-hop topology and tackles the utility maximization task of optimizing an abstract, unknown, and time-varying utility function.


\section{Notations and Preliminaries}\label{sec:preliminaries}
We use bold letters to denote vectors, \textit{e.g.}, $\bm q_t,\bm \mu_t,\bm \lambda_t$, and denote their elements with corresponding normal letters, \textit{e.g.}, $q_{t,i},\mu_{t,i},\lambda_{t,i}$.
For an integer $n\ge 0$, $[n]$ stands for $\{1,2,\ldots,n\}$. For a finite set $\mathcal S$, $\triangle(\mathcal S)$ is the simplex over $\mathcal S$, \textit{i.e.}, $\{\bm x\in \mathbb R^{\lvert \mathcal S\rvert}\mid \sum_{i=1}^{\lvert \mathcal S\rvert} x_i=1\}$, where every element $\bm x\in \triangle(\mathcal S)$ is a discrete probability distribution over $\mathcal S$.
We use $\O$ to hide all absolute constants, and use $\Otil$ to additionally hide all logarithmic factors.
For functions $f(T)$ and $g(T)$, we say $f(T)=\O_T(g(T))$ if $\limsup_{T\to \infty} \frac{f(T)}{g(T)}<\infty$ and $f(T)=o_T(g(T))$ if $\limsup_{T\to \infty} \frac{f(T)}{g(T)}=0$.

\subsection{Adversarial Network Optimization under Bandit Feedback Formulation}

We first introduce our adversarial network optimization with bandit feedback model. 
Specifically, in a network with multiple servers and directional data links, we denote the set of all servers by $\mN$ and that of all data links by $\mL\subseteq \mN\times \mN$. Suppose that $\lvert \mN\rvert$ and $\lvert \mL\rvert$ are both finite.
There are $\lvert \mN\rvert$ commodities of jobs such that those jobs belonging to commodity $k\in \mN$ are destined for server $k\in \mN$.
We denote $Q_n^{(k)}$ as the queue of unfinished commodity-$k$ jobs at server $n$, where $n\in \mN$ and $k\in \mN$. We assume the links do not interfere with each other. 

The scheduling problem lasts for $T>0$ rounds. In round $t\in [T]$, the scheduler makes two decisions: \textit{i)} \textit{arrival rates} of commodity-$k$ jobs into server $n$, $\forall n\in \mN,k\in \mN$, and \textit{ii)} \textit{link rate allocations} of transmitting how many commodity-$k$ jobs over data link $(n,m)$, $\forall k\in \mN,(n,m)\in \mL$. Both decisions are made under the bandit feedback model, \textit{i.e.}, the scheduler makes decisions in blind and only receives feedback resulting from its actions. 
Below, we describe them in detail. 

\textbf{Arrival Rates and Utility.} In every round $t$, the scheduler decides an $\lvert \mN\rvert\times \lvert \mN\rvert$ dimensional arrival rate matrix $\bm \lambda(t)$ from some fixed action set $\Lambda\subseteq \mathbb R_{\ge 0}^{\lvert \mN\rvert\times \lvert \mN\rvert}$, and consequently, $\lambda_n^{(k)}(t)$ jobs with commodity $k$ will be added to queue $Q_n^{(k)}$.
The arrival rate vector $\bm \lambda(t)$ is associated with some abstract utility $g_{t}(\bm \lambda(t))$ where $g_t\colon \Lambda\to \mathbb R$ is concave (that is, the user's marginal return diminishes gradually as the arrivals increase \citep{huang2011utility,huang2012lifo}), $L$-Lipschitz, and $[-G,G]$-bounded, where $L$ and $G$ here are known constants.
Following the adversarial network assumption, we allow $g_t$'s to be time-dependent (though they have to be pre-determined, which is called the \textit{oblivious adversary} model).
Following the bandit feedback model, the scheduler has no information about $g_t$ before the decision, and can only observe $g_t(\bm \lambda(t))$ for the chosen $\bm \lambda(t)$ but not the whole $g_t$ after decision. 

\textbf{Link Rate Allocations.} 
The capacity of each link $(n,m)\in \mL$ can be time-varying. We denote the capacity of $(n,m)$ in round $t$ as $C_{n,m}(t)$. We assume the capacities are always bounded by some finite constant $M$.
Due to the bandit feedback model, the scheduler cannot access $C_{n,m}(t)$ when deciding. Nevertheless, the scheduler can still decide a \textit{link allocation} plan which assigns a distribution over commodities on each link, or formally denoted as $\bm a_{n,m}(t)\in \triangle(\mN)$ (the $|\mN|$-dimension  distribution simplex, representing the portion of rates allocated to each commodity over the link).
Via sending jobs from each commodity along link $(n,m)$ according to distribution $\bm a_{n,m}(t)$ in a round-robin manner, approximately $a_{n,m}^{(k)}(t) C_{n,m}(t)$ jobs from queue $Q_n^{(k)}$ will be sent along link $(n,m)$ to queue $Q_m^{(k)}$.
Formally, we assume that after deciding link allocation plans $\{\bm a_{n,m}(t)\in \triangle(\mN)\}_{(n,m)\in \mL}$, the number of jobs successfully sent from $Q_n^{(k)}$ to $Q_m^{(k)}$, denoted by $\mu_{n,m}^{(k)}(t)$, are independently generated such that $\E[\mu_{n,m}^{(k)}(t)]=C_{n,m}(t) a_{n,m}^{(k)}(t)$ and $\mu_{n,m}^{(k)}(t)\in [0,M]$.
Again, we assume a bandit fededback model, which means the scheduler is able to observe $C_{n,m}(t)$ and $\bm \mu_{n,m}(t)$ for all $(n,m)\in \mL$ only after the decision is made at the end of round $t$.

Putting the two components together, by denoting the length of $Q_n^{(k)}$ at the beginning of round $t$ to be $Q_n^{(k)}(t)$, the network dynamics can then be characterized as follows:
\begin{equation}\label{eq:multi-hop stability dynamics}
Q_n^{(k)}(t+1)=\begin{cases}
\left [Q_n^{(k)}(t)-\sum_{(n,m)\in \mL}\mu_{n,m}^{(k)}(t)\right ]_+ + \sum_{(o,n)\in \mL} \mu_{o,n}^{(k)}(t) + \lambda_{n}^{(k)}(t),&k\ne n\\
0,&k=n
\end{cases},
\end{equation}
where $\lambda_n^{(k)}(t)$ is the number of jobs with commodity $k$ that the scheduler adds to server $n$, and $\mu_{n,m}^{(k)}(t)$ is the number of jobs with commodity $k$ transmitted along data link $(n,m)$.

The objective of the scheduler is to maximize its average utility over the $T$ rounds, namely $\frac 1T \E[\sum_{t=1}^T g_t(\bm \lambda(t))]$.
However, a scheduling algorithm is meaningless if it cannot ensure network stability, which requires the average number of jobs remaining in the network is non-divergent when the number of rounds is large enough.
Formally, the network stability requirement says
\begin{equation}
\frac 1T \E\left [\sum_{t=1}^T \lVert \bm Q(t)\rVert_1\right ]=\frac 1T\E\left [\sum_{t=1}^T \sum_{n\in \mN} \sum_{k\in \mN} Q_{n}^{(k)}(t)\right ]=\O_T(1),\quad \text{when }T\gg 0.\label{eq:average queue}
\end{equation}
The scheduler aims to maximize its average utility subject to the network stability condition, \textit{i.e.},
\begin{equation}\label{eq:average reward}
\text{Maximize }\frac 1T \E\left [\sum_{t=1}^T g_t(\bm \lambda(t))\right ]\text{ s.t. \Cref{eq:average queue} holds}.
\end{equation}

\subsection{Technical Overview of Our Paper}
In order to improve presentation and facilitate understanding, in \Cref{sec:multi-hop stability}, we first present the network stability algorithm \SSAlg, \textit{i.e.}, pretending $g_t$ is a constant. This algorithm will serve as a key building block for the utility maximization algorithm \UMAlg that we introduce in \Cref{sec:multi-hop utility}.

In \Cref{fig:flowchart}, we give an overview of our main technical steps when analyzing \SSAlg and \UMAlg. The steps for \SSAlg are in yellow, the ones for \UMAlg are in blue, and those in common are in green.
In general, either analysis starts from the famous Lyapunov drift(-plus-penalty) analysis \citep[\S 4]{neely2010stochastic}, which reveals the non-negativity of a Lyapunov drift(-plus-penalty) function -- see \Cref{sec:multi-hop stability Lyapunov} and \Cref{sec:multi-hop utility Lyapunov} for more details.
We then use online learning techniques to minimize them (\textit{i.e.}, making them as close to zero as possible). From here, the analyses for \SSAlg and \UMAlg become different.

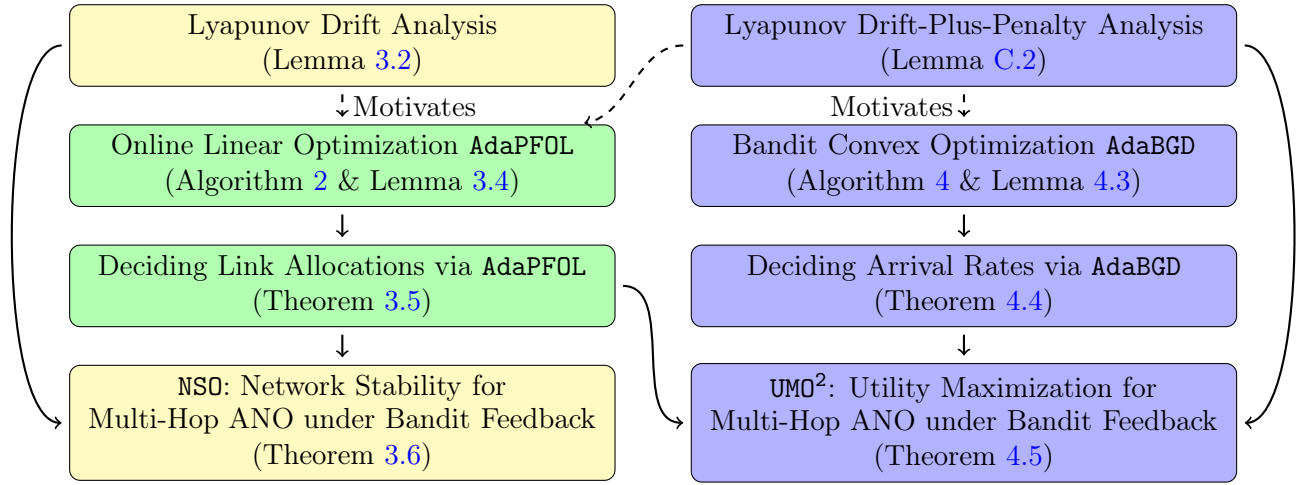
\begin{figure}[t!]
\centering
\begin{tikzpicture}[block/.style={rounded corners, minimum width=7.25cm, draw, node distance=0.5cm and 1cm}]
\node[block, fill=yellow!30!white] (3) {\shortstack{Lyapunov Drift Analysis\\(\Cref{lem:multi-hop stability Lyapunov})}};
\node[block, below=of 3, fill=green!30!white] (5) {\shortstack{Online Linear Optimization \OLOAlg\\(\Cref{alg:OLO unbounded} \& \Cref{lem:OLO unbounded guarantee})}};
\node[block, below=of 5, fill=green!30!white] (7) {\shortstack{Deciding Link Allocations via \OLOAlg\\(\Cref{thm:multi-hop stability decision making})}};
\node[block, below=of 7, fill=yellow!30!white] (9) {\shortstack{\SSAlg: Network Stability for\\Multi-Hop ANO under Bandit Feedback\\(\Cref{thm:multi-hop stability main theorem})}};
\node[block, right=of 3, fill=blue!30!white] (4) {\shortstack{Lyapunov Drift-Plus-Penalty Analysis\\(\Cref{lem:multi-hop utility Lyapunov})}};
\node[block, right=of 5, fill=blue!30!white] (6) {\shortstack{Bandit Convex Optimization \BCOAlg\\(\Cref{alg:ABGD} \& \Cref{lem:ABGD guarantee})}};
\node[block, right=of 7, fill=blue!30!white] (8) {\shortstack{Deciding Arrival Rates via \BCOAlg\\(\Cref{thm:multi-hop utility decision making})}};
\node[block, right=of 9, fill=blue!30!white] (10) {\shortstack{\UMAlg: Utility Maximization for\\Multi-Hop ANO under Bandit Feedback\\(\Cref{thm:multi-hop utility main theorem})}}; 
\begin{scope}[->, shorten >=1mm, shorten <=1mm, thick]
\draw [dashed] (3) to[out=270, in=90, edge node={node [right] {Motivates}}] (5);
\draw (3) to[out=180,in=180, looseness=0.5] (9);
\draw (5) to[out=270,in=90] (7);
\draw (7) to[out=270,in=90] (9);
\draw (4) to[out=0,in=0, looseness=0.5] (10);
\draw [dashed] (4) to[out=180, in=10] (5);
\draw [dashed] (4) to[out=270, in=90, edge node={node [left] {Motivates}}] (6);
\draw (6) to[out=270,in=90] (8);
\draw (7) to[out=0,in=180] (10);
\draw (8) to[out=270,in=90] (10);
\end{scope}
\end{tikzpicture}
\caption{Technical Overview of \SSAlg and \UMAlg Frameworks}\label{fig:flowchart}
\end{figure}

For the network stability algorithm \SSAlg, we succeeded in expressing the Lyapunov drift function as a function linear in the queue lengths $\bm Q(t)$ and the link allocation plan $\bm a(t)$.
While this belongs to the classical Online Linear Optimization (OLO) problem in online learning \citep{zinkevich2003online}, we face two unique challenges due to the potentially unbounded queue lengths and the self-bounding analysis for network stability guarantees; see \Cref{sec:multi-hop stability sketch} for more discussions.
These two requirements rules out existing OLO algorithms, and thus we have to design our own algorithm crafted towards the network optimization objective. Specifically, we designed an OLO algorithm \texttt{AdaPFOL} (see \Cref{alg:OLO unbounded}) that can handle occasionally large loss magnitudes and ensures a performance guarantee depending on all the losses.
When plugging it into \SSAlg, we are able to see that the link allocation plans perform well, as detailed in \Cref{thm:multi-hop stability decision making}. Therefore, combining it with some reference policy assumption (\Cref{def:multi-hop stability network stability}) and the Lyapunov drift analysis, we obtain the network stability guarantee of \SSAlg in \Cref{thm:multi-hop stability main theorem}.
A more detailed overview of \SSAlg is in \Cref{sec:multi-hop stability sketch}.

Regarding the utility maximization algorithm \UMAlg, we have to decompose the Lyapunov drift-plus-penalty function into two parts. The first part is still linear in $\bm Q(t)$ and $\bm a(t)$, which we can  reuse the \OLOAlg algorithm. For the second part, as the utility function $g_t$ is an arbitrary time-varying concave function and we only receive bandit feedback (recall \Cref{table}), OLO cannot capture it. Instead, we model this part as a Bandit Convex Optimization (BCO) problem \citep{flaxman2005online}. Unfortunately, again due to the potentially unbounded queue lengths and the self-bounding analysis, the loss functions' magnitudes and Lipschitzness can be very large for some rounds but we still want to adapt to them -- thus, existing BCO algorithms are inapplicable either.
To this end, we develop a BCO algorithm \BCOAlg (\Cref{alg:ABGD}) which allows loss functions with large magnitudes or Lipschitzness and enjoys a performance adaptive to the loss functions. When plugging in \BCOAlg to the \UMAlg framework, it can generate a good arrival rate sequence as we analyze in \Cref{thm:multi-hop utility decision making}.
Therefore, similar to the analysis of \SSAlg in \Cref{thm:multi-hop stability main theorem}, if we combine the \OLOAlg in \Cref{alg:OLO unbounded} and the \BCOAlg in \Cref{alg:ABGD} together, we are able to derive the utility maximization guarantee of \UMAlg in \Cref{thm:multi-hop utility main theorem}. Again, a more detailed overview of \UMAlg can be found in \Cref{sec:multi-hop utility sketch}.

\section{Network Stability in Adversarial Multi-Hop Networks}\label{sec:multi-hop stability}

In this section, we do a first step towards our ultimate goal of multi-hop utility maximization, which is network stability (recall that \Cref{eq:average reward} requires \Cref{eq:average queue} as a condition).
That is, this section only focuses on stablizing the average number of tasks in the system $\frac 1T \E[\sum_{t=1}^T \lVert \bm Q(t)\rVert_1]$ and does not consider utilities.
The algorithm designed for this purpose (\SSAlg in \Cref{alg:stability}) will serve as the network stability component of our utility maximization algorithm (\UMAlg in \Cref{alg:utility}).

One may observe that if we only want to ensure the network stability condition in \Cref{eq:average queue}, it suffices to pick all arrival rate vectors $\bm \lambda(t)\equiv 0$.
To avoid such trivial algorithms, we assume the arrival rates are adversarially chosen for now. That is, $\bm \lambda(t)\in [0,R]^{\lvert \mN\rvert\times \lvert \mN\rvert}$ is some arbitrary, unknown, and time-varying vector following the oblivious adversary model. It is only revealed post-decision, at the end of round $t$.
The rationale of assuming adversarial $\bm \lambda(t)$ is because in the \UMAlg algorithm for utility maximiztion, another algorithmic component decides $\bm \lambda(t)$ and the network stability component that we design in this section must adapt to such an arbitrary arrival rate matrix $\bm \lambda(t)$.

\begin{algorithm}[t!]
\caption{\SSAlg: \textbf{N}etwork \textbf{S}tability via Online Linear \textbf{O}ptimization}
\label{alg:stability}
\begin{algorithmic}[1]
\REQUIRE{Number of rounds $T$, set of servers $\mN$ and links $\mL$, maximum capacity $M$, feasible arrival rates $\Lambda$ (adversarial arrival rates given during execution -- only assumed in this section). An online linear optimization algorithm $\OLO$ (\Cref{alg:OLO unbounded}).}
\STATE For each link $(n,m)\in \mL$, initialize an instance of $\OLO$ with action set $\triangle(\mN)$ as $\OLO_{n,m}$.
\FOR{$t=1,2,\ldots,T$}
\STATE For each link $(n,m)\in \mL$, pass the maximum loss magnitude for this round $M\lVert \bm Q_m(t)-\bm Q_n(t)\rVert_\infty$ to $\OLO_{n,m}$. Pick link allocation $\bm a_{n,m}(t)\in \triangle(\mN)$ as the output of $\OLO_{n,m}$.
\STATE Observe arrival rates $\bm \lambda(t)\in \Lambda$. \COMMENT{In the utility maximization algorithm \UMAlg (\Cref{alg:utility}), this step will be replaced by another algorithmic component.}
\STATE Observe capacities $\{C_{n,m}(t)\}_{(n,m)\in \mL}$ and actual data transmissions $\{\mu_{n,m}^{(k)}(t)\}_{(n,m)\in \mL,k\in \mN}$.
\STATE Calculate queue lengths $\bm Q(t+1)$ from $\bm Q(t)$ according to \Cref{eq:multi-hop stability dynamics}.
\STATE For each link $(n,m)\in \mL$, pass the loss vector $C_{n,m}(t)(\bm Q_m(t)-\bm Q_n(t))$ to $\OLO_{n,m}$.
\ENDFOR
\end{algorithmic}
\end{algorithm}

\subsection{Motivation of Our Algorithmic Framework}\label{sec:multi-hop stability sketch}

In \Cref{alg:stability}, we present \textbf{N}etwork \textbf{S}tability via Online Linear \textbf{O}ptimization (\SSAlg), an algorithmic framework which achieves stability in adversarial multi-hop networks under bandit feedback.
One key ingredient of \SSAlg is the plug-in \textit{Online Linear Optimization} (OLO) algorithm $\OLO$.
Before going into details of the \OLOAlg algorithm, we first introduce why we need it.

The design of \SSAlg is based on the famous Lyapunov drift analysis \citep[\S 4]{neely2010stochastic}. Conducting standard Lyapunov analysis on the network dynamics defined in \Cref{eq:multi-hop stability dynamics}, we are able to derive
\begin{align}
&\quad -\frac 12 N^2 ((NM)^2+2(NM)^2+2R^2) T \nonumber\\
&\le \E\left [\sum_{t=1}^T \sum_{n\in \mN} \sum_{k\in \mN} Q_{n}^{(k)}(t) \left (\sum_{(o,n)\in \mL} \mu_{o,n}^{(k)}(t) + \lambda_n^{(k)}(t) - \sum_{(n,m)\in \mL} \mu_{n,m}^{(k)}(t) \right )\right ],\label{eq:multi-hop stability Lyapunov informal}
\end{align}
whose formal statement and proof can be found in \Cref{lem:multi-hop stability Lyapunov}.

Based on this inequality, a Lyapunov drift based algorithm can be constructed by  minimizing the RHS of \Cref{eq:multi-hop stability Lyapunov informal} \citep[\S 4]{neely2010stochastic}.  
As the arrival rate $\lambda$'s is regarded as a constant in this section, we may only focus on the terms related to $\mu$. Thus, minimizing RHS of \Cref{eq:multi-hop stability Lyapunov informal} is equivalent to
\begin{align}
\text{Minimizing }&\quad \E\left [\sum_{t=1}^T \sum_{(n,m)\in \mL} \sum_{k=1}^K \mu_{n,m}^{(k)}(t) (Q_m^{(k)}(t)-Q_n^{(k)}(t))\right ] \nonumber\\
&=\E\left [\sum_{t=1}^T \sum_{(n,m)\in \mL} \langle \bm Q_m(t)-\bm Q_n(t),\bm \mu_{n,m}(t)\rangle\right ]\nonumber\\
&=\E\left [\sum_{t=1}^T \sum_{(n,m)\in \mL} \langle C_{n,m}(t) (\bm Q_m(t)-\bm Q_n(t)),\bm a_{n,m}(t)\rangle\right ], \label{eq:multi-hop stability OLO objective informal}
\end{align}
where the last step uses the assumption that $\E[\bm \mu_{n,m}(t)]=\E[C_{n,m} \bm a_{n,m}(t)]$.

For illustration purposes, let us focus on a single data link $(n,m)\in \mL$. Motivated by \citet{huang2023queue}, we consider designing a scheduling algorithm via minimizing the following expectation:
\begin{equation*}
\E\left [\sum_{t=1}^T \langle C_{n,m}(t) (\bm Q_m(t)-\bm Q_n(t)),\bm a_{n,m}(t)\rangle\right ].
\end{equation*}

\begin{remark}
While having similarities, this objective is different from that of \citet{huang2023queue} in two aspects:

First, the network topology in \citep{huang2023queue} is a single-server, single-hop one, thus it suffices to conduct the Lyapunov drift optimization on the centralized server. In contrast, due to our multi-hop topology, our optimization task \Cref{eq:multi-hop stability OLO objective informal} has to be distributed onto every data link $(n,m)\in \mL$ and extra efforts are needed to ensure a good overall scheduling effect.

Second, the coefficient $C_{n,m}(t) (Q_m^{(k)}(t)-Q_n^{(k)}(t))$ before $a_{n,m}^{(k)}(t)$ can be either positive or negative, whereas that of \citet{huang2023queue} is always non-negative. Such a negativity also increases the difficulty as many online learning algorithms are typically bad at handling potentially negative losses, see, \textit{e.g.}, \citep{zheng2019equipping,dai2023refined}.
\end{remark}

Recall that $\bm a_{n,m}(t)\in \triangle(\mN)$ can be any probability distribution from the simplex. Hence, one may view $\triangle(\mN)$ as the action set in round $t$.
Moreover, for an action $\bm a$ from the action set $\triangle(\mN)$, picking it in round $t$ will incur a loss $\langle C_{n,m}(t) (\bm Q_m(t)-\bm Q_n(t)),\bm a\rangle$ -- which is linear in $\bm a$. Thus, this problem belongs to the class of Online Linear Optimization (OLO) problems \citep{zinkevich2003online,mcmahan2010adaptive,duchi2011adaptive}, whose formal definition will be presented later as \Cref{def:OLO}.
Although our problem belongs to the OLO formulation, we face significantly different challenges due to our network optimization context:
\begin{enumerate}
\item[\textit{i)}] In our task of minimizing $\E[\sum_{t=1}^T \langle C_{n,m}(t) (\bm Q_m(t)-\bm Q_n(t)),\bm a_{n,m}(t)\rangle]$, the magnitude of the loss $C_{n,m}(t) (\bm Q_m(t) - \bm Q_n(t))$ can occasionally be large because $\bm Q_n(t)$ or $\bm Q_m(t)$ may be unbounded.
Note that, despite our system stability condition \Cref{eq:average queue} requires $\frac 1T \E[\sum_{t=1}^T \lVert \bm Q(t)\rVert_1]$ to be small, some $\bm Q(t)$'s inside average and expectation are still allowed to be large.
However, existing algorithms in OLO mostly require the losses to be uniformly bounded by a constant (see, \textit{e.g.}, \citep{mcmahan2010adaptive,cutkosky2020parameter}), which means extra efforts should be made to handle occasionally large losses.
\item[\textit{ii)}] Moreover, we also want our performance to depend on the geometric mean of all loss magnitudes (which in turn relates to queue lengths since the losses $C_{n,m}(t) (\bm Q_m(t) - \bm Q_n(t))$ depend on $\lVert \bm Q(t)\rVert$). On a high level, this can be understood as follows: The average queue length $\frac 1T\E[\sum_{t=1}^T \lVert \bm Q(t)\rVert_1]$ can be controlled by the OLO performance via some other arguments (see \Cref{sec:multi-hop stability assumption}). Thus, if we can additionally show that OLO performance is bounded by queue lengths, we can conduct a self-bounding analysis on the queue lengths which informally reads
\begin{align}
&\quad \E\left [\sum_{t=1}^T \lVert \bm Q(t)\rVert_1\right ]\lesssim \text{Online Learning Performance} \lesssim \O_T(T) + o_T(T^{1/4}) \E\left [\sum_{t=1}^T \lVert \bm Q(t)\rVert_1\right ]^{3/4} \nonumber \\
&\Longrightarrow \frac 1T \E\left [\sum_{t=1}^T \lVert \bm Q(t)\rVert_1\right ] = \O_T(1),\textit{ i.e.},\text{ the system is stabilized}.\label{eq:self-bounding informal}
\end{align}
More details regarding the self-bounding analysis can be found in \Cref{sec:multi-hop stability main theorem}. Nevertheless, it suffices to remember that a performance depending on all loss magnitudes is beneficial.
\end{enumerate}

In \Cref{sec:multi-hop stability OLO}, we introduce our novel OLO algorithm of \OLOAlg (\Cref{alg:OLO unbounded}) that enjoys these two properties.
Equipped with such an algorithm, we can minimize \Cref{eq:multi-hop stability OLO objective informal} and achieve network stability guarantee by deploying it on every link $(n,m)\in \mL$ for $T$ rounds with action set $\triangle(\mN)$. This idea of deploying \OLOAlg onto every link exactly gives our \SSAlg framework in \Cref{alg:stability}.

Therefore, we are able to analyze the network stability effect when the \SSAlg framework is equipped with \OLOAlg in \Cref{alg:OLO unbounded}, which we do in the rest of this section:
We introduce our reference policy assumptions in \Cref{sec:multi-hop stability assumption}, conduct the Lyapunov drift analysis in \Cref{sec:multi-hop stability Lyapunov}, introduct and analyze the novel \OLOAlg algorithm in \Cref{sec:multi-hop stability OLO}, and present our final analysis in \Cref{sec:multi-hop stability main theorem}.

\subsection{Reference Policy Assumption}\label{sec:multi-hop stability assumption}
We first make the following \textit{multi-hop piecewise stability} assumption, which, informally speaking, assumes that there exists a reference policy that stabilizes the system piecewisely.
It is an extension of the piecewise stability assumption \citep[Assumption 1]{huang2023queue} to multi-hop cases.
\begin{assumption}[Multi-Hop Piecewise Stability for Network Stability]\label{def:multi-hop stability network stability}
There exists a reference action sequence $\{\mathring{\bm a}(t)\}_{t\in [T]}$ (where $\mathring{\bm a}(t)=\{\mathring{\bm a}_{n,m}(t)\in \triangle(\mN)\}_{(n,m)\in \mL}$, in analogue to the scheduler's action sequence), such that
there are some constants $C_W\ge 0$, $\epsilon_W\ge 0$ and a partition $W_1,W_2,\ldots,W_J$ of $[T]$,\footnote{A partition $W_1,W_2,\ldots,W_N$ of $[T]$ is a collection of a few non-intersecting intervals whose union is $[T]$} which ensure that $\sum_{j=1}^J (\lvert W_j\rvert-1)^2\le C_W T$ and
\begin{align}
&\frac{1}{\lvert W_j\rvert} \sum_{t\in W_j} \sum_{(n,m)\in \mL} C_{n,m}(t) \mathring a_{n,m}^{(k)}(t) \ge \epsilon_W + \frac{1}{\lvert W_j\rvert} \sum_{t\in W_j} \left (\lambda_n^{(k)}(t) + \sum_{(o,n)\in \mL} C_{o,n}(t) \mathring a_{o,n}^{(k)}(t)\right ), \nonumber\\
&\quad \forall j\in [J],n\in \mN,k\in \mN,\label{eq:multi-hop stability network stability}
\end{align}
where $\lambda_n^{(k)}(t)$ is the obliviously decided arrival rates that we assume in this section.
\end{assumption}

Intuitively, \Cref{def:multi-hop stability network stability} means that there exists some ``good'' action sequence $\{\mathring{\bm a}(t)\}_{t\in [T]}$ making the network stable, in the sense that there are multiple windows $W_1,W_2,\ldots,W_J$ such that in expectation, for each window $W_j$ and for each queue $Q_n^{(k)}$, the average service rate it receives (whose expectation is $\sum_{(n,m)\in \mL} C_{n,m}(t) \mathring a_{n,m}^{(k)}(t)$ in round $t$), is strictly more than its net arrival rate (which includes both external data flows $\lambda_n^{(k)}(t)$ and internal data flows that are forwarded from other queues $\sum_{(o,n)\in \mL} C_{o,n}(t) \mathring{a}_{o,n}^{(k)}(t)$), by a constant gap of at least $\epsilon_W$.

\begin{remark}\label{remark:assumption}
Such assumptions are typical in network optimization literature. In the case when the network is stationary, \Cref{def:multi-hop stability network stability} recovers the classical capacity region assumption in SNO  \citep{neely2010stochastic}.
However, extending this condition to adversarial network is highly non-trivial.
For adversarial networks, an alternative assumption is the $(W,\epsilon)$-constrained dynamics assumption \citep{liang2018minimizing}, which roughly says \Cref{eq:multi-hop stability network stability} holds for every window of size $W$. \Cref{def:multi-hop stability network stability} thus allows more flexibility.
Finally, our \Cref{def:multi-hop stability network stability} can be viewed as a generalization of the piecewise stability assumption \citep{huang2023queue}, which was crafted for a single centralized server.
\end{remark}

Before moving on, we shall remark that the reference action sequence $\{\mathring{\bm a}(t)\}_{t\in [T]}$ in \Cref{def:multi-hop stability network stability} is \textit{unknown} to the scheduler. Instead, the scheduler needs to learn its own way of stabilizing the network via observations. 
To characterize the ability of $\{\mathring{\bm a}(t)\}_{t\in [T]}$ in stabilizing the network, the following lemma controls the average queue length resulting from any scheduling policy.
\begin{lemma}[Ability of $\{\mathring{\bm a}(t)\}_{t\in [T]}$ in Stabilizing the Network]\label{lem:multi-hop stability network stability}
If $\{\mathring{\bm a}(t)\}_{t\in [T]}$ satisfies \Cref{def:multi-hop stability network stability}, then for any scheduler-generated queue lengths $\{\bm Q(t)\}_{t\in [T]}$,
\begin{align*}
&\quad \epsilon_W \E\left [\sum_{t=1}^T \sum_{n\in \mN} \sum_{k\in \mN} Q_n^{(k)}(t)\right ] - (N^2 (2NM+R)^2 + \epsilon_W N^2 (2NM+R)) C_W T \\
&\le-\E\left [\sum_{t=1}^T \sum_{(n,m)\in \mL} \sum_{k\in \mN} C_{n,m}(t) \mathring{a}_{n,m}^{(k)}(t) (Q_m^{(k)}(t) - Q_n^{(k)}(t))\right ]-\E\left [\sum_{t=1}^T \sum_{n\in \mN} \sum_{k\in \mN} Q_n^{(k)}(t) \lambda_n^{(k)}(t)\right ].
\end{align*}
\end{lemma}

\Cref{lem:multi-hop stability network stability} says the Lyapunov drift (defined later) under $\{\mathring{\bm a}(t)\}_{t=1}^T$ is always negative, which is useful when analyzing queue-based policies \citep[\S 3.1]{neely2010stochastic}. Its proof can be found in \Cref{sec:multi-hop stability network stability appendix}.

\subsection{Lyapunov Drift Analysis}\label{sec:multi-hop stability Lyapunov}
We carry out our analysis based on the Lyapunov drift analysis \citep[\S 4]{neely2010stochastic}, which considers the Lyapunov function $L_t$ and its drift $\Delta(\bm Q(t))$, defined as follows:
\begin{equation*}
\text{Lyapunov function }L_t=\frac 12 \sum_{n\in \mN} \sum_{k\in \mN} \left (Q_{n}^{(k)}\right )^2,\quad \text{Lyapunov drift } \Delta(\bm Q(t)) = \E[L_{t+1} - L_t \mid \bm Q(t)].
\end{equation*}

We give the following result which is almost diretly applying the classical Lyapunov drift analysis to the queue dynamics in \Cref{eq:multi-hop stability dynamics}. The proof is standard and thus deferred to \Cref{sec:multi-hop stability Lyapunov appendix}.
\begin{lemma}[Lyapunov Drift Analysis]\label{lem:multi-hop stability Lyapunov}
Under the queue dynamics of \Cref{eq:multi-hop stability dynamics},
\begin{align}
0\le \E\left [\sum_{t=1}^T \Delta(\bm Q(t))\right ]&\le \E\left [\sum_{t=1}^T \sum_{(n,m)\in \mL} \sum_{k\in \mN} \mu_{n,m}^{(k)}(t)\left (Q_m^{(k)}(t)-Q_n^{(k)}(t)\right ) + \sum_{n\in \mN} \sum_{k\in \mN} Q_n^{(k)}(t) \lambda_n^{(k)}(t)\right ] + \nonumber \\
&\quad \frac 12 N^2 ((NM)^2+2(NM)^2+2R^2) T. \label{eq:multi-hop stability Lyapunov}
\end{align}
\end{lemma}

As sketched in \Cref{sec:multi-hop stability sketch}, our algorithm is designed to approximately  minimize the RHS of \Cref{eq:multi-hop stability Lyapunov} via online learning, which contains two non-constant terms $\E[\sum_{(n,m)\in \mL}\langle \bm \mu_{n,m}(t),\bm Q_m(t)-\bm Q_n(t)\rangle]$ and $\E[\sum_{n\in \mN} \langle \bm Q_n(t), \bm \lambda_n(t)\rangle]$.
As $\lambda_n^{(k)}(t)$ are obliviously chosen, the second term is also constant. Therefore, it remains to minimize the following term: 
\begin{equation}\label{eq:multi-hop OLO objective}
\E\left [\sum_{t=1}^T \langle \bm \mu_{n,m}(t),\bm Q_m(t)-\bm Q_n(t)\rangle\right ]=\E\left [\sum_{t=1}^T \langle C_{n,m}(t) (\bm Q_m(t)-\bm Q_n(t)), \bm a_{n,m}(t)\rangle\right ],\quad \forall (n,m)\in \mL.
\end{equation}

For each data link $(n,m)\in \mL$, \Cref{eq:multi-hop OLO objective} corresponds to an {Online Linear Optimization} (OLO) problem with $ \langle C_{n,m}(t) (\bm Q_m(t)-\bm Q_n(t)), \bm a_{n,m}(t)\rangle$ being the loss for round $t$.
In the next section, we first rigorously define the OLO problem in \Cref{def:OLO} and then present our novel algorithm that tackles the unique challenges we face in network optimization contexts -- potentially large losses due to unbounded queue lengths (recall \Cref{eq:multi-hop OLO objective}), and adapting to all the loss magnitudes because we want to conduct a self-bounding analysis on the queue lengths (recall \Cref{eq:self-bounding informal}).

\subsection{\OLOAlg: Learning for Network Stability}\label{sec:multi-hop stability OLO}
In this section, we will have a small detour to the OLO problem mentioned in \Cref{sec:multi-hop stability sketch}. 
We first rigorously define the OLO problem \citep{zinkevich2003online,mcmahan2010adaptive,duchi2011adaptive} in \Cref{def:OLO}. Then, we present the construction our novel OLO algorithm of \OLOAlg (\Cref{alg:OLO unbounded}).
Finally, we prove that plugging it into the \SSAlg framework in \Cref{alg:stability} indeed ensures good optimization effect of \Cref{eq:multi-hop OLO objective}.
\begin{definition}[Online Linear Optimization]\label{def:OLO}
Consider a $T$-round game. Every round $t\in [T]$, the player selects an action $\bm x_t$ from a convex set $\mathcal X\subseteq \mathbb R^d$. The environment simultaneously decides a loss vector $\bm g_t\in \mathbb R^d$ such that the loss of the player for round $t$ is $\langle \bm x_t,\bm g_t\rangle$. The player will observe the whole vector of $\bm g_t$ (\textit{i.e.}, full-information feedback is available, instead of the more restrictive bandit feedback model). Dynamic regret minimization in OLO considers minimizing
\begin{equation*}
\text{D-Regret}_T^{\text{OLO}}(\mathring{\bm x}_1,\mathring{\bm x}_2,\ldots,\mathring{\bm x}_T)=\sum_{t=1}^T \langle \bm g_t,\bm x_t-\bm x_t^\circ\rangle,\quad \forall \mathring{\bm x}_1,\mathring{\bm x}_2,\ldots,\mathring{\bm x}_T\in \mathcal X.
\end{equation*}
\end{definition}

Before moving on, we recall the two challenges \textit{i)} and \textit{ii)} in \Cref{sec:multi-hop stability sketch}: Due to potentially unbounded queue lengths, \OLOAlg must resist from large and negative losses. Meanwhile, as we want to conduct the self-bounding analysis in \Cref{eq:self-bounding informal}, $\OLO$ shall additionally enjoy a performance guarantee ($\text{D-Regret}_T^{\text{OLO}}$ in \Cref{def:OLO}) depending on the geometric mean of all the loss magnitudes.
Thus, our ideal algorithm for \Cref{def:OLO} must satisfy the following:
\begin{enumerate}
    \item[\textit{i)}] it can resist against occasionally large loss magnitudes, \textit{i.e.}, $\sup_t \lVert \bm g_t\rVert_\infty$ can be large, and 
    \item[\textit{ii)}] it enjoys a performance guarantee depending on all the loss magnitudes, \textit{e.g.}, $\sqrt{\sum_{t=1}^T \lVert \bm g_t\rVert_\infty^2}$.
\end{enumerate}

\begin{algorithm}[t!]
\caption{\OLOAlg: \textbf{Ada}ptive \textbf{P}amameter-\textbf{F}ree \textbf{O}nline \textbf{L}earning}
\label{alg:OLO unbounded}
\begin{algorithmic}[1]
\REQUIRE{Action set $\mathcal X$. For each round $t$, the maximum loss magnitude $G_t$ will be reveled at the beginning, and a loss vector $\bm g_t$ satisfying $\lVert \bm g_t\rVert_\infty \le G_t$ will be given in the end.}
\STATE Set $G\gets 1$. Initialize an instance $\mathcal A$ of the algorithm in \Cref{lem:OLO guarantee} with action set $\mathcal X$.
\FOR{$t=1,2,\ldots$}
\STATE Observe the maximum loss magnitude $G_t>0$ for this round.
\IF{$G_t>G$}
\STATE Set $G=2G_t$. Reset $\mathcal A$ as a new instance of the algorithm in \Cref{lem:OLO guarantee} with action set $\mathcal X$.
\ENDIF
\STATE Output the ouptut of $\mathcal A$. Observe loss vector $\bm g_t$ (such that $\lVert \bm g_t\rVert_\infty\le G_t$). Feed $G^{-1}\bm g_t$ to $\mathcal A$.
\ENDFOR
\end{algorithmic}
\end{algorithm}

To design such an algorithm, we build upon the Parameter-Free Online Learning (\texttt{PFOL}) algorithm by \citet{cutkosky2020parameter}. It ensures condition \textit{ii)} by enjoying $\text{D-Regret}_T^{\text{OLO}}(\mathring{\bm x}_1,\mathring{\bm x}_2,\ldots,\mathring{\bm x}_T)\propto \sqrt{\sum_{t=1}^T \lVert \bm g_t\rVert_\infty^2}$ \citep[Theorem 6]{cutkosky2020parameter}, but fails to bare large loss magnitudes as it requires $\lVert \bm g_t\rVert_\infty\le 1$ for all $t\in [T]$. 

Fortunately, as $C_{n,m}(t)\in [0,M]$, we know $\lVert \bm g_t\rVert_\infty \le M \lVert \bm Q_m(t) - \bm Q_n(t)\rVert_\infty$. Even better, $\bm Q(t)$ can be calculated at the beginning of round $t$ -- before deciding $\bm a(t)$.
Utilizing this knowledge, we are able to design our OLO algorithm which enjoys both property \textit{i)} and \textit{ii)}.
We call this algorithm \OLOAlg (\textbf{Ada}ptive \textbf{P}amameter-\textbf{F}ree \textbf{O}nline \textbf{L}earning), whose pseudo-code is presented in \Cref{alg:OLO unbounded}. \OLOAlg deploys a doubling technique to the \texttt{PFOL} algorithm of \citet{cutkosky2020parameter}, which restarts every time observing a large $G_t$. We can show that this only introduces a logarithmic overhead as the original \texttt{PFOL} algorithm also enjoys \textit{ii)}.

\OLOAlg algorithm enjoys the following dynamic regret guarantee, satisfying both \textit{i)} and \textit{ii)}:
\begin{lemma}[Guarantee of \OLOAlg Algorithm]\label{lem:OLO unbounded guarantee}
Consider the OLO problem in \Cref{def:OLO}. Let the action set $\mathcal X$ has diameter $D=\sup_{\bm x,\bm y\in \mathcal X}\lVert \bm x-\bm y\rVert_1$. Suppose that $\lVert \bm g_t\rVert_\infty\le G_t$, where $G_t$ is some $\mathcal F_{t-1}$-measurable random variable and $(\mathcal F_t)_{t=0}^T$ is the natural filtration, \textit{i.e.}, $\mathcal F_t$ is the $\sigma$-algebra generated by all random observations made during the first $t$ rounds.
Then, \OLOAlg (\Cref{alg:OLO unbounded}) ensures that for any comparator sequence $\mathring{\bm x}_1,\mathring{\bm x}_2,\ldots,\mathring{\bm x}_T\in \mathcal X$, if $\max_{t\in [T]} G_t\ge 1$, then
\begin{equation*}
\text{D-Regret}_T^{\text{OLO}}(\mathring{\bm x}_1,\mathring{\bm x}_2,\ldots,\mathring{\bm x}_T)=\O\left (\sqrt{D \left (D+\sum_{t=1}^{T-1} \lVert \mathring{\bm x}_t-\mathring{\bm x}_{t+1}\rVert_1\right )}\sqrt{\sum_{t=1}^T \lVert \bm g_t\rVert_\infty^2}\log T \log \left (\max_{t=1}^T G_t\right )\right ).
\end{equation*}
\end{lemma}
\begin{remark}\label{remark:dynamic regret}
Note that in general, it is impossible to guarantee $\text{D-Regret}_T^{\text{OLO}}(\mathring{\bm x}_1,\mathring{\bm x}_2,\ldots,\mathring{\bm x}_T)=o_T(T)$ simultaneously for all $\{\mathring{\bm x}_t\in \mathcal X\}_{t\in [T]}$ \citep{zinkevich2003online}.
Therefore, many dynamic regret bounds, including ours,  depend on the notion of {path length} $P_T=\sum_{t=1}^{T-1} \lVert \mathring{\bm x}_t-\mathring{\bm x}_{t+1}\rVert$. Although the path length is linear in $T$ in the worst case, $\text{D-Regret}_T^{\text{OLO}}(\mathring{\bm x}_1,\mathring{\bm x}_2,\ldots,\mathring{\bm x}_T)=o_T(T)$ can still be ensured in cases where $P_T=o_T(T)$.
\end{remark}

The proof of \Cref{lem:OLO unbounded guarantee} will be presented in \Cref{sec:appendix cutkosky}.
It can be seen that \OLOAlg indeed satisfies both \textit{i)} and \textit{ii)}: It allows the loss magnitudes $\lVert \bm g_t\rVert_\infty$ to be large, and also enjoys a magnitude-aware dynamic regret guarantee of $\text{D-Regret}_T^{\text{OLO}}(\mathring{\bm x}_1,\mathring{\bm x}_2,\ldots,\mathring{\bm x}_T)\propto \sqrt{\sum_{t=1}^T \lVert \bm g_t\rVert_\infty^2}$.

Therefore, if we deploy \OLOAlg (\Cref{alg:OLO unbounded}) to decide $\bm a_{n,m}(t)$ on each link $(n,m)\in \mL$ as we do in \Cref{alg:stability}, the RHS of \Cref{eq:multi-hop stability Lyapunov} can consequently be minimized in the sense that it is close to that induced by the reference actions $\{\mathring{\bm a}(t)\}_{t\in [T]}$. Formally, we give the following theorem:
\begin{theorem}[Deciding $\bm a(t)$ via \OLOAlg Algorithm]\label{thm:multi-hop stability decision making}
For each link $(n,m)\in \mL$, as we did in \SSAlg, we execute an instance of \OLOAlg (\Cref{alg:OLO unbounded}) where $\mathcal X=\triangle(\mN)$, $\bm g_t=C_{n,m}(t) (\bm Q_m(t) - \bm Q_n(t))$, and $G_t=M\lVert \bm Q_m(t)-\bm Q_n(t)\rVert_\infty$. We make their outputs $\bm x_t$ as $\bm a_{n,m}(t)$ for every round $t$. Let $\mu_{n,m}^{(k)}(t)$ be the number of actually transmitted jobs from $Q_n^{(k)}(t)$ to $Q_m^{(k)}(t+1)$ induced by $a_{n,m}^{(k)}(t)$.

Consider an arbitrary reference action sequence $\{\mathring{\bm a} (t)\}_{t\in [T]}$ satisfying \Cref{def:multi-hop stability network stability}. Let $\mathring{\mu}_{n,m}^{(k)} (t)=C_{n,m}(t) \mathring{a}_{n,m}^{(k)} (t)\in [0,M]$ (as $C_{n,m}(t)\in [0,M]$ and $\mathring a_{n,m}^{(k)}(t)\in [0,1]$). Then
\begin{align*}
&\quad \E\left [\sum_{t=1}^T \sum_{(n,m)\in \mL} \sum_{k\in \mN} (\mu_{n,m}^{(k)}(t)-\mathring{\mu}_{n,m}^{(k)}(t))\left (Q_m^{(k)}(t)-Q_n^{(k)}(t)\right )\right ]\\
&=\O\left (M\sqrt{1+P_T^a} \E\left [\sqrt{\sum_{t=1}^T \lVert \bm Q(t)\rVert_2^2} \log T \log \left (\max_{t=1}^T \max_{(n,m)\in \mL} M \lVert \bm Q_m(t) - \bm Q_n(t)\rVert_\infty\right )\right ]\right )\text{,}
\end{align*}
where $P_T^a\triangleq \sum_{t=1}^{T-1} \sum_{(n,m)\in \mL} \lVert \mathring{\bm a}_{n,m}(t) - \mathring{\bm a}_{n,m}(t+1)\rVert_1$ is the path length of $\{\mathring{\bm a}(t)\}_{t=1}^T$.
\end{theorem}

The proof is almost directly applying \Cref{lem:ABGD guarantee}, so we postpone it to \Cref{sec:multi-hop Lyapunov and Online Learning appendix}.
Thanks to property \textit{ii)}, the RHS of \Cref{thm:multi-hop stability decision making} depends on the $\ell_2$-norm of the queue lengths $\sqrt{\sum_{t=1}^T \lVert \bm Q(t)\rVert_2^2}$. As we will see shortly, this is pivotal to the self-bounding argument sketched in \Cref{eq:self-bounding informal}.

\subsection{Main Theorem for Multi-Hop Network Stability}\label{sec:multi-hop stability main theorem}
As sketched in \Cref{sec:multi-hop stability sketch}, putting previous conclusions together and use a so-called {self-bounding} property, the following guarantee for multi-hop network stability can be derived:
\begin{theorem}[Main Theorem for Multi-Hop Network Stability]\label{thm:multi-hop stability main theorem}
Suppose that $\{\mathring{\bm a}_{n,m}(t)\in \triangle(\mN)\}_{(n,m)\in \mL,t\in [T]}$ satisfies \Cref{def:multi-hop stability network stability} and its path length satisfies \footnote{This assumption on path lengths comes from \citep[Assumption 2]{huang2023queue}. As discussed in \Cref{remark:dynamic regret}, such conditions are necessary.}
\begin{equation}\label{eq:network stability path length}
P_t^a\triangleq \sum_{s=1}^{t-1} \sum_{(n,m)\in \mL} \lVert \mathring{\bm a}_{n,m}(s) - \mathring{\bm a}_{n,m}(s+1)\rVert_1\le C^a t^{1/2 - \delta_a},\quad \forall t=1,2,\ldots,T,
\end{equation}
where $C^a$ and $\delta_a$ are assumed to be known constants but the precise $P_t^a$ or $\{\mathring{\bm a}_{n,m}(t)\in \triangle(\mN)\}_{(n,m)\in \mL,t\in [T]}$ both remain unknown.
Then, if we execute the \SSAlg framework in \Cref{alg:stability} with \OLOAlg defined in \Cref{alg:OLO unbounded}, the following performance guarantee is enjoyed:
\begin{equation*}
\frac 1T\E\left [\sum_{t=1}^T \lVert \bm Q(t)\rVert_1\right ]=\O\left (\frac{(N^2(2NM+R)^2 + \epsilon_W N^2 (2NM+R))C_W + (N^4 M^2 + N^2 R^2)}{\epsilon_W}\right )+o_T(1).
\end{equation*}
That is, when $T\gg 0$, we have $\frac{1}{T} \E\left [\sum_{t=1}^T \lVert \bm Q(t)\rVert_1\right ]=\O_T(1)$, \textit{i.e.}, \Cref{eq:average queue} holds and the system is stable.
\end{theorem}
\begin{remark}\label{remark:stability main theorem}
Any reference policy $\{\mathring{\bm a}(t)\}_{t\in [T]}$ satisfying \Cref{eq:network stability path length} is called ``mildly varying''. As mentioned in \Cref{remark:dynamic regret}, it is impossible to achieve non-trivial performance without restricting the reference sequence.
We also compare our result with two previous results \citep{huang2023queue,yang2023learning} which also ensured network stability in adversarial networks under bandit feedback: Under a path length assumption similar to but looser than \Cref{eq:network stability path length}, \citet{huang2023queue} stabilized single-server networks. 
By assuming the environment (instead of the reference policy which we do) is mildly varying, \citet{yang2023learning} stabilized single-hop networks.
Thus, \Cref{thm:multi-hop stability main theorem} is the first guarantee applicable to adversarial multi-hop networks under bandit feedback.
\end{remark}

A formal proof resides in \Cref{sec:multi-hop stability appendix}. We only highlight the self-bounding step here:
\begin{proof}[Proof Sketch of \Cref{thm:multi-hop stability main theorem}]
The first step of the proof is comparing the guarantee from \Cref{def:multi-hop stability network stability} and that from Lyapunov drift analysis.
Specifically, recall that \Cref{lem:multi-hop stability network stability} upper bounds the total queue length using $-\E[\sum_{t=1}^{T} \sum_{(n,m)\in \mL} \langle \bm C_{n,m}(t) (\bm Q_m(t) - \bm Q_n(t)),\mathring{\bm a}_{n,m}(t)\rangle]$ (together with some other terms, which are constants after taking expectations) while \Cref{lem:multi-hop stability Lyapunov} reveals the non-positivity of $-\E[\sum_{(n,m)\in \mL} \langle \bm C_{n,m}(t) (\bm Q_m(t) - \bm Q_n(t)),{\bm a}_{n,m}(t)\rangle]$ (again, many constants omitted).
Furthermore, via the guarantee of \OLOAlg (\Cref{alg:OLO unbounded}) in \Cref{thm:multi-hop stability decision making}, these two terms are actually pretty close -- they only differ by $\Otil_T\left (\sqrt{1+C^a T^{1/2-\delta_a}} \E\left [\sqrt{\sum_{t=1}^T \lVert \bm Q(t)\rVert_2^2}\right ]\right )$, where we used the assumption that $P_T^a\le C^a T^{1/2-\delta_a}$.
By a property that ensures $\sum_{t=1}^T x_t^2 \le 4(\sum_{t=1}^T x_t)^{1/5}$ when $\lvert x_t-x_{t+1}\rvert\le 1$ (\Cref{lem:2 upper bound}), this $\E\left [\sqrt{\sum_{t=1}^T \lVert \bm Q(t)\rVert_2^2}\right ]$ can be controlled by $\E\left [\sum_{t=1}^T \lVert \bm Q(t)\rVert_1\right ]^{3/4}$.

Finishing these steps, which are detailed in \Cref{lem:multi-hop stability tedious calc}, we are able to conclude
\begin{equation*}
\E\left [\sum_{t=1}^T \lVert \bm Q(t)\rVert_1\right ]\le f(T) + g(T) \E\left [\sum_{t=1}^T \lVert \bm Q(t)\rVert_1\right ]^{3/4} \log \E\left [\sum_{t=1}^T \lVert \bm Q(t)\rVert_1\right ],
\end{equation*}
where $f(T) = \O_T(T)$ and $g(T) = \O_T(T^{1/4-\delta_a/2})$ are some abstract functions to simplify notations.

Therefore, this inequality is in a self-bounding form that $y\le f+y^{3/4} g \log y$ where $y$ is the total queue lengths in $T$ rounds.
As we informally stated in \Cref{eq:self-bounding informal}, this gives our system stability guarantee.
Indeed, in \Cref{lem:3/4 with log self-bounding}, we show that $y\le f+y^{3/4} g \log y$ implies $y=\O(f)+\Otil(g^4)$. Therefore, $y=\E[\sum_{t=1}^T \lVert \bm Q(t)\rVert_1]=\O_T(T)+\Otil_T(T^{1-2\delta_a})$ and thus $\frac 1T \E[\sum_{t=1}^T \lVert \bm Q(t)\rVert_1]=\O_T(1)$.
\end{proof}


\begin{algorithm}[t!]
\caption{\UMAlg: \textbf{U}tility \textbf{M}aximization via Online Linear \textbf{O}ptimization and Bandit Convex \textbf{O}ptimization}
\label{alg:utility}
\begin{algorithmic}[1]
\REQUIRE{Number of rounds $T$, set of servers $\mN$ and links $\mL$, maximum capacity $M$, feasible arrival rates $\Lambda$. Parameter $V$. An online linear optimization algorithm $\OLO$ (\Cref{alg:OLO unbounded}) and a bandit convex optimization algorithm $\BCO$ (\Cref{alg:ABGD}).}
\STATE For each link $(n,m)\in \mL$, initialize an instance of $\OLO$ with action set $\triangle(\mN)$ as $\OLO_{n,m}$.
\STATE {\color{blue} Initialize an instance of $\BCO$ with action set $\Lambda$.}
\FOR{$t=1,2,\ldots,T$}
\STATE For each link $(n,m)\in \mL$, pass the maximum loss magnitude for this round $M\lVert \bm Q_m(t)-\bm Q_n(t)\rVert_\infty$ to $\OLO_{n,m}$. Pick link allocation $\bm a_{n,m}(t)\in \triangle(\mN)$ as the output of $\OLO_{n,m}$.
\STATE {\color{blue}Call $\BCO$ with learning rates defined in \Cref{eq:true learning rate in BCO}. Pick arrival rates $\bm \lambda(t)\in \Lambda$ as its output.
\begin{align}
\eta_t&=\left (C^\lambda T^{1/2 - \delta_\lambda}\middle / \substack{\left (C^\lambda T^{1/2 - \delta_\lambda}\right )^{7/3} \left (4r^{-3} d^2 \right )^{28/9} \left (M+R\right )^{4/3}+\\
C^\lambda T^{1/2-\delta_\lambda} (r^{-3} d^2 VG^2 / L)^{4/3}+\\
\sum_{s=1}^t \left ((\lVert \bm q_s\rVert_\infty + VG)^2 (\lVert \bm q_s\rVert_2 + VL)^2\right )^{1/3}}\right )^{3/4}, \nonumber\\
\delta_t &= \left (\eta_t d^2 \frac{(\lVert \bm Q(t)\rVert_\infty + VG)^2}{(\lVert \bm Q(t)\rVert_2 + VL)}\right )^{1/3},\quad \alpha_t = \frac{\delta_t}{r}.\label{eq:true learning rate in BCO}
\end{align}}
\STATE Observe capacities $\{C_{n,m}(t)\}_{(n,m)\in \mL}$ and actual data transmissions $\{\mu_{n,m}^{(k)}(t)\}_{(n,m)\in \mL,k\in \mN}$.
\STATE Calculate queue lengths $\bm Q(t+1)$ from $\bm Q(t)$ according to \Cref{eq:multi-hop stability dynamics}.
\STATE For each link $(n,m)\in \mL$, pass the loss vector $C_{n,m}(t)(\bm Q_m(t)-\bm Q_n(t))$ to $\OLO_{n,m}$.
\STATE {\color{blue} Observe the collected utility $g_t(\bm \lambda(t))$. Pass the loss $\langle \bm Q(t),\bm \lambda (t)\rangle - V g_t(\bm \lambda(t))$ to $\BCO$.}
\ENDFOR
\end{algorithmic}
\end{algorithm}

\section{Utility Maximization in Adversarial Multi-Hop Networks}\label{sec:multi-hop utility}
We now turn to the utility maximization task. 
In this task, in addition to the capacity allocations, the arrival rates are also decided by the scheduler with an objective of maximizing the unknown and time-varying utility function.
The scheduler's objective is to maximize the average utility it gains (\Cref{eq:average reward}), while ensuring the average number of jobs in the network remains small (\Cref{eq:average queue}). 

This section is organized similar to \Cref{sec:multi-hop stability}: We first explain the motivation of our algorithmic framework \UMAlg in \Cref{alg:utility} and then present the assumptions together with analysis. 

\subsection{Motivation of Our Algorithmic Framework}\label{sec:multi-hop utility sketch}
In \Cref{alg:utility}, we give the general algorithmic framework of \textbf{U}tility \textbf{M}aximization via Online Linear \textbf{O}ptimization and Bandit Convex \textbf{O}ptimization (\UMAlg) which achieves utility optimization via the plug-in of two optimization sub-rountines $\OLO$ and $\BCO$.
The differences between it and our system stability algorithm (\SSAlg; \Cref{alg:stability}) is marked in blue.
As we already motivated in \Cref{sec:multi-hop stability sketch} that an OLO algorithm $\OLO$ can help stabilize the system (recall \Cref{eq:multi-hop stability OLO objective informal}), this section focuses on motivating the other sub-rountine $\BCO$ by going through the design of \UMAlg.

To handle the utility function, instead of the Lyapunov analysis in the previous section, \UMAlg is based on the Lyapunov drift-plus-penalty analysis \citep[Theorem 4.2]{neely2010stochastic}.
In \Cref{lem:multi-hop utility Lyapunov}, we derive
\begin{align}
&\E\left [\sum_{t=1}^T \sum_{(n,m)\in \mL} \langle C_{n,m}(t) (\bm Q_m(t) - \bm Q_n(t)), \bm a_{n,m}(t)\rangle\right ] + \nonumber\\
&\E\left [\sum_{t=1}^T \sum_{n\in \mN} \langle \bm Q_n(t),\bm \lambda_n(t)\rangle\right ] - V \E\left [\sum_{t=1}^T \bigg (g_t(\bm \lambda(t) - g_t(\mathring{\bm \lambda}(t)))\bigg )\right ]\gtrsim 0,\label{eq:multi-hop utility Lyapunov informal}
\end{align}
where $V$ is a constant that we can arbitrarily pick for analytical purposes.
Intuitively, this $V$ stands for a trade-off between the stability part $\langle \bm Q_n(t),\bm \lambda_n(t)\rangle$ and the utility part $g_t(\bm \lambda(t))$.

Again motivated by \citet{huang2023queue}, our goal is to minimize \Cref{eq:multi-hop utility Lyapunov informal}.
The first term in \Cref{eq:multi-hop utility Lyapunov informal} is exactly the OLO optimization objective from the previous section (recall \Cref{eq:multi-hop stability OLO objective informal}), which can be minimized by the \OLOAlg algorithm given in \Cref{alg:OLO unbounded}.
For the second and third term, we would like to
\begin{equation}
\text{Minimize }\E\left [\sum_{t=1}^T \bigg (\langle \bm Q(t),\bm \lambda(t)\rangle-Vg_t(\bm \lambda(t))\bigg )\right ].\label{eq:multi-hop utility BCO objective informal}
\end{equation}

We now also tackle \Cref{eq:multi-hop utility BCO objective informal} using online learning techniques:
As $g_t$ is defined over all possible $\bm \lambda(t)$'s and $\bm \lambda(t)$ can be arbitrarily chosen from the feasible action set $\Lambda$, we regard ``deciding $\bm \lambda(t)$'' as an learning problem with action set $\Lambda$ (instead of making decisions on each link or server separately in \Cref{sec:multi-hop stability OLO}) where the loss of picking $\bm \lambda$ in round $t$ is $
\ell_t(\bm \lambda)\triangleq \langle \bm Q(t),\bm \lambda\rangle - V g_t(\bm \lambda)$.
This loss function is convex \textit{w.r.t.} $\bm \lambda$ as $\langle \cdot,\cdot \rangle$ is linear and $g_t$ is concave.
However, since only bandit feedback on $g_t$ is available, we can only calculate $\ell_t(\bm \lambda)$ but not the whole $\ell_t$. Thus, this problem is not an OLO problem as \Cref{def:OLO} requires full information feedback. Instead, it belongs to the category of \textit{Bandit Convex Optimization} (BCO) \citep{flaxman2005online}, which we will define in \Cref{def:BCO}.

Similar to \Cref{sec:multi-hop stability sketch}, we now discuss what properties the $\BCO$ sub-routine should enjoy. We have the following challenges that is unique due to the network optimization context:
\begin{enumerate}
\item[\textit{i)}] Again, the queue lengths can potentially go unbounded, which means the loss $\ell_t(\bm \lambda)$ can have large magnitudes. However, different from the OLO problem we met in \Cref{eq:multi-hop stability OLO objective informal}, in BCO problems we face general convex functions and thus Lipschitzness (\textit{i.e.}, the maximum gradient magnitude) also plays a role as it characterizes how fast $\ell_t$ changes with $\bm \lambda$. Therefore, our $\BCO$ shall not only bare large loss magnitudes but also resist from huge Lipshictzness. As we will see in \Cref{eq:true learning rate in BCO}, adapting to both magnitudes and Lipschitzness is particularly difficult.
\item[\textit{ii)}] The second challenge is again due to self-bounding analysis: Specifically, we want to conduct self-bounding analyses on the queue lengths and also the utility gap (both similar to \Cref{eq:self-bounding informal}), we also want $\BCO$ to be adaptive to the loss functions' magnitudes and Lipschitzness.
\end{enumerate}

In \Cref{sec:multi-hop utility BCO}, we introduce the details of our \BCOAlg algorithm that ensures both \textit{i)} and \textit{ii)}.
Equipped with this algorithm, we can optimize \Cref{eq:multi-hop utility BCO objective informal} by deploying it over the action set $\Lambda$.
Moreover, as deploying \OLOAlg (\Cref{alg:OLO unbounded}) on each link $(n,m)\in \mL$ minimizes $\E[\sum_{t=1}^T \sum_{(n,m)\in \mL} \langle \bm Q_m(t)-\bm Q_n(t),\bm \mu_{n,m}(t)\rangle]$, these two algorithms can together minimize the Lyapunov drift-plus-penalty in \Cref{eq:multi-hop utility Lyapunov informal}. Such a combination gives the \UMAlg framework in \Cref{alg:utility}.

In the remaining of this section, we present the analysis of \UMAlg: In \Cref{sec:multi-hop utility assumption}, we introduce the reference sequence assumption. In \Cref{sec:multi-hop utility Lyapunov}, we present the Lyapunov drift-plus-penalty analysis. In \Cref{sec:multi-hop utility BCO}, we rigorously define the BCO problem and present our \BCOAlg algorithm (\Cref{alg:ABGD}). Finally, by combining the \OLOAlg guarantee from \Cref{sec:multi-hop stability OLO} and the \BCOAlg guarantee from \Cref{sec:multi-hop utility BCO}, we yield the utility maximization guarantee in \Cref{sec:multi-hop utility main theorem}.

\subsection{Reference Policy Assumption}\label{sec:multi-hop utility assumption}
The assumption we need in the multi-hop utility maximization task is similar to the one in multi-hop network stability (\Cref{def:multi-hop stability network stability}), with one important difference that our arrival rates are no longer fixed but decided by the scheduler. Hence, instead of assuming a sequence of $\{\mathring{\bm a}(t)\}_{t\in [T]}$ stabilizing the system with the obliviously adversarial arrival rates $\{\bm \lambda(t)\}_{t\in [T]}$, we assume the existence of reference sequence $\{(\mathring{\bm a}(t),\mathring{\bm \lambda}(t))\}_{t\in [T]}$ such that $\{\mathring{\bm a}(t)\}_{t\in [T]}$ stabilizes system with reference arrival rates $\{\mathring{\bm \lambda}(t)\}_{t\in [T]}$. Formally, we make the following assumption.
\begin{assumption}[Multi-Hop Piecewise Stability for Utility Maximization]\label{def:multi-hop utility network stability}
There exists a {reference action sequence} $\{(\mathring{\bm a}(t),\mathring{\bm \lambda}(t))\}_{t\in [T]}$ (where $\mathring{\bm a}(t)=\{\mathring{\bm a}_{n,m}(t)\in \triangle(\mN)\}_{(n,m)\in \mL}$ and $\mathring{\bm \lambda}(t)=\{\mathring{\bm \lambda}_n(t)\in \Lambda_n\}_{n\in \mN}$, in analogue to the scheduler's action sequence) such that there are constants $C_W\ge 0$, $\epsilon_W\ge 0$, and a partition $W_1,W_2,\ldots,W_J$ of $[T]$, such that $\sum_{j=1}^J (\lvert W_j\rvert-1)^2\le C_W T$ and
\begin{align*}
&\frac{1}{\lvert W_j\rvert} \sum_{t\in W_j} \sum_{(n,m)\in \mL} C_{n,m}(t) \mathring a_{n,m}^{(k)}(t) \ge \epsilon_W + \frac{1}{\lvert W_j\rvert} \sum_{t\in W_j} \left (\mathring \lambda_n^{(k)}(t) + \sum_{(o,n)\in \mL} C_{o,n}(t) \mathring a_{o,n}^{(k)}(t)\right ),\\
&\quad \forall j\in [J],n\in \mN,k\in \mN.
\end{align*}
\end{assumption}

Imitating \Cref{lem:multi-hop stability network stability}, one can derive the following, whose proof is in \Cref{sec:multi-hop utility network stability appendix}.
\begin{lemma}[Ability of $\{(\mathring{\bm a}(t),\mathring{\bm \lambda}(t))\}_{t\in [T]}$ in Stabilizing the Network]\label{lem:multi-hop utility network stability}
If $\{(\mathring{\bm a}(t),\mathring{\bm \lambda}(t))\}_{t\in [T]}$ satisfies \Cref{def:multi-hop utility network stability}, then for any scheduler-generated queue lengths $\{\bm Q(t)\}_{t\in [T]}$,
\begin{align}
&\quad \epsilon_W \E\left [\sum_{t=1}^T \sum_{n\in \mN} \sum_{k\in \mN} Q_n^{(k)}(t)\right ] - (N^2 (2NM+R)^2 + \epsilon_W N^2 (2NM+R)) C_W T \nonumber\\
&\le-\E\left [\sum_{t=1}^T \sum_{(n,m)\in \mL}\sum_{k\in \mN} \mathring \mu_{n,m}^{(k)}(t) (Q_m^{(k)}(t) - Q_n^{(k)}(t))\right ] - \E\left [\sum_{t=1}^T \sum_{n\in \mN} \sum_{k\in \mN} Q_n^{(k)}(t) \mathring \lambda_n^{(k)}(t)\right ].\label{eq:multi-hop utility network stability}
\end{align}
\end{lemma}

Still, we remark that our scheduler cannot access the reference action sequence $\{(\mathring{\bm a}(t),\mathring{\bm \lambda}(t))\}_{t\in [T]}$. However, similar to the \SSAlg algorithm, our \UMAlg algorithm can also learn to stabilize the system.
Even more, it also learns to outperform the utility maximization performance of \textit{any} mildly varying reference policy.
Specifically, \textit{i)} our action sequence $\{(\bm a(t),\bm \lambda(t))\}_{t\in [T]}$ also stabilizes the system, \textit{i.e.}, $\frac 1T \E[\sum_{t=1}^T \lVert \bm Q(t)\rVert_1\rvert]=\O_T(1)$; and, \textit{ii)} its utility matches any mildly varying reference policy $\{(\mathring{\bm a}(t),\mathring{\bm \lambda}(t))\}_{t\in [T]}$ asymptotically --- that is, $\frac 1T \E[\sum_{t=1}^T g_t(\bm \lambda(t))]\xrightarrow{\text{polynomially}} \frac 1T \sum_{t=1}^T g_t(\mathring{\bm \lambda}(t))$.\footnote{According to the oblivious adversary assumption, $g_t$ is pre-determined. Thus, the $g_t$'s on the LHS and RHS are the same, so there is no need to take conditional expectation \textit{w.r.t.} previous actions $\{(\bm a(t),\bm \lambda(t))\}_{t\in [T]}$ or $\{(\mathring{\bm a}(t),\mathring{\bm \lambda}(t))\}_{t\in [T]}$.}

\subsection{Lyapunov Drift-plus-Penalty Analysis}\label{sec:multi-hop utility Lyapunov}
In the Lyapunov drift analysis (\Cref{sec:multi-hop stability Lyapunov}), we consider the drift function $\Delta(\bm Q(t))\triangleq \E[L_{t+1}-L_t \mid \bm Q(t)]$ where $L_t\triangleq \frac 12 \lVert \bm Q(t)\rVert_2^2$ is the Lyapunov function.
In the Lyapunov drift-plus-penalty (DPP) analysis \citep[Theorem 4.2]{neely2010stochastic}, we consider the DPP function $\Delta(\bm Q(t))-V\E[g_t(\bm \lambda(t)) \mid \bm Q(t)]$, where $V$ is arbitrarily determined for our purpose.
As we will see in \Cref{thm:multi-hop utility main theorem}, when $V$ is chosen to be no larger than a polynomial of $T$, our utility is at least that of any mildly varying reference policy minus $\O(V^{-1})$, thus implying a polynomially decaying gap between these two utilities.

Similar to the calculations in \Cref{lem:multi-hop stability network stability}, one can derive the following inequality (see \Cref{lem:multi-hop utility Lyapunov}):
\begin{align}
&\quad -\E\left [\sum_{t=1}^T \sum_{(n,m)\in \mL} \sum_{k\in \mN} C_{n,m}(t) (Q_m^{(k)}(t)-Q_n^{(k)}(t))a_{n,m}^{(k)}(t)\right ] \nonumber\\
&\quad -\E\left [\sum_{t=1}^T \sum_{n\in \mN} \sum_{k\in \mN} Q_n^{(k)}(t) \lambda_n^{(k)}(t)\right ]+V \E\left [\sum_{t=1}^T (g_t(\bm \lambda(t))-g_t(\mathring{\bm \lambda}(t)))\right ] \nonumber\\
&\le \frac 12 N^2 ((NM)^2+2(NM)^2+2R^2) T+V  \E\left [\sum_{t=1}^T (g_t(\bm \lambda(t))-g_t(\mathring{\bm \lambda}(t)))\right ].\label{eq:multi-hop utility Lyapunov}
\end{align}

Similar to the previous section, we want to make the RHS of \Cref{eq:multi-hop utility Lyapunov} close to that of \Cref{eq:multi-hop utility network stability}.
Specifically, we decompose these two RHS's into two parts and show the following inequalities:
\begin{align}
&\scalemath{0.95}{-\E\left [\sum_{t=1}^T \sum_{(n,m)\in \mL}\langle \mathring {\bm \mu}_{n,m}(t), \bm Q_m(t) - \bm Q_n(t)\rangle\right ]\lesssim -\E\left [\sum_{t=1}^T \sum_{(n,m)\in \mL} \langle \bm \mu_{n,m}(t), \bm Q_m(t)-\bm Q_n(t)\rangle\right ],~ \text{and}} \nonumber\\
&\scalemath{0.95}{- \E\left [\sum_{t=1}^T \sum_{n\in \mN} \langle \bm Q_n(t), \mathring {\bm \lambda}_n(t)\rangle\right ]\lesssim -\E\left [\sum_{t=1}^T \sum_{n\in \mN} \langle \bm Q_n(t), \bm \lambda_n(t)\rangle\right ]+V \E\left [\sum_{t=1}^T \bigg (g_t(\bm \lambda(t))-g_t(\mathring{\bm \lambda}(t))\bigg )\right ].} \label{eq:multi-hop utility BCO objective}
\end{align}

The first inequality is exactly our objective in \Cref{eq:multi-hop stability OLO objective informal}, which can be ensured via the \OLOAlg algorithm in \Cref{alg:OLO unbounded} -- recall its performance guarantee in \Cref{thm:multi-hop stability decision making}.
On the other hand, the second inequality is new in the utility maximization task.
While some algorithmic ingredients can be borrowed from the {Bandit Convex Optimization} (BCO) problem \citep{flaxman2005online}, new efforts need be made due to the network optimization context: Our BCO algorithm shall accept large loss magnitudes and Lipschitzness (due to potentially unbounded queue lengths), and its performance must be adaptive to the loss functions' magnitudes and Lipschitzness as well.

\subsection{\BCOAlg: Learning for Utility Maximization}\label{sec:multi-hop utility BCO}
As mentioned in the sketch (\Cref{sec:multi-hop utility sketch}), \Cref{eq:multi-hop utility BCO objective} is equivalent to minimizing the time-varying loss function $\ell_t(\bm \lambda)\triangleq \langle Q(t),\bm \lambda\rangle - V g_t(\bm \lambda)$ under bandit feedback over the action set $\Lambda$.
This problem is different from the OLO problem introduced in \Cref{def:OLO} as we do not have full-information feedback:
Indeed, we assume $g_t(\bm \lambda(t))$ instead of the whole $g_t$ will be revealed to the scheduler, hence only $\ell_t(\bm \lambda(t))$, the actual loss associated with our action, can be accurately calculated.
We provide a formal definition of the {Bandit Convex Optimization} (BCO) problem \citep{flaxman2005online,chen2018bandit} below. 
\begin{definition}[Bandit Convex Optimization]\label{def:BCO}
Consider a $T$-round game. In round $t=1,2,\ldots,T$, the player picks an action $\bm x_t$ from a convex action set $\mathcal X\subseteq \mathbb R^d$, and the environment simultaneously picks an arbitrary convex loss $\ell_t\colon \mathcal X\to \mathbb R$. The player observes and suffers loss $\ell_t(\bm x_t)$. Dynamic regret minimization in BCO considers minimizing
\begin{equation*}
\text{D-Regret}_T^{\text{BCO}}(\bm u_1,\bm u_2,\ldots,\bm u_T)=\E\left [\sum_{t=1}^T (\ell_t(\bm x_t)-\ell_t(\bm u_t))\right ],\quad \forall \bm u_1,\bm u_2,\ldots,\bm u_T\in \mathcal X.
\end{equation*}
\end{definition}

Again, we recall the two challenges that our BCO algorithm should overcome:
\begin{enumerate}
\item[\textit{i)}] It must handle $\ell_t$'s with large magnitude $\sup_{\bm \lambda\in \Lambda} \lvert \ell_t(\bm \lambda)\rvert$ or Lipschitzness $\sup_{\bm \lambda\in \Lambda} \lVert \nabla \ell_t(\bm \lambda)\rVert_2$.
\item[\textit{ii)}] Its performance should depend on magnitudes and Lipschitzness of all loss functions.
\end{enumerate}

Our algorithm is based on the Bandit Gradient Descent (\texttt{BGD}) algorithm \citep[Algorithm 1]{zhao2021bandit}, which does not satisfy \textit{i)} or \textit{ii)} as it requires losses to be uniformly bounded by some $C$ and Lipschitzness to be always bounded by some $L$.
In our case where $\ell_t(\bm \lambda)=\langle \bm Q(t),\bm \lambda\rangle-Vg_t(\bm \lambda)$, the loss magnitude $\lVert \bm Q(t)\rVert_\infty + VG$ and the Lipschitzness $\lVert \bm Q(t)\rVert_2+VL$ are both large when $\lVert \bm Q(t)\rVert$ is large.

Nevertheless, based on the \texttt{BGD} algorithm, we designed a BCO algorithm called \textbf{Ada}ptive \textbf{BGD} (\BCOAlg; \Cref{alg:ABGD}) which satisfies both \textit{i)} and \textit{ii)}.
Specifically, to ensure \textit{i)}, we utilize the fact that $\bm Q(t)$ is known before deciding $\bm \lambda_t$, which means the magnitude $C_t$ and the Lipschitzness $L_t$ of loss function $\ell_t$ can be calculated before decision. To enjoy \textit{ii)}, instead of the doubling technique in \OLOAlg (\Cref{alg:OLO unbounded}), we now design an \textit{adaptive learning rate} scheduling mechanism which involves a sequence of time-varying learning rates, namely $\eta_1>\eta_2>\cdots>\eta_T$, instead of using a single $\eta$ throughout execution.
Formally, \BCOAlg has the following dynamic regret guarantee:
\begin{algorithm}[t!]
\caption{\BCOAlg: \textbf{A}daptive \textbf{B}andit \textbf{G}radient \textbf{D}escent}
\label{alg:ABGD}
\begin{algorithmic}[1]
\REQUIRE{Action set $\mathcal X$ bounded by $[r,R]$ (\textit{i.e.}, $r\mathbb B\subseteq \mathcal X\subseteq R\mathbb B$), hyper-parameters $\eta_1>\eta_2>\cdots>\eta_T$, $\delta_1,\delta_2,\ldots,\delta_T$, and $\alpha_t\triangleq \delta_t / r,\forall t\in [T]$.}
\STATE Initialize $\bm y_1=\bm 0$ (an internal variable of the algorithm).
\FOR{$t=1,2,\ldots,T$}
\STATE Calculate this round's action $\bm x_t\in \mathcal X$, observe loss $\ell_t(x_t)$, and update internal variable $\bm y_{t+1}$:
\begin{equation}\label{eq:ABGD}
\bm x_t=\bm y_t+\delta \bm s_t,\quad \bm y_{t+1}=\text{Proj}_{(1-\alpha_t)\mathcal X}\left [\bm y_t-\eta_t \frac{d}{\delta_t} \ell_t(\bm x_t) \bm s_t\right ],
\end{equation}
where $\bm s_t\in \mathbb R^d$ is a uniformly sampled unit vector used to estimate gradients \citep{flaxman2005online}.
\ENDFOR
\end{algorithmic}
\end{algorithm}

\begin{lemma}[Guarantee of \BCOAlg Algorithm]\label{lem:ABGD guarantee}
Suppose that $r\mathbb B\subseteq \mathcal X\subseteq R\mathbb B$, the $t$-th loss $\ell_t$ is bounded by $C_t$ and is $L_t$-Lipschitz.
Suppose that $\eta_t$ and $\delta_t$ are both $\mathcal F_{t-1}$-measurable (where $(\mathcal F_t)_{t=0}^T$ is the natural filtration), $\eta_1>\eta_2>\cdots>\eta_T$, and $\alpha_t\triangleq \delta_t / r<1$ \textit{a.s.} for all $t\in [T]$. Then for any fixed $\bm u_1,\bm u_2,\ldots,\bm u_T\in \mathcal X$, the \BCOAlg algorithm in \Cref{alg:ABGD} enjoys the following guarantee:
\begin{align*}
&\quad \text{D-Regret}_T^{\text{BCO}}(\bm u_1,\bm u_2,\ldots,\bm u_T)=\E\left [\sum_{t=1}^T (\ell_t(\bm x_t)-\ell_t(\bm u_t))\right ]\\
&\le \E\left [\frac{7R^2}{4\eta_T}+\frac{P_T R}{\eta_T}+\sum_{t=1}^T \left (\frac{\eta_t}{2} \frac{d^2}{\delta_t^2} C_t^2 + 3L_t \delta_t + L_t \alpha_t R\right )\right ],
\end{align*}
where $P_T=\sum_{t=1}^{T-1}\lVert \bm u_t-\bm u_{t+1}\rVert$ is the path length of the comparator sequence $\{\bm u_t\}_{t\in [T]}$.
\end{lemma}

Compared to the algorithm itself in \Cref{alg:ABGD} and the guarantee in \Cref{lem:ABGD guarantee}, our main innovation on the BCO side lies in the novel learning rate scheduling mechanism in \Cref{eq:true learning rate in BCO}.
Therefore, we do not prove \Cref{lem:ABGD guarantee} at this moment and postpone it to \Cref{sec:ABGD guarantee}.
Instead, we prove:

\begin{theorem}[Deciding $\bm \lambda(t)$ via \BCOAlg Algorithm]\label{thm:multi-hop utility decision making}
For the reference arrival rates $\{\mathring{\bm \lambda}(t)\}_{t\in [T]}$ defined in \Cref{def:multi-hop utility network stability}, suppose that its path length ensures
\begin{equation*}
P_t^\lambda\triangleq \sum_{t=1}^{T-1} \lVert \mathring{\bm \lambda}({t+1}) - \mathring{\bm \lambda}(t)) \rVert_1\le C^\lambda t^{1/2 - \delta_\lambda},\quad \forall t=1,2,\ldots,T,
\end{equation*}
where, similar to \Cref{thm:multi-hop stability main theorem}, $C^\lambda$ and $\delta_\lambda$ are assumed to be known constants but the precise $P_t^\lambda$ or $\{\mathring{\bm \lambda}(t)\}_{t\in [T]}$ both remain unknown.
Suppose that the action set $\Lambda$ is bounded by $[r,R]$ (\textit{i.e.}, $r\mathbb B\subseteq \Lambda\subseteq R\mathbb B$).
If we execute \BCOAlg (\Cref{alg:ABGD}) over $\Lambda$ with loss functions $\ell_t(\bm \lambda)=\langle \bm Q(t),\bm \lambda\rangle-V g_t(\bm \lambda)$ and parameters $\eta_t,\delta_t,\alpha_t$ defined in \Cref{eq:true learning rate in BCO} (restated below as \Cref{eq:true learning rate in BCO restatement} to ease reading):
\begin{align}
\eta_t&=\left (C^\lambda T^{1/2 - \delta_\lambda}\middle / \substack{\left (C^\lambda T^{1/2 - \delta_\lambda}\right )^{7/3} \left (4r^{-3} d^2 \right )^{28/9} \left (M+R\right )^{4/3}+\\
C^\lambda T^{1/2-\delta_\lambda} (r^{-3} d^2 VG^2 / L)^{4/3}+\\
\sum_{s=1}^t \left ((\lVert \bm q_s\rVert_\infty + VG)^2 (\lVert \bm q_s\rVert_2 + VL)^2\right )^{1/3}}\right )^{3/4}, \nonumber\\
\delta_t &= \left (\eta_t d^2 \frac{(\lVert \bm Q(t)\rVert_\infty + VG)^2}{(\lVert \bm Q(t)\rVert_2 + VL)}\right )^{1/3},\quad \alpha_t = \frac{\delta_t}{r},\label{eq:true learning rate in BCO restatement}
\end{align}
then the outputs $\bm \lambda(1),\bm \lambda(2),\ldots,\bm \lambda(T) \in \Lambda$ of \BCOAlg ensure
\begin{align*}
&\scalemath{0.95}{\quad \E\left [\sum_{t=1}^T \left ((\langle \bm Q(t),\bm \lambda(t)\rangle - Vg_t(\bm \lambda(t))) - (\langle \bm Q(t),\mathring{\bm \lambda}(t)\rangle - Vg_t(\mathring{\bm \lambda}(t)))\right )\right ]}\\
&\scalemath{0.95}{=\O\left (\frac{R(2NM+R)}{r^7} d^{14/3} (C^\lambda T^{1/2-\delta_\lambda})^2\right )+\O\left (\E\left [\left (\frac{R}{r} d^{2/3}+R\right ) (C^\lambda T^{1/2 - \delta_\lambda})^{1/4} \left (\sum_{t=1}^T \left (\lVert \bm Q(t)\rVert_2 + V(L+G)\right )^{4/3}\right )^{3/4}\right ]\right ).}
\end{align*}
\end{theorem}

In words, it means that the optimization objective \Cref{eq:multi-hop utility BCO objective} is ensured.
The second term on the RHS is a key term as it ensures property \textit{ii)} -- which is due to the $\eta_t$ definition in \Cref{eq:true learning rate in BCO} which contains all historical magnitudes and Lipschitzness.
Below, we quick overview this proof and see why \textit{ii)} can be ensured by the $\eta_t$ defined in \Cref{eq:true learning rate in BCO restatement}. A formal version is included in \Cref{sec:general bandit part appendix}.

\begin{proof}[Proof Sketch of \Cref{thm:multi-hop utility decision making}]
The terms $(\lVert \bm Q(t)\rVert_\infty+VG)$ and $(\lVert \bm Q(t)\rVert_2+VL)$ are the boundedness and Lipschitzness of $\ell_t$, respectively, which we denote by $C_t$ and $L_t$ to simplify notations.

First suppose that all conditions in \Cref{lem:ABGD guarantee} are satisfied by our $\{\eta_t\}_{t\in [T]}$.
To balance the term inside $\sum_{t=1}^T (\cdot)$ in \Cref{lem:ABGD guarantee} by fixing $\eta_t$ and altering $\delta_t$, one shall pick $\delta_t=(\eta_t d^2 C_t^2 / L_t)^{1/3}$ and roughly have (hiding many constant terms; see \Cref{eq:BCO regret decomposition} in the appendix for an accurate form):
\begin{equation}\label{eq:BCO regret decomposition informal}
\text{D-Regret}_T^{\text{BCO}}(\mathring{\bm \lambda}(1),\mathring{\bm \lambda}(2),\ldots,\mathring{\bm \lambda}(T))\lesssim \frac{C^\lambda T^{1/2-\delta_\lambda}}{\eta_T} + \sum_{t=1}^T \left (\eta_t C_t^2 L_t^2\right )^{1/3}.
\end{equation}

We derive in \Cref{lem:3/4 summation lemma} that $\sum_{t=1}^T \frac{x_t}{(\sum_{s\le t} x_s)^{1/4}}\lesssim (\sum_{t=1}^T x_t)^{3/4}$, which is a variant of the famous summation lemma \citep{auer2002nonstochastic}.
Therefore, if we pick $\eta_t\approx \left (\frac{C^\lambda T^{1/2-\delta_\lambda}}{\sum_{s=1}^t (C_t^2 L_t^2)^{1/3}}\right )^{3/4}$ (\textit{i.e.}, only keeping the third term in the denominator of \Cref{eq:true learning rate in BCO restatement}), \Cref{eq:BCO regret decomposition informal} becomes $\O((C^\lambda T^{1/2-\delta_\lambda})^{1/4} \E[(\sum_{t=1}^T (C_t^2 L_t^2)^{1/3})^{3/4}])$.
When focusing on $T$-related terms, this is roughly $\O_T(T^{1/8-\delta_\lambda/4} \E[(\sum_{t=1}^T \lVert \bm Q(t)\rVert_1)]^{7/8})$, which allows the self-bounding analysis similar to \Cref{eq:self-bounding informal}.
However, such a configuration of $\{\eta_t\}_{t\in [T]}$ may not ensure the $\alpha_t = \delta_t / r = (\eta_t d^2 C_t^2 / L_t)^{1/3}/r<1$ condition in \Cref{lem:ABGD guarantee}.
The other two terms in \Cref{eq:true learning rate in BCO restatement} are added for this purpose. We refer the readers to \Cref{sec:general bandit part appendix} for detailed verification.
\end{proof}

\subsection{Main Theorem for Multi-Hop Utility Maximization}\label{sec:multi-hop utility main theorem}
As sketched in \Cref{sec:multi-hop utility sketch}, if we use a Lyapunov drift-plus-penalty analysis, exploit the network stability assumption, use \OLOAlg (\Cref{alg:OLO unbounded}) to decide link allocations $\bm a(t)$, and use \BCOAlg (\Cref{alg:ABGD}) to decide arrival rates $\bm \lambda(t)$, we get the following utility maximization guarantee.
\begin{theorem}[Main Theorem for Multi-Hop Utility Maximization]\label{thm:multi-hop utility main theorem}
Suppose that the feasible set of arrival rates vector $\Lambda$ is bounded by $[r,R]$.
Assume all (unknown) utility functions $g_t$ to be concave, $L$-Lipschitz, and $[-G,G]$-bounded.
Consider a reference action sequence $\{(\mathring{\bm a}(t),\mathring{\bm \lambda}(t))\}_{t\in [T]}$ satisfying \Cref{def:multi-hop utility network stability}, such that their path lengths satisfy
\begin{align*}
P_t^a\triangleq \sum_{s=1}^{t-1} \lVert \mathring{\bm a}(s)-\mathring{\bm a}(s+1)\rVert_1\le C^a t^{1/2-\delta_a},~P_t^\lambda\triangleq \sum_{s=1}^{t-1} \lVert \mathring{\bm \lambda}(s)-\mathring{\bm \lambda}(s+1)\rVert_1\le C^\lambda t^{1/2 - \delta_\lambda},\quad \forall t\in [T].
\end{align*}
Here, $M,R,r,L,G,C^a,\delta_a,C^\lambda,\delta_\lambda$ are assumed to be known constants, whereas the specific $\{(\mathring{\bm a}(t),\mathring{\bm \lambda}(t))\}_{t\in [T]}$ remains unknown.
If we execute the \UMAlg framework in \Cref{alg:utility} with the $\OLO$ sub-rountine given in \Cref{alg:OLO unbounded} and the $\BCO$ sub-routine given in \Cref{alg:ABGD}, when $T$ is large enough such that the constant $V=o_T(\min\{T^{2\delta_a/3},T^{2\delta_\lambda/7}\})$, the following inequalities hold simultaneously:
\begin{align*}
&\scalemath{0.95}{\frac 1T \E\left [\sum_{t=1}^T \lVert \bm Q(t)\rVert_1\right ]=\O\left (\frac{(N^2 (2NM+R)^2 + \epsilon_W N^2 (2NM+R)) C_W + (N^4 M^2+N^2 R^2)}{\epsilon_W}\right ) + o_T(1),}\\
&\scalemath{0.95}{\frac 1T \E\left [\sum_{t=1}^T \left (g_{t}(\mathring{\bm \lambda}(t)) - g_{t}(\bm \lambda(t))\right )\right ]=\O\left (\frac{(N^2 (2NM+R)^2 + \epsilon_W N^2 (2NM+R)) C_W + (N^4 M^2+N^2 R^2)}{V}\right ) + o_T(V^{-1}).}
\end{align*}
That is, when $T\gg 0$, our algorithm not only stabilizes the system so that $\frac 1T \E\left [\sum_{t=1}^T \lVert \bm Q(t)\rVert_1\right ]=\O_T(1)$, but also enjoys an average utility approaching that of the reference policy polynomially fast, \textit{i.e.}, $\frac 1T \E\left [\sum_{t=1}^T \left (g_{t}(\mathring{\bm \lambda}(t)) - g_{t}(\bm \lambda(t))\right )\right ]=\O_T(V^{-1})$ -- the utility maximization objective \Cref{eq:average reward} is ensured.
\end{theorem}
\begin{remark}\label{remark:utility main theorem}
The condition $V=o_T(\min\{T^{2\delta_a/3},T^{2\delta_\lambda/7}\})$ says $T$ cannot be too small compared to $V$, which was not an issue in SNO as people often let $T$ approach infinity \citep{neely2008fairness}.
Albeit this condition looks restrictive, we remark that $V$ can still be as large as a polynomial of $T$ and thus $\O_T(V^{-1})$ means a polynomially decaying gap between our utility and that of any mildly varying policies whose path lengths are small -- which is the first guarantee that applies to utility maximization tasks in adversarial networks under bandit feedback. 
Similar to the discussions in \Cref{remark:stability main theorem}, due to non-stationary environments and bandit feedback, it is highly non-trivial to define ``optimal reference policy'' in ANO. Nevertheless, our mildly varying reference policy class allows optimal policies for SNO settings.
\end{remark}

The full proof of \Cref{thm:multi-hop utility main theorem} is presented in \Cref{sec:multi-hop utility appendix}.
We outline three key steps here.
\begin{proof}[Proof Sketch of \Cref{thm:multi-hop utility main theorem}]
The first step of the proof is plugging the algorithmic guarantees for \OLOAlg (\Cref{thm:multi-hop stability decision making}) and for \BCOAlg (\Cref{thm:multi-hop utility decision making}) into the reference policy assumption (\Cref{eq:multi-hop utility network stability}) and then making use of the Lyapunov DPP guarantee (\Cref{eq:multi-hop utility Lyapunov}). We present the detailed derivation in \Cref{lem:multi-hop utility tedious calc}. The conclusion reads
\begin{align}
\E\left [\sum_{t=1}^T \lVert \bm Q(t)\rVert_1\right ]&\le -\frac{V}{\epsilon_W} \E\left [\sum_{t=1}^T \bigg (g_t(\mathring{\bm \lambda}(t)) - g_t(\bm \lambda(t))\bigg )\right ] + f(T) + \nonumber\\
&\quad g(T) \E\left [\sum_{t=1}^T \lVert \bm Q(t)\rVert_1\right ]^{3/4} \log \E \left [\sum_{t=1}^T \lVert \bm Q(t)\rVert_1\right ] + h(T) \E \left [\sum_{t=1}^T \lVert \bm Q(t)\rVert_1\right ]^{7/8},\label{eq:multi-hop utility self-bounding informal}
\end{align}
where $f(T)=\O_T(T)$, $g(T)=\O_T(T^{1/4-\delta_a/2})$, and $h(T)=\O_T(T^{1/8-\delta_\lambda/4})$.

\paragraph{Step 1 (Develop a Coarse Average Queue Length Bound).} By the boundedness of $g_t$, the first term on the RHS of \Cref{eq:multi-hop utility self-bounding informal} is controlled by $\frac{V}{\epsilon_W} T G=\O_T(VT)$ (note that, as $V=\text{poly}(T)$, $V$ is not a constant that can be hidden in $\O_T$ and this term is actually super-linear).
Similar to the one used when proving \Cref{thm:multi-hop stability main theorem}, we develop another self-bounding property that says $y\le f+y^{3/4}g\log y +y^{7/8}h$ infers $y=\O(f)+\Otil(g^{4}) + \O(h^8)$ (see \Cref{lem:3/4 and 7/8 self-bounding}). Therefore,
\begin{equation*}
\E\left [\sum_{t=1}^T \lVert \bm Q(t)\rVert_1\right ]=\O_T\left (VT + T\right ) + \Otil_T\left (T^{1-2\delta_a}\right ) + \O_T\left (T^{1 - 2 \delta_\lambda}\right )=\O_T(VT),
\end{equation*}
only giving a $\frac 1T\E[\sum_{t=1}^T \lVert \bm Q(t)\rVert_1]=\O_T(V)=\omega_T(1)$ bound on the average queue length (violating the system stability condition \Cref{eq:average queue}). However, this inequality can be used to derive the polynomial convergence result on the utility, which in turn further refines the queue length bound.

\paragraph{Step 2 (Yield Polynomial Convergence on the Utility).}
Moving the difference in the average utility to the LHS in \Cref{eq:multi-hop utility self-bounding informal} and plugging in the just-derived bound on average queue length, we have
\begin{align*}
\frac{V}{\epsilon_W T} \E\left [\sum_{t=1}^T \left (g_t(\mathring{\bm \lambda}(t)) - g_t(\bm \lambda(t)))\right )\right ]&=\O_T\left (-0+\frac{f(T)}{T}+\frac{g(T)}{T}(VT)^{3/4}+\frac{h(T)}{T}(VT)^{7/8}\right )\\
&=\O_T\left (1+\left (\frac{T^{1-2\delta_a}V^3}{T}\right )^{1/4}+\left (\frac{T^{1-2\delta_\lambda}V^7}{T}\right )^{1/8}\right ).
\end{align*}

According to the assumption that $V=\O_T(\min\{T^{2\delta_a/3},T^{2\delta_\lambda/7}\})$, the RHS is of order $\O_T(1)$. Thus
\begin{equation*}
\frac 1T \E\left [\sum_{t=1}^T \left (g_t(\mathring{\bm \lambda}(t)) - g_t(\bm \lambda(t)))\right )\right ]=\O_T(V^{-1}),
\end{equation*}
\textit{i.e.}, a polynomial convergence rate on the expected average utility is derived.

\paragraph{Step 3 (Refine the Average Queue Length Bound).}
Now we are ready to refine our average queue length bound. Instead of controlling the utility with boundedness $g_t\in [-G,G]$, we utilize the just-derived convergence result and yields (again using the self-bounding property in \Cref{lem:3/4 and 7/8 self-bounding})
\begin{equation*}
\E\left [\sum_{t=1}^T \lVert \bm Q(t)\rVert_1\right ]=\O_T\left (V \times V^{-1} T + T\right ) + \Otil_T\left (T^{1-2\delta_a}\right ) + \O_T\left (T^{1-2\delta_\lambda}\right )=\O_T(T),
\end{equation*}
that is, the average queue length $\frac 1T\E\left [\sum_{t=1}^T \lVert \bm Q(t)\rVert_1\right ]=\O_T(1)$, which means \Cref{eq:average queue} holds and the system is stable. Putting Step 2 and Step 3 together gives our conclusion.

Once again, we omitted all factors except for those $\text{poly}(T)$ ones throughout this proof sketch. Please refer to \Cref{sec:multi-hop utility appendix} for the complete version.
\end{proof}

\section{Conclusion}
We study utility maximization in Adversarial Network Optimization (ANO) under bandit feedback.
We design a network stability algorithm \SSAlg and a utility maximization algorithm \UMAlg, which both ingeniously integrate online learning components into Lyapunov drift framework to allow a joint analysis.
When designing the online learning components of \UMAlg, due to the potentially unbounded queue lengths in network optimization and the self-bounding analysis we want to conduct, we develop a novel OLO algorithm \OLOAlg which adapts to occasionally large losses and a BCO algorithm \BCOAlg which suites large loss magnitudes and Lipschitzness via a meticulous learning rate scheduling scheme.
One important future research direction will be defining other alternative reference policy classes that allows competing to more policies, even the optimal ones.

\bibliography{references}

\begin{thebibliography}{48}
\providecommand{\natexlab}[1]{#1}
\providecommand{\url}[1]{\texttt{#1}}
\expandafter\ifx\csname urlstyle\endcsname\relax
  \providecommand{\doi}[1]{doi: #1}\else
  \providecommand{\doi}{doi: \begingroup \urlstyle{rm}\Url}\fi

\bibitem[Andrews and Zhang(2004)]{andrews2004scheduling}
Matthew Andrews and Lisa Zhang.
\newblock Scheduling over nonstationary wireless channels with finite rate sets.
\newblock In \emph{IEEE INFOCOM 2004}, volume~3, pages 1694--1704. IEEE, 2004.

\bibitem[Andrews and Zhang(2005)]{andrews2005scheduling}
Matthew Andrews and Lisa Zhang.
\newblock Scheduling over a time-varying user-dependent channel with applications to high-speed wireless data.
\newblock \emph{Journal of the ACM (JACM)}, 52\penalty0 (5):\penalty0 809--834, 2005.

\bibitem[Andrews et~al.(2001)Andrews, Awerbuch, Fern{\'a}ndez, Leighton, Liu, and Kleinberg]{andrews2001universal}
Matthew Andrews, Baruch Awerbuch, Antonio Fern{\'a}ndez, Tom Leighton, Zhiyong Liu, and Jon Kleinberg.
\newblock Universal-stability results and performance bounds for greedy contention-resolution protocols.
\newblock \emph{Journal of the ACM (JACM)}, 48\penalty0 (1):\penalty0 39--69, 2001.

\bibitem[Andrews et~al.(2007)Andrews, Jung, and Stolyar]{andrews2007stability}
Matthew Andrews, Kyomin Jung, and Alexander Stolyar.
\newblock Stability of the max-weight routing and scheduling protocol in dynamic networks and at critical loads.
\newblock In \emph{Proceedings of the thirty-ninth annual ACM symposium on Theory of computing}, pages 145--154, 2007.

\bibitem[Ashjaei et~al.(2021)Ashjaei, Bello, Daneshtalab, Patti, Saponara, and Mubeen]{ashjaei2021time}
Mohammad Ashjaei, Lucia~Lo Bello, Masoud Daneshtalab, Gaetano Patti, Sergio Saponara, and Saad Mubeen.
\newblock Time-sensitive networking in automotive embedded systems: State of the art and research opportunities.
\newblock \emph{Journal of systems architecture}, 117:\penalty0 102137, 2021.

\bibitem[Auer(2002)]{auer2002using}
Peter Auer.
\newblock Using confidence bounds for exploitation-exploration trade-offs.
\newblock \emph{Journal of Machine Learning Research}, 3\penalty0 (Nov):\penalty0 397--422, 2002.

\bibitem[Auer et~al.(2002)Auer, Cesa-Bianchi, Freund, and Schapire]{auer2002nonstochastic}
Peter Auer, Nicolo Cesa-Bianchi, Yoav Freund, and Robert~E Schapire.
\newblock The nonstochastic multiarmed bandit problem.
\newblock \emph{SIAM journal on computing}, 32\penalty0 (1):\penalty0 48--77, 2002.

\bibitem[Borodin et~al.(2001)Borodin, Kleinberg, Raghavan, Sudan, and Williamson]{borodin2001adversarial}
Allan Borodin, Jon Kleinberg, Prabhakar Raghavan, Madhu Sudan, and David~P Williamson.
\newblock Adversarial queuing theory.
\newblock \emph{Journal of the ACM (JACM)}, 48\penalty0 (1):\penalty0 13--38, 2001.

\bibitem[Chen and Giannakis(2018)]{chen2018bandit}
Tianyi Chen and Georgios~B Giannakis.
\newblock Bandit convex optimization for scalable and dynamic iot management.
\newblock \emph{IEEE Internet of Things Journal}, 6\penalty0 (1):\penalty0 1276--1286, 2018.

\bibitem[Cholvi and Echag{\"u}e(2007)]{cholvi2007stability}
Vicent Cholvi and Juan Echag{\"u}e.
\newblock Stability of fifo networks under adversarial models: State of the art.
\newblock \emph{Computer Networks}, 51\penalty0 (15):\penalty0 4460--4474, 2007.

\bibitem[Choudhury et~al.(2021)Choudhury, Joshi, Wang, and Shakkottai]{choudhury2021job}
Tuhinangshu Choudhury, Gauri Joshi, Weina Wang, and Sanjay Shakkottai.
\newblock Job dispatching policies for queueing systems with unknown service rates.
\newblock In \emph{Proceedings of the Twenty-second International Symposium on Theory, Algorithmic Foundations, and Protocol Design for Mobile Networks and Mobile Computing}, pages 181--190, 2021.

\bibitem[Cruz(1991)]{cruz1991calculus}
Rene~L Cruz.
\newblock A calculus for network delay. i. network elements in isolation.
\newblock \emph{IEEE Transactions on information theory}, 37\penalty0 (1):\penalty0 114--131, 1991.

\bibitem[Cutkosky(2020)]{cutkosky2020parameter}
Ashok Cutkosky.
\newblock Parameter-free, dynamic, and strongly-adaptive online learning.
\newblock In \emph{International Conference on Machine Learning}, pages 2250--2259. PMLR, 2020.

\bibitem[Dai and Gluzman(2022)]{dai2022queueing}
Jim~G Dai and Mark Gluzman.
\newblock Queueing network controls via deep reinforcement learning.
\newblock \emph{Stochastic Systems}, 12\penalty0 (1):\penalty0 30--67, 2022.

\bibitem[Dai et~al.(2023)Dai, Luo, Wei, and Zimmert]{dai2023refined}
Yan Dai, Haipeng Luo, Chen-Yu Wei, and Julian Zimmert.
\newblock Refined regret for adversarial mdps with linear function approximation.
\newblock In \emph{International Conference on Machine Learning}, pages 6726--6759. PMLR, 2023.

\bibitem[Duchi et~al.(2011)Duchi, Hazan, and Singer]{duchi2011adaptive}
John Duchi, Elad Hazan, and Yoram Singer.
\newblock Adaptive subgradient methods for online learning and stochastic optimization.
\newblock \emph{Journal of machine learning research}, 12\penalty0 (7), 2011.

\bibitem[Flaxman et~al.(2005)Flaxman, Kalai, and McMahan]{flaxman2005online}
Abraham~D Flaxman, Adam~Tauman Kalai, and H~Brendan McMahan.
\newblock Online convex optimization in the bandit setting: gradient descent without a gradient.
\newblock In \emph{Proceedings of the sixteenth annual ACM-SIAM symposium on Discrete algorithms}, pages 385--394, 2005.

\bibitem[Fu and Modiano(2022)]{fu2022joint}
Xinzhe Fu and Eytan Modiano.
\newblock Joint learning and control in stochastic queueing networks with unknown utilities.
\newblock \emph{Proceedings of the ACM on Measurement and Analysis of Computing Systems}, 6\penalty0 (3):\penalty0 1--32, 2022.

\bibitem[Gaddam et~al.(2020)Gaddam, Wilkin, Angelova, and Gaddam]{gaddam2020detecting}
Anuroop Gaddam, Tim Wilkin, Maia Angelova, and Jyotheesh Gaddam.
\newblock Detecting sensor faults, anomalies and outliers in the internet of things: A survey on the challenges and solutions.
\newblock \emph{Electronics}, 9\penalty0 (3):\penalty0 511, 2020.

\bibitem[Harrison and Wein(1990)]{harrison1990scheduling}
J~Michael Harrison and Lawrence~M Wein.
\newblock Scheduling networks of queues: Heavy traffic analysis of a two-station closed network.
\newblock \emph{Operations research}, 38\penalty0 (6):\penalty0 1052--1064, 1990.

\bibitem[Huang et~al.(2024)Huang, Golubchik, and Huang]{huang2023queue}
Jiatai Huang, Leana Golubchik, and Longbo Huang.
\newblock When lyapunov drift based queue scheduling meets adversarial bandit learning.
\newblock \emph{IEEE/ACM Transactions on Networking}, 2024.

\bibitem[Huang and Neely(2011)]{huang2011utility}
Longbo Huang and Michael~J Neely.
\newblock Utility optimal scheduling in processing networks.
\newblock \emph{Performance Evaluation}, 68\penalty0 (11):\penalty0 1002--1021, 2011.

\bibitem[Huang et~al.(2012)Huang, Moeller, Neely, and Krishnamachari]{huang2012lifo}
Longbo Huang, Scott Moeller, Michael~J Neely, and Bhaskar Krishnamachari.
\newblock Lifo-backpressure achieves near-optimal utility-delay tradeoff.
\newblock \emph{IEEE/ACM Transactions On Networking}, 21\penalty0 (3):\penalty0 831--844, 2012.

\bibitem[Khan et~al.(2020)Khan, Das, and Pati]{khan2020channel}
Md~Rizwan Khan, Bikramaditya Das, and Bibhuti~Bhusan Pati.
\newblock Channel estimation strategies for underwater acoustic (uwa) communication: An overview.
\newblock \emph{Journal of the Franklin Institute}, 357\penalty0 (11):\penalty0 7229--7265, 2020.

\bibitem[Krishnasamy et~al.(2018)Krishnasamy, Akhil, Arapostathis, Sundaresan, and Shakkottai]{krishnasamy2018augmenting}
Subhashini Krishnasamy, PT~Akhil, Ari Arapostathis, Rajesh Sundaresan, and Sanjay Shakkottai.
\newblock Augmenting max-weight with explicit learning for wireless scheduling with switching costs.
\newblock \emph{IEEE/ACM Transactions on Networking}, 26\penalty0 (6):\penalty0 2501--2514, 2018.

\bibitem[Krishnasamy et~al.(2021)Krishnasamy, Sen, Johari, and Shakkottai]{krishnasamy2021learning}
Subhashini Krishnasamy, Rajat Sen, Ramesh Johari, and Sanjay Shakkottai.
\newblock Learning unknown service rates in queues: A multiarmed bandit approach.
\newblock \emph{Operations research}, 69\penalty0 (1):\penalty0 315--330, 2021.

\bibitem[Liang and Modiano(2018{\natexlab{a}})]{liang2018network}
Qingkai Liang and Evtan Modiano.
\newblock Network utility maximization in adversarial environments.
\newblock In \emph{IEEE INFOCOM 2018-IEEE Conference on Computer Communications}, pages 594--602. IEEE, 2018{\natexlab{a}}.

\bibitem[Liang and Modiano(2018{\natexlab{b}})]{liang2018minimizing}
Qingkai Liang and Eytan Modiano.
\newblock Minimizing queue length regret under adversarial network models.
\newblock \emph{Proceedings of the ACM on Measurement and Analysis of Computing Systems}, 2\penalty0 (1):\penalty0 1--32, 2018{\natexlab{b}}.

\bibitem[Lim et~al.(2013)Lim, Jung, and Andrews]{lim2013stability}
Sungsu Lim, Kyomin Jung, and Matthew Andrews.
\newblock Stability of the max-weight protocol in adversarial wireless networks.
\newblock \emph{IEEE/ACM Transactions on Networking}, 22\penalty0 (6):\penalty0 1859--1872, 2013.

\bibitem[Liu et~al.(2022)Liu, Xie, and Modiano]{liu2022rl}
Bai Liu, Qiaomin Xie, and Eytan Modiano.
\newblock Rl-qn: A reinforcement learning framework for optimal control of queueing systems.
\newblock \emph{ACM Transactions on Modeling and Performance Evaluation of Computing Systems}, 7\penalty0 (1):\penalty0 1--35, 2022.

\bibitem[Maguluri et~al.(2012)Maguluri, Srikant, and Ying]{maguluri2012stochastic}
Siva~Theja Maguluri, Rayadurgam Srikant, and Lei Ying.
\newblock Stochastic models of load balancing and scheduling in cloud computing clusters.
\newblock In \emph{2012 Proceedings IEEE Infocom}, pages 702--710. IEEE, 2012.

\bibitem[McMahan and Streeter(2010)]{mcmahan2010adaptive}
H~Brendan McMahan and Matthew Streeter.
\newblock Adaptive bound optimization for online convex optimization.
\newblock \emph{Annual Conference on Learning Theory 2010}, page 244, 2010.

\bibitem[Neely(2008)]{neely2008order}
Michael~J Neely.
\newblock Order optimal delay for opportunistic scheduling in multi-user wireless uplinks and downlinks.
\newblock \emph{IEEE/ACM Transactions on Networking}, 16\penalty0 (5):\penalty0 1188--1199, 2008.

\bibitem[Neely(2009)]{neely2009delay}
Michael~J Neely.
\newblock Delay analysis for max weight opportunistic scheduling in wireless systems.
\newblock \emph{IEEE Transactions on Automatic Control}, 54\penalty0 (9):\penalty0 2137--2150, 2009.

\bibitem[Neely(2010{\natexlab{a}})]{neely2010stochastic}
Michael~J Neely.
\newblock Stochastic network optimization with application to communication and queueing systems.
\newblock \emph{Synthesis Lectures on Communication Networks}, 3\penalty0 (1):\penalty0 1--211, 2010{\natexlab{a}}.

\bibitem[Neely(2010{\natexlab{b}})]{neely2010universal}
Michael~J Neely.
\newblock Universal scheduling for networks with arbitrary traffic, channels, and mobility.
\newblock In \emph{49th IEEE Conference on Decision and Control (CDC)}, pages 1822--1829. IEEE, 2010{\natexlab{b}}.

\bibitem[Neely et~al.(2008)Neely, Modiano, and Li]{neely2008fairness}
Michael~J Neely, Eytan Modiano, and Chih-Ping Li.
\newblock Fairness and optimal stochastic control for heterogeneous networks.
\newblock \emph{IEEE/ACM Transactions On Networking}, 16\penalty0 (2):\penalty0 396--409, 2008.

\bibitem[Neely et~al.(2012)Neely, Rager, and La~Porta]{neely2012max}
Michael~J Neely, Scott~T Rager, and Thomas~F La~Porta.
\newblock Max weight learning algorithms for scheduling in unknown environments.
\newblock \emph{IEEE Transactions on Automatic Control}, 57\penalty0 (5):\penalty0 1179--1191, 2012.

\bibitem[Rahdar et~al.(2018)Rahdar, Wang, and Hu]{rahdar2018tri}
Mohammad Rahdar, Lizhi Wang, and Guiping Hu.
\newblock A tri-level optimization model for inventory control with uncertain demand and lead time.
\newblock \emph{International Journal of Production Economics}, 195:\penalty0 96--105, 2018.

\bibitem[Sadiq and De~Veciana(2009)]{sadiq2009throughput}
Bilal Sadiq and Gustavo De~Veciana.
\newblock Throughput optimality of delay-driven maxweight scheduler for a wireless system with flow dynamics.
\newblock In \emph{2009 47th Annual Allerton Conference on Communication, Control, and Computing (Allerton)}, pages 1097--1102. IEEE, 2009.

\bibitem[Srikant and Ying(2013)]{srikant2013communication}
Rayadurgam Srikant and Lei Ying.
\newblock \emph{Communication networks: an optimization, control, and stochastic networks perspective}.
\newblock Cambridge University Press, 2013.

\bibitem[Tsibonis et~al.(2003)Tsibonis, Georgiadis, and Tassiulas]{tsibonis2003exploiting}
Vagelis Tsibonis, Leonidas Georgiadis, and Leandros Tassiulas.
\newblock Exploiting wireless channel state information for throughput maximization.
\newblock In \emph{IEEE INFOCOM 2003. Twenty-second Annual Joint Conference of the IEEE Computer and Communications Societies (IEEE Cat. No. 03CH37428)}, volume~1, pages 301--310. IEEE, 2003.

\bibitem[Wei et~al.(2024)Wei, Liu, Wang, and Ying]{wei2024sample}
Honghao Wei, Xin Liu, Weina Wang, and Lei Ying.
\newblock Sample efficient reinforcement learning in mixed systems through augmented samples and its applications to queueing networks.
\newblock \emph{Advances in Neural Information Processing Systems}, 36, 2024.

\bibitem[Yang et~al.(2023)Yang, Srikant, and Ying]{yang2023learning}
Zixian Yang, R~Srikant, and Lei Ying.
\newblock Learning while scheduling in multi-server systems with unknown statistics: Maxweight with discounted ucb.
\newblock In \emph{International Conference on Artificial Intelligence and Statistics}, pages 4275--4312. PMLR, 2023.

\bibitem[Zhao et~al.(2021)Zhao, Wang, Zhang, and Zhou]{zhao2021bandit}
Peng Zhao, Guanghui Wang, Lijun Zhang, and Zhi-Hua Zhou.
\newblock Bandit convex optimization in non-stationary environments.
\newblock \emph{Journal of Machine Learning Research}, 22\penalty0 (125):\penalty0 1--45, 2021.

\bibitem[Zheng et~al.(2019)Zheng, Luo, Diakonikolas, and Wang]{zheng2019equipping}
Kai Zheng, Haipeng Luo, Ilias Diakonikolas, and Liwei Wang.
\newblock Equipping experts/bandits with long-term memory.
\newblock \emph{Advances in Neural Information Processing Systems}, 32, 2019.

\bibitem[Zinkevich(2003)]{zinkevich2003online}
Martin Zinkevich.
\newblock Online convex programming and generalized infinitesimal gradient ascent.
\newblock In \emph{Proceedings of the 20th international conference on machine learning (icml-03)}, pages 928--936, 2003.

\bibitem[Zou et~al.(2016)Zou, Zhu, Wang, and Hanzo]{zou2016survey}
Yulong Zou, Jia Zhu, Xianbin Wang, and Lajos Hanzo.
\newblock A survey on wireless security: Technical challenges, recent advances, and future trends.
\newblock \emph{Proceedings of the IEEE}, 104\penalty0 (9):\penalty0 1727--1765, 2016.

\end{thebibliography}

\newpage
\appendix
\renewcommand{\appendixpagename}{\centering \LARGE Supplementary Materials}
\appendixpage

\startcontents[section]
\printcontents[section]{l}{1}{\setcounter{tocdepth}{2}}

\section{Additional Related Works}\label{sec:more related work}
\paragraph{Adversarial Components in Network Optimization}
The study of adversarial components in network optimization dates back to the 1990s, when \citet{cruz1991calculus} gave the first network model with adversarial dynamics. This was further generalized to Adversarial Queueing Theory \citep{borodin2001adversarial} and Leaky Bucket \citep{andrews2001universal} models. Many follow-up works, as surveyed by \citet{cholvi2007stability}, considered network optimization under various types of adversarial traffic injections (\textit{i.e.}, the arrival rates to each queue are adversarial). However, these early works only considered adversarial arrival rates but assume link conditions are stationary, which cannot capture the fact that wireless communication networks can have very different link conditions from time to time due to congestions \citep{zou2016survey}.

Noticing this shortcoming, \citet{andrews2004scheduling} and \citet{andrews2005scheduling} studied a single-hop network where link conditions are also adversarial. Their works were extended to multi-hop networks by \citet{andrews2007stability} and \citet{lim2013stability}.
For a more up-to-date discussion on these works, we refer the readers to the discussions in \citep{liang2018minimizing}.

\textit{Utility Maximization in Adversarial Networks.}
While it has been more and more results considering network stability in adversarial networks, the utility maximization guarantees are not so common.
\citet{neely2010universal} proposed the universal network utility maximization problem which considers competing with a look-ahead policy that has perfect knowledge about the near future.
\citet{liang2018network} generalized the aforementioned stability requirement in another way and showed a trade-off between network stability and utility maximization.
However, both papers assumed perfect knowledge on network conditions, as summarized in \Cref{table}.

\textit{Feedback Models.}
Most previous works considered the \textit{perfect knowledge} model which assumes known network conditions before decison \citep{liang2018minimizing,liang2018network}. Despite its simplicity, this assumption eliminates the hardness of estimating the network topology or link conditions, which we argue is highly non-trivial due to the unpredictability in a drastically varying network like underwater communications \citep{khan2020channel} or IoT systems \citep{gaddam2020detecting}.
The harder \textit{full-information feedback} model assumes no prior knowledge at decision-making but requires the network conditions to be fully revealed after decision; a variant of this model was proposed by \citet{neely2012max} under the name of 2-stage decision model.
In this paper, we consider the hardest \textit{bandit feedback} model, which rules out all counterfactual feedback that are associated with the actions that are not really deployed. This model, under various different names, were recently proposed and considered by \citet{fu2022joint}, \citet{yang2023learning}, and \citet{huang2023queue}.

\textit{Adversarial Networks under Bandit Feedback.}
As we are aware of, the papers closest to ours are the ones by \citet{huang2023queue} and
\citet{yang2023learning}, which also studied adversarial networks under bandit feedback.
However, both papers assumed a single-hop network model, \textit{i.e.}, jobs immediately leave the network upon being served. This model finds shortcoming when trying to reflect the reality where some jobs may be forwarded within the network for many hops -- for example, in the classical criss-cross network extensively studied in the SNO literature \citep{harrison1990scheduling}. In contrast, our multi-hop model allows jobs being forwarded from one server to another and is much more general.

Moreover, both papers only considered the task of network stability, \textit{i.e.}, the number of jobs remaining in the network (which is the sum of queue lengths) does not diverge; see \Cref{eq:average queue}. However, only stabilizing the system may not be enough in many realistic problems, where the throughput (average number of jobs getting served) \citep{tsibonis2003exploiting,sadiq2009throughput} or delay (average waiting time for each job from entering the system to being served) \citep{neely2008order,neely2009delay} should be optimized.
In our paper, we consider the utility maximization task \citep{huang2011utility,huang2012lifo} where an abstract utility function shall be optimized, allowing various network optimization objectives other than simply stabilizing the system.

\textit{Learning-Augmented Algorithms in Network Optimization.}
Finally, we give a brief overview of recent learning-augmented algorithms in network optimization.
To tackle the lack of accurate channel information (\textit{e.g.}, under the feedback model), exploration approaches like $\epsilon$-greedy \citep{krishnasamy2018augmenting,krishnasamy2021learning} or Upper Confidence Bound (UCB) \citep{auer2002using} (\textit{e.g.}, \citep{choudhury2021job,krishnasamy2021learning,yang2023learning}) were widely used.
More recently, using a Reinforcement Learning (RL) approach, \citet{liu2022rl} proposed the RL-QN algorithm by putting the queue lengths as the state in RL, which outperforms many existing SNO algorithms.
Empirically, utilizing the recent advances in Deep RL (DRL), \citet{dai2022queueing} established state-of-the-art control performance in the criss-cross network. Following this line, \citet{wei2024sample} compressed the number of states and yielded improved performance in other networks as well.
However, most aforementioned works only considered the stochastic regime. Moreover, RL approaches (especially those DRL ones) typically have extremely large space when modelling network optimization problems, making the algorithm computationally infeasible.
In contrast, our online-learning-based approach makes the algorithm not only capable of adversarial environment and bandit feedback but also computationally efficient.

\section{Omitted Proofs for Multi-Hop Network Stability Tasks}
Before presenting the theorem proofs, we first give a bound on the queue length increments.
\begin{lemma}[Queue Length Increment]\label{lem:queue length increment}
For every $n\in \mN$, $k\in \mN$, and $t\in [T]$, we have
\begin{equation*}
\lvert Q_n^{(k)}(t+1)-Q_n^{(k)}(t)\rvert\le (2NM+R).
\end{equation*}
\end{lemma}
\begin{proof}
According to the queue length dynamics in \Cref{eq:multi-hop stability dynamics}, we know
\begin{equation*}
\lvert Q_n^{(k)}(t+1)-Q_n^{(k)}(t)\rvert\le \left \lvert \sum_{(n,m)\in \mL} \mu_{n,m}^{(k)}(t) + \sum_{(o,n)\in \mL} \mu_{o,n}^{(k)}(t) + \lambda_n^{(k)}(t) \right \rvert\le (2NM+R),
\end{equation*}
which utilizes the assumption that $\lambda_n^{(k)}(t)\in [0,R]$ and $\mu_{(n,m)}^{(k)}(t)\in [0,M]$.
\end{proof}

\subsection{Reference Policy Assumption (Proof of \Cref{lem:multi-hop stability network stability})}\label{sec:multi-hop stability network stability appendix}
\begin{lemma}[Restatement of \Cref{lem:multi-hop stability network stability}; Ability of $\{\mathring{\bm a}(t)\}_{t\in [T]}$ in Stabilizing the Network]
If $\{\mathring{\bm a}(t)\}_{t\in [T]}$ satisfies \Cref{def:multi-hop stability network stability}, then for any scheduler-generated queue lengths $\{\bm Q(t)\}_{t\in [T]}$,
\begin{align*}
&\quad \epsilon_W \E\left [\sum_{t=1}^T \sum_{n\in \mN} \sum_{k\in \mN} Q_n^{(k)}(t)\right ] - (N^2 (2NM+R)^2 + \epsilon_W N^2 (2NM+R)) C_W T \\
&\le-\E\left [\sum_{t=1}^T \sum_{(n,m)\in \mL} \sum_{k\in \mN} C_{n,m}(t) \mathring{a}_{n,m}^{(k)}(t) (Q_m^{(k)}(t) - Q_n^{(k)}(t))\right ]-\E\left [\sum_{t=1}^T \sum_{n\in \mN} \sum_{k\in \mN} Q_n^{(k)}(t) \lambda_n^{(k)}(t)\right ].
\end{align*}
\end{lemma}
\begin{proof}
The proof almost follows that of \citet[Lemma 2]{huang2023queue}. We adapt their proof here for completeness.
For each interval $W_j$ in \Cref{def:multi-hop stability network stability}, let $T_0$ be the first round in $W_j$. Then
\begin{align*}
&\quad \sum_{t\in W_j} \sum_{n\in \mN} \sum_{k\in \mN} Q_n^{(k)}(t)\left (\sum_{(n,m)\in \mL} C_{n,m}(t) \mathring a_{n,m}^{(k)}(t) - \lambda_n^{(k)}(t) - \sum_{(o,n)\in \mL} C_{o,n}(t) \mathring a_{o,n}^{(k)}(t)\right ) \\
&=\sum_{t\in W_j} \sum_{n\in \mN} \sum_{k\in \mN} Q_n^{(k)}(T_0)\left (\sum_{(n,m)\in \mL} C_{n,m}(t) \mathring a_{n,m}^{(k)}(t) - \lambda_n^{(k)}(t) - \sum_{(o,n)\in \mL} C_{o,n}(t) \mathring a_{o,n}^{(k)}(t)\right ) + \\
&\quad \sum_{t\in W_j} \sum_{n\in \mN} \sum_{k\in \mN} \left (Q_n^{(k)}(t)-Q_n^{(k)}(T_0)\right )\left (\sum_{(n,m)\in \mL} C_{n,m}(t) \mathring a_{n,m}^{(k)}(t) - \lambda_n^{(k)}(t) - \sum_{(o,n)\in \mL} C_{o,n}(t) \mathring a_{o,n}^{(k)}(t)\right ).
\end{align*}

For the first term, we have the following for every $n\in \mN$ and $k\in \mN$, according to \Cref{def:multi-hop stability network stability}:
\begin{align*}
\sum_{t\in W_j} \left (\sum_{(n,m)\in \mL} C_{n,m}(t) \mathring a_{n,m}^{(k)}(t) - \lambda_n^{(k)}(t) - \sum_{(o,n)\in \mL} C_{o,n}(t) \mathring a_{o,n}^{(k)}(t)\right )\ge \epsilon_W \lvert W_j\rvert.
\end{align*}

Thus, the first summation enjoys the following lower bound:
\begin{align*}
&\quad \sum_{t\in W_j} \sum_{n\in \mN} \sum_{k\in \mN} Q_n^{(k)}(T_0)\left (\sum_{(n,m)\in \mL} C_{n,m}(t) \mathring a_{n,m}^{(k)}(t) - \lambda_n^{(k)}(t) - \sum_{(o,n)\in \mL} C_{o,n}(t) \mathring a_{o,n}^{(k)}(t)\right ) \\
&\ge \left (\sum_{n\in \mN} \sum_{k\in \mN} Q_n^{(k)}(T_0)\right ) \epsilon_W \lvert W_j\rvert=\epsilon_W \lvert W_j\rvert \times \lVert \bm Q(T_0)\rVert_1\ge \epsilon_W \sum_{t\in W_j} \lVert \bm Q(t)\rVert_1-\epsilon_W N^2 (2NM+R) (\lvert W_j\rvert-1)^2,
\end{align*}
where the last step uses the fact that $\lVert \bm Q(t)\rVert_1 - \lVert \bm Q(T_0)\rVert_1\le \lVert \bm Q(t)-\bm Q(T_0)\rVert_1$ and the queue length increment bound in \Cref{lem:queue length increment}.

For the second summation, again utilizing the bound on $\lVert \bm Q(t)-\bm Q(T_0)\rVert_1$, we have
\begin{align*}
&\quad \sum_{t\in W_j} \sum_{n\in \mN} \sum_{k\in \mN} \left (Q_n^{(k)}(t)-Q_n^{(k)}(T_0)\right )\left (\sum_{(n,m)\in \mL} C_{n,m}(t) \mathring a_{n,m}^{(k)}(t) - \lambda_n^{(k)}(t) - \sum_{(o,n)\in \mL} C_{o,n}(t) \mathring a_{o,n}^{(k)}(t)\right )\\
&\ge -\sum_{t\in W_j} \lVert \bm Q(t)-\bm Q(T_0)\rVert_1\times \max_{t\in W_j} \max_{n\in \mN} \max_{k\in \mN} \left \lvert \sum_{(n,m)\in \mL} C_{n,m}(t) \mathring a_{n,m}^{(k)}(t) - \lambda_n^{(k)}(t) - \sum_{(o,n)\in \mL} C_{o,n}(t) \mathring a_{o,n}^{(k)}(t) \right \rvert \\
&\ge -N^2 (2NM+R) (\lvert W_j\rvert-1)^2\times (2NM+R)=-N^2 (2NM+R)^2 (\lvert W_j\rvert-1)^2.
\end{align*}

Therefore, recall the assumption that $\sum_{j\in [J]}(\lvert W_j\rvert-1)^2\le C_W T$, summing over $j=1,2,\ldots,J$ gives
\begin{align*}
&\quad \epsilon_W \E\left [\sum_{t=1}^T \sum_{n\in \mN} \sum_{k\in \mN} Q_n^{(k)}(t)\right ] - (N^2 (2NM+R)^2 + \epsilon_W N^2 (2NM+R)) C_W T \\
&\le \E\left [\sum_{t=1}^T \sum_{n\in \mN} \sum_{k\in \mN} Q_n^{(k)}(t)\left (\sum_{(n,m)\in \mL} C_{n,m}(t) \mathring a_{n,m}^{(k)}(t) - \lambda_n^{(k)}(t) - \sum_{(o,n)\in \mL} C_{o,n}(t) \mathring a_{o,n}^{(k)}(t)\right )\right ] \nonumber\\
&=-\E\left [\sum_{t=1}^T \sum_{(n,m)\in \mL} \sum_{k\in \mN} C_{n,m}(t) \mathring{a}_{n,m}^{(k)}(t) (Q_m^{(k)}(t) - Q_n^{(k)}(t))\right ]-\E\left [\sum_{t=1}^T \sum_{n\in \mN} \sum_{k\in \mN} Q_n^{(k)}(t) \lambda_n^{(k)}(t)\right ],
\end{align*}
thus giving our conclusion.
\end{proof}

\subsection{Lyapunov Drift Analysis (Proof of \Cref{lem:multi-hop stability Lyapunov})}
\label{sec:multi-hop stability Lyapunov appendix}
\begin{lemma}[Restatement of \Cref{lem:multi-hop stability Lyapunov}; Lyapunov Drift Analysis]
Under the queue dynamics of \Cref{eq:multi-hop stability dynamics},
\begin{align*}
0\le \E\left [\sum_{t=1}^T \Delta(\bm Q(t))\right ]&\le \E\left [\sum_{t=1}^T \sum_{(n,m)\in \mL} \sum_{k\in \mN} \mu_{n,m}^{(k)}(t)\left (Q_m^{(k)}(t)-Q_n^{(k)}(t)\right ) + \sum_{n\in \mN} \sum_{k\in \mN} Q_n^{(k)}(t) \lambda_n^{(k)}(t)\right ] + \nonumber \\
&\quad \frac 12 N^2 ((NM)^2+2(NM)^2+2R^2) T.
\end{align*}
\end{lemma}
\begin{proof}
According to \Cref{eq:multi-hop stability dynamics}, for any round $t\in [T]$, server $n\in \mN$, and commodity $k\in \mN$, we
\begin{align*}
(Q_n^{(k)}(t+1))^2&\le \left [Q_n^{(k)}(t)-\sum_{(n,m)\in \mL}\mu_{n,m}^{(k)}(t)\right ]_+^2+\left (\sum_{(o,n)\in \mL} \mu_{o,n}^{(k)}(t) + \lambda_{n}^{(k)}(t)\right )^2+\\
&\quad 2 \left [Q_n^{(k)}(t)-\sum_{(n,m)\in \mL}\mu_{n,m}^{(k)}(t)\right ]_+\left (\sum_{(o,n)\in \mL} \mu_{o,n}^{(k)}(t) + \lambda_{n}^{(k)}(t)\right )\\
&\le (Q_n^{(k)}(t))^2 - 2 Q_n^{(k)}(t) \left (\sum_{(n,m)\in \mL}\mu_{n,m}^{(k)}(t)\right ) + \left (\sum_{(n,m)\in \mL}\mu_{n,m}^{(k)}(t)\right )^2 + \\
&\quad \left (\sum_{(o,n)\in \mL} \mu_{o,n}^{(k)}(t) + \lambda_{n}^{(k)}(t)\right )^2+2 Q_n^{(k)}(t) \left (\sum_{(o,n)\in \mL} \mu_{o,n}^{(k)}(t) + \lambda_{n}^{(k)}(t)\right )\\
&= (Q_n^{(k)}(t))^2 - 2 Q_n^{(k)}(t) \left (\sum_{(n,m)\in \mL}\mu_{n,m}^{(k)}(t) - \sum_{(o,n)\in \mL} \mu_{o,n}^{(k)}(t) - \lambda_{n}^{(k)}(t)\right ) + \\ 
&\quad \left (\sum_{(n,m)\in \mL}\mu_{n,m}^{(k)}(t)\right )^2 + \left (\sum_{(o,n)\in \mL} \mu_{o,n}^{(k)}(t) + \lambda_{n}^{(k)}(t)\right )^2.
\end{align*}

Therefore, by definition of the Lyapunov function $L_t$ and the Lyapunov drift $\Delta(\bm Q(t))$,
\begin{align*}
\Delta(\bm Q(t))&\le - \sum_{n\in \mN} \sum_{k\in \mN} Q_n^{(k)}(t) \E\left [\sum_{(n,m)\in \mL}\mu_{n,m}^{(k)}(t) - \sum_{(o,n)\in \mL} \mu_{o,n}^{(k)}(t) - \lambda_{n}^{(k)}(t)\middle \vert \bm Q(t)\right ] + \nonumber\\
&\quad \frac 12 \sum_{n\in \mN} \sum_{k\in \mN} \E\left [\left (\sum_{(n,m)\in \mL}\mu_{n,m}^{(k)}(t)\right )^2 + \left (\sum_{(o,n)\in \mL} \mu_{o,n}^{(k)}(t) + \lambda_{n}^{(k)}(t)\right )^2\middle \vert \bm Q(t)\right ].
\end{align*}

Exchanging summations and using the bounded assumptions that $\mu_{n,m}^{(k)}(t)\in [0,M]$ and $\lambda_n^{(k)}(t)\in [0,R]$, we get
\begin{align*}
\Delta(\bm Q(t))&\le \E\left [\sum_{(n,m)\in \mL} \sum_{k\in \mN} \mu_{n,m}^{(k)}(t)\left (Q_m^{(k)}(t)-Q_n^{(k)}(t)\right ) + \sum_{n\in \mN} \sum_{k\in \mN} Q_n^{(k)}(t) \lambda_n^{(k)}(t)\middle \vert \bm Q(t)\right ] + \nonumber \\
&\quad \frac 12 N^2 ((NM)^2+2(NM)^2+2R^2).
\end{align*}

Taking expectation \textit{w.r.t.} $\bm Q(t)$ and summing up from $t=1,2,\ldots,T$, we have
\begin{align*}
\E\left [\sum_{t=1}^T \Delta(\bm Q(t))\right ]&\le \E\left [\sum_{t=1}^T \sum_{(n,m)\in \mL} \sum_{k\in \mN} \mu_{n,m}^{(k)}(t)\left (Q_m^{(k)}(t)-Q_n^{(k)}(t)\right ) + \sum_{n\in \mN} \sum_{k\in \mN} Q_n^{(k)}(t) \lambda_n^{(k)}(t)\right ] + \nonumber \\
&\quad \frac 12 N^2 ((NM)^2+2(NM)^2+2R^2) T.
\end{align*}

By telescoping sums, we know
\begin{equation*}
\sum_{t=1}^T \Delta(\bm q_t)=L_{T+1}-L_1=L_{T+1}\ge 0,
\end{equation*}
where the last step uses the fact that $L_{T+1}$ is the sum of squares. Therefore, our conclusion follows.
\end{proof}

\subsection{Guarantee of \OLOAlg Algorithm (Proof of \Cref{lem:OLO unbounded guarantee})}\label{sec:appendix cutkosky}
Before analyzing our \OLOAlg algorithm, we first include the guarantee of the \texttt{PFOL} algorithm \citep{cutkosky2020parameter} as follows. It roughly says that for bounded losses (\textit{i.e.}, $\lVert \bm g_t\rVert\le 1$), there exists an algorithm that enjoys the following parameter-free (\textit{i.e.}, performance depending on loss magnitudes) guarantee.
\begin{lemma}[Guarantee of \texttt{PFOL} Algorithm {\citep[Theorem 6]{cutkosky2020parameter}}]\label{lem:OLO guarantee appendix}\label{lem:OLO guarantee}
Consider the OLO problem in \Cref{def:OLO}. Suppose that $\mathcal X$ has diameter $D=\sup_{\bm x,\bm y\in \mathcal X}\lVert \bm x-\bm y\rVert_1$ and all $\lVert \bm g_t\rVert_\infty \le 1$. Then there exists an algorithm, such that for any comparator sequence $\mathring{\bm x}_1,\mathring{\bm x}_2,\ldots,\mathring{\bm x}_T\in \mathcal X$,
\begin{equation*}
\text{D-Regret}_T^{\text{OLO}}(\mathring{\bm x}_1,\mathring{\bm x}_2,\ldots,\mathring{\bm x}_T)=\O\left (\sqrt{D \left (D+\sum_{t=1}^{T-1} \lVert \mathring{\bm x}_t-\mathring{\bm x}_{t+1}\rVert_1\right )}\sqrt{1+\sum_{t=1}^T \lVert \bm g_t\rVert_\infty^2}\log \left (T \sum_{t=1}^T \lVert \bm g_t\rVert_\infty^2\right )\right ).
\end{equation*}
\end{lemma}

Due to the complicatedness of the original algorithm, we do not present its pseudo-code here but instead use it as a black-box. Please refer to the original paper by \citet{cutkosky2020parameter} for more details. Note that, although the original analysis by \citet{cutkosky2020parameter} uses $\ell_2$-norm for both comparators $\{\mathring{\bm x}_t\}_{t=1}^T$ and loss vectors $\{\bm g_t\}_{t=1}^T$, it is straightforward to extend to a pair of dual norms, which is $\ell_1$-norm for $\{\mathring{\bm x}_t\}_{t=1}^T$ and $\ell_\infty$-norm for $\{\bm g_t\}_{t=1}^T$ in our case.

Now, we are ready to present the guarantee for our \OLOAlg algorithm:
\begin{lemma}[Restatement of \Cref{lem:OLO unbounded guarantee}; Guarantee of \OLOAlg Algorithm]
Consider the OLO problem in \Cref{def:OLO}. Let the action set $\mathcal X$ has diameter $D=\sup_{\bm x,\bm y\in \mathcal X}\lVert \bm x-\bm y\rVert_1$. Suppose that $\lVert \bm g_t\rVert_\infty\le G_t$, where $G_t$ is some $\mathcal F_{t-1}$-measurable random variable and $(\mathcal F_t)_{t=0}^T$ is the natural filtration, \textit{i.e.}, $\mathcal F_t$ is the $\sigma$-algebra generated by all random observations made during the first $t$ rounds.
Then, \OLOAlg (\Cref{alg:OLO unbounded}) ensures that for any comparator sequence $\mathring{\bm x}_1,\mathring{\bm x}_2,\ldots,\mathring{\bm x}_T\in \mathcal X$, if $\max_{t\in [T]} G_t\ge 1$, then
\begin{equation*}
\text{D-Regret}_T^{\text{OLO}}(\mathring{\bm x}_1,\mathring{\bm x}_2,\ldots,\mathring{\bm x}_T)=\O\left (\sqrt{D \left (D+\sum_{t=1}^{T-1} \lVert \mathring{\bm x}_t-\mathring{\bm x}_{t+1}\rVert_1\right )}\sqrt{\sum_{t=1}^T \lVert \bm g_t\rVert_\infty^2}\log T \log \left (\max_{t=1}^T G_t\right )\right ).
\end{equation*}
\end{lemma}
\begin{proof}
As $G$ changes only if $G_t>2\max_{s<t} G_s$, it cannot change for more than $\lceil \log_2 (\max_t G_t)\rceil$ times.
For a fixed $G$, suppose that it is used for rounds $t_1,t_1+1,\ldots,t_2$, then we must have $G_t\le G$, $\forall t\in [t_1,t_2]$ as otherwise a new instance of \texttt{PFOL} will be launched.

Therefore, as $\lVert \bm g_t\rVert_\infty\le G_t\le G$, we have $\lVert G^{-1} \bm g_t\rVert_\infty\le 1$. This allows us to apply \Cref{lem:OLO guarantee} and yield
\begin{equation*}
\sum_{t=t_1}^{t_2} \langle G^{-1} \bm g_t,\bm x_t-\bm x_t^\circ\rangle=\O\left (\sqrt{D \left (D+\sum_{t=t_1}^{t_2-1} \lVert \mathring{\bm x}_t-\mathring{\bm x}_{t+1}\rVert_1\right )}\sqrt{\sum_{t=t_1}^{t_2} \lVert G^{-1} \bm g_t\rVert_\infty^2}\log T\right ).
\end{equation*}
Multiplying $G$ on both sides, we have
\begin{equation}\label{eq:doubling trick single phase}
\sum_{t=t_1}^{t_2} \langle \bm g_t,\bm x_t-\bm x_t^\circ\rangle=\O\left (\sqrt{D \left (D+\sum_{t=t_1}^{t_2-1} \lVert \mathring{\bm x}_t-\mathring{\bm x}_{t+1}\rVert_1\right )}\sqrt{\sum_{t=t_1}^{t_2} \lVert \bm g_t\rVert_\infty^2}\log T\right ).
\end{equation}

As all $[t_1,t_2]$'s form a partition of $[T]$, summing up all \Cref{eq:doubling trick single phase} gives
\begin{equation*}
\sum_{t=1}^T \langle \bm g_t,\bm x_t-\bm x_t^\circ\rangle=\O\left (\sqrt{D \left (D+\sum_{t=1}^{T-1} \lVert \mathring{\bm x}_t-\mathring{\bm x}_{t+1}\rVert_1\right )}\sqrt{\sum_{t=1}^T \lVert \bm g_t\rVert_\infty^2}\log T \log \left (\max_{t=1}^T G_t\right )\right ),
\end{equation*}
which utilizes the fact that at most $\O(\lceil \log_2 (\max_t G_t)\rceil)$ distinct $[t_1,t_2]$'s can occur.
\end{proof}

\subsection{Deciding $\bm a(t)$ via \OLOAlg Algorithm (Proof of \Cref{thm:multi-hop stability decision making})}\label{sec:multi-hop Lyapunov and Online Learning appendix}

\begin{theorem}[Restatement of \Cref{thm:multi-hop stability decision making}; Deciding $\bm a(t)$ via \OLOAlg Algorithm]
For each link $(n,m)\in \mL$, as we did in \SSAlg, we execute an instance of \OLOAlg (\Cref{alg:OLO unbounded}) where $\mathcal X=\triangle(\mN)$, $\bm g_t=C_{n,m}(t) (\bm Q_m(t) - \bm Q_n(t))$, and $G_t=M\lVert \bm Q_m(t)-\bm Q_n(t)\rVert_\infty$. We make their outputs $\bm x_t$ as $\bm a_{n,m}(t)$ for every round $t$. Let $\mu_{n,m}^{(k)}(t)$ be the number of actually transmitted jobs from $Q_n^{(k)}(t)$ to $Q_m^{(k)}(t+1)$ induced by $a_{n,m}^{(k)}(t)$.

Consider an arbitrary reference action sequence $\{\mathring{\bm a} (t)\}_{t\in [T]}$ satisfying \Cref{def:multi-hop stability network stability}. Let $\mathring{\mu}_{n,m}^{(k)} (t)=C_{n,m}(t) \mathring{a}_{n,m}^{(k)} (t)\in [0,M]$ (as $C_{n,m}(t)\in [0,M]$ and $\mathring a_{n,m}^{(k)}(t)\in [0,1]$). Then
\begin{align*}
&\quad \E\left [\sum_{t=1}^T \sum_{(n,m)\in \mL} \sum_{k\in \mN} (\mu_{n,m}^{(k)}(t)-\mathring{\mu}_{n,m}^{(k)}(t))\left (Q_m^{(k)}(t)-Q_n^{(k)}(t)\right )\right ]\\
&=\O\left (M\sqrt{1+P_T^a} \E\left [\sqrt{\sum_{t=1}^T \lVert \bm Q(t)\rVert_2^2} \log T \log \left (\max_{t=1}^T \max_{(n,m)\in \mL} M \lVert \bm Q_m(t) - \bm Q_n(t)\rVert_\infty\right )\right ]\right )\text{,}
\end{align*}
where $P_T^a\triangleq \sum_{t=1}^{T-1} \sum_{(n,m)\in \mL} \lVert \mathring{\bm a}_{n,m}(t) - \mathring{\bm a}_{n,m}(t+1)\rVert_1$ is the path length of $\{\mathring{\bm a}(t)\}_{t=1}^T$.
\end{theorem}
\begin{proof}
According to the definitions of $C_{n,m}^{(k)}(t)$ and $\mathring C_{n,m}^{(k)}(t)$ together with \Cref{lem:OLO unbounded guarantee},
\begin{align*}
&\quad \sum_{t=1}^T \E\left [\sum_{(n,m)\in \mL} \sum_{k\in \mN} (C_{n,m}^{(k)}(t)-\mathring{C}_{n,m}^{(k)}(t))\left (Q_m^{(k)}(t)-Q_n^{(k)}(t)\right )\middle \vert \bm Q(t)\right ] \\
&=\sum_{t=1}^T \sum_{(n,m)\in \mL} \sum_{k\in \mN} C_{n,m}(t)\left (Q_m^{(k)}(t) - Q_n^{(k)}(t)\right ) \left (a_{n,m}^{(k)}(t)-\mathring a_{n,m}^{(k)}(t)\right )\\
&\le \O\Bigg (\sqrt{1+\sum_{t=1}^{T-1} \sum_{(n,m)\in \mL} \lVert \mathring{\bm a}_{n,m}(t) - \mathring{\bm a}_{n,m}(t+1)\rVert_1}\\
&\qquad \sqrt{\sum_{t=1}^T \sum_{(n,m)\in \mL}\sum_{k\in \mN} \big (C_{n,m}(t)(Q_m^{(k)}(t)-Q_n^{(k)}(t))\big )^2}\\
&\qquad \log T \log \left (\max_{t=1}^T \max_{(n,m)\in \mL} M \lVert \bm Q_m(t) - \bm Q_n(t)\rVert_\infty\right )\Bigg ).
\end{align*}

As $C_{n,m}(t)\in [0,M]$, we can upper bound the RHS of the above inequality by
\begin{equation*}
\O\left (\sqrt{1+P_T^a} \sqrt{2M^2 \sum_{t=1}^T \sum_{n\in \mN}\sum_{k\in \mN} Q_n^{(k)}(t)^2}\log T \log \left (\max_{t=1}^T \max_{(n,m)\in \mL} M \lVert \bm Q_m(t) - \bm Q_n(t)\rVert_\infty\right )\right ),
\end{equation*}
where $P_T^a\triangleq \sum_{t=1}^{T-1} \sum_{(n,m)\in \mL} \lVert \mathring{\bm a}_{n,m}(t) - \mathring{\bm a}_{n,m}(t+1)\rVert_1$ is the path length of $\mathring{\bm a}(t)$. Taking expectations gives our conclusion.
\end{proof}

\subsection{Main Theorem for Multi-Hop Network Stability (Proof of \Cref{thm:multi-hop stability main theorem})}\label{sec:multi-hop stability appendix}
\begin{theorem}[Restatement of \Cref{thm:multi-hop stability main theorem}; Main Theorem for Multi-Hop Network Stability]
Suppose that $\{\mathring{\bm a}_{n,m}(t)\in \triangle(\mN)\}_{(n,m)\in \mL,t\in [T]}$ satisfies \Cref{def:multi-hop stability network stability} and its path length satisfies
\begin{equation*}
P_t^a\triangleq \sum_{s=1}^{t-1} \sum_{(n,m)\in \mL} \lVert \mathring{\bm a}_{n,m}(s) - \mathring{\bm a}_{n,m}(s+1)\rVert_1\le C^a t^{1/2 - \delta_a},\quad \forall t=1,2,\ldots,T,
\end{equation*}
where $C^a$ and $\delta_a$ are assumed to be known constants but the precise $P_t^a$ or $\{\mathring{\bm a}_{n,m}(t)\in \triangle(\mN)\}_{(n,m)\in \mL,t\in [T]}$ both remain unknown.
Then, if we execute the \SSAlg framework in \Cref{alg:stability} with \OLOAlg defined in \Cref{alg:OLO unbounded}, the following performance guarantee is enjoyed:
\begin{equation*}
\frac 1T\E\left [\sum_{t=1}^T \lVert \bm Q(t)\rVert_1\right ]=\O\left (\frac{(N^2(2NM+R)^2 + \epsilon_W N^2 (2NM+R))C_W + (N^4 M^2 + N^2 R^2)}{\epsilon_W}\right )+o_T(1).
\end{equation*}
That is, when $T\gg 0$, we have $\frac{1}{T} \E\left [\sum_{t=1}^T \lVert \bm Q(t)\rVert_1\right ]=\O_T(1)$, \textit{i.e.}, \Cref{eq:average queue} holds and the system is stable.
\end{theorem}
\begin{proof}
We first defer some calculations into \Cref{lem:multi-hop stability tedious calc}, which basically combines \Cref{lem:multi-hop stability network stability}, \Cref{lem:multi-hop stability Lyapunov}, and \Cref{thm:multi-hop stability decision making} together. \Cref{lem:multi-hop stability tedious calc} says
\begin{align*}
\E\left [\sum_{t=1}^T \lVert \bm Q(t)\rVert_1\right ]&\le f(T) + g(T) \E\left [\sum_{t=1}^T \lVert \bm Q(t)\rVert_1\right ]^{3/4} \log \E \left [ \sum_{t=1}^T \lVert \bm Q(t)\rVert_1\right ],
\end{align*}
where
\begin{align*}
f(T)&=\epsilon_W^{-1}\O\left ((N^2(2NM+R)^2 + \epsilon_W N^2 (2NM+R))C_W T + (N^4 M^2 + N^2 R^2)T\right ),\\
g(T)&=\epsilon_W^{-1} \O\left (M (2NM+R)^{1/4} \sqrt{1+P_T^a} \log T\right ).
\end{align*}

In \Cref{lem:3/4 with log self-bounding}, we will prove a self-bounding inequality that says, if $y\le f + y^{3/4} g \log y$, then $y\le \left (f^{1/4} + g \log \left (2(f^{1/4} + g)^2\right )\right )^4$. Therefore, we can apply \Cref{lem:3/4 with log self-bounding} to conclude that
\begin{align*}
\E\left [\sum_{t=1}^T \lVert \bm Q(t)\rVert_1\right ]
&= \O\left (f(T) + g(T)^4  \log^4 \left (2(f(T)^{1/4} + g(T))^2\right )\right )\\
&=\epsilon_W^{-1}\O\left ((N^2(2NM+R)^2 + \epsilon_W N^2 (2NM+R))C_W T + (N^4 M^2 + N^2 R^2)T\right ) + \\
&\quad \left (\epsilon_W^{-1} \O\left (M (2NM+R)^{1/4} \sqrt{1+P_T^a} \log T\right ) \Otil_T(1)\right )^4.
\end{align*}

Since
\begin{align*}
\left (\sqrt{1+P_T^a}\right )^4\le \left (\sqrt{T^{1/2 - \delta_a}}\right )^4 = \O_T (T^{1-2\delta_a}),
\end{align*}
we know $\left (\epsilon_W^{-1} \O\left (M (2NM+R)^{1/4} \sqrt{1+P_T^a} \log T\right ) \Otil_T(1)\right )^4=\Otil_T(T^{1-2\delta_a})=o_T(T)$.
The conclusion then follows.
\end{proof}

\begin{lemma}[Calculations when Proving \Cref{thm:multi-hop stability main theorem}]\label{lem:multi-hop stability tedious calc}
Following all the assumptions in \Cref{thm:multi-hop stability main theorem}, we have
\begin{align*}
\E\left [\sum_{t=1}^T \lVert \bm Q(t)\rVert_1\right ]&\le f(T) + g(T) \E\left [\sum_{t=1}^T \lVert \bm Q(t)\rVert_1\right ]^{3/4} \log \E \left [ \sum_{t=1}^T \lVert \bm Q(t)\rVert_1\right ],
\end{align*}
where
\begin{align*}
f(T)&=\epsilon_W^{-1}\O\left ((N^2(2NM+R)^2 + \epsilon_W N^2 (2NM+R))C_W T + (N^4 M^2 + N^2 R^2)T\right ),\\
g(T)&=\epsilon_W^{-1} \O\left (M (2NM+R)^{1/4} \sqrt{1+P_T^a} \log T\right ).
\end{align*}
\end{lemma}
\begin{proof}
Recall the piecewise stability assumption in \Cref{def:multi-hop stability network stability} infers \Cref{lem:multi-hop stability network stability}:
\begin{align*}
&\quad \epsilon_W \E\left [\sum_{t=1}^T \sum_{n\in \mN} \sum_{k\in \mN} Q_n^{(k)}(t)\right ] - (N^2 (2NM+R)^2 + \epsilon_W N^2 (2NM+R)) C_W T \\
&\le-\E\left [\sum_{t=1}^T \sum_{(n,m)\in \mL} \sum_{k\in \mN} C_{n,m}(t) \mathring{a}_{n,m}^{(k)}(t) (Q_m^{(k)}(t) - Q_n^{(k)}(t))\right ]-\E\left [\sum_{t=1}^T \sum_{n\in \mN} \sum_{k\in \mN} Q_n^{(k)}(t) \lambda_n^{(k)}(t)\right ].
\end{align*}

From the non-negativity of Lyapunov drifts in \Cref{lem:multi-hop stability Lyapunov}, we know
\begin{align*}
0&\le \E\left [\sum_{t=1}^T \sum_{(n,m)\in \mL} \sum_{k\in \mN} \mu_{n,m}^{(k)}(t)\left (Q_m^{(k)}(t)-Q_n^{(k)}(t)\right ) + \sum_{n\in \mN} \sum_{k\in \mN} Q_n^{(k)}(t) \lambda_n^{(k)}(t)\right ] + \nonumber \\
&\quad \frac 12 N^2 ((NM)^2+2(NM)^2+2R^2) T.
\end{align*}

Furthermore, recall the guarantee of \OLOAlg algorithm in \Cref{thm:multi-hop stability decision making} that
\begin{align*}
&\quad \E\left [\sum_{t=1}^T \sum_{(n,m)\in \mL} \sum_{k\in \mN} (\mu_{n,m}^{(k)}(t)-\mathring{\mu}_{n,m}^{(k)}(t))\left (Q_m^{(k)}(t)-Q_n^{(k)}(t)\right )\right ]\\
&=\O\left (M\sqrt{1+P_T^a} \E\left [\sqrt{\sum_{t=1}^T \lVert \bm Q(t)\rVert_2^2} \log T \log \left (\max_{t=1}^T \max_{(n,m)\in \mL} M \lVert \bm Q_m(t) - \bm Q_n(t)\rVert_\infty\right )\right ]\right )\text{,}
\end{align*}

Therefore, we are able to get
\begin{align*}
&\quad \epsilon_W \E\left [\sum_{t=1}^T \sum_{n\in \mN} \sum_{k\in \mN} Q_n^{(k)}(t)\right ] - (N^2 (2NM+R)^2 + \epsilon_W N^2 (2NM+R)) C_W T\\
&\le \O\left (M\sqrt{1+P_T^a} \E\left [\sqrt{\sum_{t=1}^T \lVert \bm Q(t)\rVert_2^2} \log T \log \left (\max_{t=1}^T \max_{(n,m)\in \mL} M \lVert \bm Q_m(t) - \bm Q_n(t)\rVert_\infty\right )\right ]\right ) + \\
&\quad \frac 12 N^2 ((NM)^2+2(NM)^2+2R^2) T.
\end{align*}

\Cref{lem:2 upper bound} states that if $x_1=0$, $x_2,\ldots,x_T \ge 0$, and $\lvert x_{t+1}-x_t\rvert\le 1$, then $\sum_{t=1}^T x_t^2 = \O\left ((\sum_{t=1}^T x_t)^{3/2}\right )$. From \Cref{lem:queue length increment}, any single queue $Q_n^{(k)}(t)$ satisfies $\lvert Q_n^{(k)}(t+1)-Q_n^{(k)}(t)\rvert \le (2NM + R)$.
Hence, applying \Cref{lem:2 upper bound} to $\{Q_n^{(k)}(t) / (2NM+R)\}_{t\in [T]}$ to every $n\in \mN$ and $k\in \mN$, we have
\begin{equation*}
\sum_{t=1}^T \lVert \bm Q(t)\rVert_2^2 = (2NM+R)^2 \sum_{t=1}^T \sum_{n\in \mN} \sum_{k\in \mN} \left (\frac{Q_n^{(k)}(t)}{2NM+R}\right )^2 =\O\left (\sqrt{2NM+R} \left (\sum_{t=1}^T \lVert \bm Q(t)\rVert_1\right )^{1.5}\right ).
\end{equation*}

Further noticing that
\begin{equation*}
\max_{t=1}^T \max_{(n,m)\in \mL} M \lVert \bm Q_m(t) - \bm Q_n(t)\rVert_\infty\le \sum_{t=1}^T M \sum_{n\in \mN} \lVert \bm Q_n(t)\rVert_1\le M \sum_{t=1}^T \lVert \bm Q(t)\rVert_1,
\end{equation*}
the above inequality becomes
\begin{align*}
\epsilon_W \E\left [\sum_{t=1}^T \lVert \bm Q(t)\rVert_1\right ]&\le \O\left (M (2NM+R)^{1/4} \sqrt{1+P_T^a} \E\left [\left (\sum_{t=1}^T \lVert \bm Q(t)\rVert_1\right )^{3/4} \log T \log \left (M \sum_{t=1}^T \lVert \bm Q(t)\rVert_1\right )\right ]\right )+\\
&\quad \O\left ((N^2(2NM+R)^2 + \epsilon_W N^2 (2NM+R))C_W T + (N^4 M^2 + N^2 R^2)T\right ).
\end{align*}

Noticing that $x\mapsto x^{3/4} \log (M x)$ is concave when $x$ is large enough, Jensen inequality then gives
\begin{equation*}
\O\left (\E\left [\left (\sum_{t=1}^T \lVert \bm Q(t)\rVert_1\right )^{3/4} \log \left (M \sum_{t=1}^T \lVert \bm Q(t)\rVert_1\right ) \right ]\right )= \O\left (\E\left [\sum_{t=1}^T \lVert \bm Q(t)\rVert_1\right ]^{3/4} \log \E \left [ \sum_{t=1}^T \lVert \bm Q(t)\rVert_1\right ]\right ).
\end{equation*}

Therefore, if we define the auxiliary functions $f(T)$ and $g(T)$ as
\begin{align*}
f(T)&=\epsilon_W^{-1}\O\left ((N^2(2NM+R)^2 + \epsilon_W N^2 (2NM+R))C_W T + (N^4 M^2 + N^2 R^2)T\right ),\\
g(T)&=\epsilon_W^{-1} \O\left (M (2NM+R)^{1/4} \sqrt{1+P_T^a} \log T\right ),
\end{align*}
we are able to conclude that
\begin{align*}
\E\left [\sum_{t=1}^T \lVert \bm Q(t)\rVert_1\right ]&\le f(T) + g(T) \E\left [\sum_{t=1}^T \lVert \bm Q(t)\rVert_1\right ]^{3/4} \log \E \left [ \sum_{t=1}^T \lVert \bm Q(t)\rVert_1\right ],
\end{align*}
as claimed.
\end{proof}

\section{Omitted Proofs for Multi-Hop Utility Maximization Tasks}
\subsection{Reference Policy Assumption (Proof of \Cref{lem:multi-hop utility network stability})}\label{sec:multi-hop utility network stability appendix}
\begin{lemma}[Restatement of \Cref{lem:multi-hop utility network stability}; Ability of $\{(\mathring{\bm a}(t),\mathring{\bm \lambda}(t))\}_{t\in [T]}$ in Stabilizing the Network]
If $\{(\mathring{\bm a}(t),\mathring{\bm \lambda}(t))\}_{t\in [T]}$ satisfies \Cref{def:multi-hop utility network stability}, then for any scheduler-generated queue lengths $\{\bm Q(t)\}_{t\in [T]}$,
\begin{align*}
&\quad \epsilon_W \E\left [\sum_{t=1}^T \sum_{n\in \mN} \sum_{k\in \mN} Q_n^{(k)}(t)\right ] - (N^2 (2NM+R)^2 + \epsilon_W N^2 (2NM+R)) C_W T \nonumber\\
&\le-\E\left [\sum_{t=1}^T \sum_{(n,m)\in \mL}\sum_{k\in \mN} \mathring \mu_{n,m}^{(k)}(t) (Q_m^{(k)}(t) - Q_n^{(k)}(t))\right ] - \E\left [\sum_{t=1}^T \sum_{n\in \mN} \sum_{k\in \mN} Q_n^{(k)}(t) \mathring \lambda_n^{(k)}(t)\right ].
\end{align*}
\end{lemma}
\begin{proof}
The proof of this lemma is identical to that of \Cref{lem:multi-hop stability network stability}, except for replacing the environment-generated $\bm \lambda(t)$ with the $\mathring{\bm \lambda}(t)$ generated by the reference policy. For more details, please refer to \Cref{sec:multi-hop stability network stability appendix}.
\end{proof}

\subsection{Lyapunov Drift-Plus-Penalty Analysis (Proof of \Cref{lem:multi-hop utility Lyapunov})}\label{sec:multi-hop utility Lyapunov appendix}
\begin{lemma}[Lyapunov Drift-Plus-Penalty Analysis]\label{lem:multi-hop utility Lyapunov}
Under the queue dynamics of \Cref{eq:multi-hop stability dynamics},
\begin{align*}
&\quad -\E\left [\sum_{t=1}^T \sum_{(n,m)\in \mL} \sum_{k\in \mN} C_{n,m}(t) (Q_m^{(k)}(t)-Q_n^{(k)}(t))a_{n,m}^{(k)}(t)\right ] \nonumber\\
&\quad -\E\left [\sum_{t=1}^T \sum_{n\in \mN} \sum_{k\in \mN} Q_n^{(k)}(t) \lambda_n^{(k)}(t)\right ]+V \E\left [\sum_{t=1}^T (g_t(\bm \lambda(t))-g_t(\mathring{\bm \lambda}(t)))\right ] \nonumber\\
&\le \frac 12 N^2 ((NM)^2+2(NM)^2+2R^2)T +V  \E\left [\sum_{t=1}^T (g_t(\bm \lambda(t))-g_t(\mathring{\bm \lambda}(t)))\right ].
\end{align*}
\end{lemma}
\begin{proof}
The proof follows from applying \Cref{lem:multi-hop stability Lyapunov} and adding
\begin{equation*}
V  \E\left [\sum_{t=1}^T (g_t(\bm \lambda(t))-g_t(\mathring{\bm \lambda}(t)))\right ]
\end{equation*}
to both sides.
\end{proof}

\subsection{Guarantee of \BCOAlg Algorithm (Proof of \Cref{lem:ABGD guarantee})}\label{sec:ABGD guarantee}
\begin{lemma}[Restatement of \Cref{lem:ABGD guarantee}; Guarantee of \BCOAlg Algorithm]
Suppose that $r\mathbb B\subseteq \mathcal X\subseteq R\mathbb B$, the $t$-th loss $\ell_t$ is bounded by $C_t$ and is $L_t$-Lipschitz.
Suppose that $\eta_t$ and $\delta_t$ are both $\mathcal F_{t-1}$-measurable (where $(\mathcal F_t)_{t=0}^T$ is the natural filtration), $\eta_1>\eta_2>\cdots>\eta_T$, and $\alpha_t\triangleq \delta_t / r<1$ \textit{a.s.} for all $t\in [T]$. Then for any fixed $\bm u_1,\bm u_2,\ldots,\bm u_T\in \mathcal X$, the \BCOAlg algorithm in \Cref{alg:ABGD} enjoys the following guarantee:
\begin{align*}
&\quad \text{D-Regret}_T^{\text{BCO}}(\bm u_1,\bm u_2,\ldots,\bm u_T)=\E\left [\sum_{t=1}^T (\ell_t(\bm x_t)-\ell_t(\bm u_t))\right ]\\
&\le \E\left [\frac{7R^2}{4\eta_T}+\frac{P_T R}{\eta_T}+\sum_{t=1}^T \left (\frac{\eta_t}{2} \frac{d^2}{\delta_t^2} C_t^2 + 3L_t \delta_t + L_t \alpha_t R\right )\right ],
\end{align*}
where $P_T=\sum_{t=1}^{T-1}\lVert \bm u_t-\bm u_{t+1}\rVert$ is the path length of the comparator sequence $\{\bm u_t\}_{t\in [T]}$.
\end{lemma}
\begin{proof}
Similar to the proof by \citet[Theorem 1]{zhao2021bandit}, let $\hat \ell_t(\bm x)=\E_{\bm v\in \mathbb B}[\ell_t(\bm x+\delta \bm v)]$ (where $\mathbb B=\{\bm x\in \mathbb R^d\mid \lVert \bm x\rVert\le 1\}$ is the unit ball in $\mathbb R^d$) and $\bm v_t=(1-\alpha_t)\bm u_t\in (1-\alpha_t) \mathcal X$, then
\begin{equation*}
\sum_{t=1}^T (\ell_t(\bm x_t)-\ell_t(\bm u_t))=\sum_{t=1}^T (\hat \ell_t(\bm y_t)-\hat \ell_t(\bm v_t))+\sum_{t=1}^T (\ell_t(\bm x_t) - \hat \ell_t(\bm y_t)) + \sum_{t=1}^T (\hat \ell_t(\bm v_t) - \ell_t(\bm u_t)).
\end{equation*}

According to the original proof, the expectation of the latter two terms are controlled by $\E[\sum_{t=1}^T 2L \delta_t]$ and $\E[\sum_{t=1}^T (L\delta_t + L \alpha_t R)]$, respectively.
For the first term, according to \citet[Lemma 2.1]{flaxman2005online}, $\E_{\bm s_t}[\frac d\delta \ell_t(\bm x_t) \bm s_t]=\nabla \hat \ell_t(\bm x_t)$.
Therefore, since $\lVert \frac d{\delta_t} \ell_t(\bm x_t) \bm s_t\rVert\le \frac d{\delta_t} C_t$, we get from \Cref{lem:SGD with adaptive LR} that
\begin{equation*}
\E\left [\sum_{t=1}^T (\hat \ell_t(\bm x_t)-\hat \ell_t(\bm v_t))\right ]\le \E\left [\frac{7R^2}{4\eta_T}+\frac{P_T R}{\eta_T}+\sum_{t=1}^T \frac{\eta_t}{2} \frac{d^2}{\delta_t^2} C_t^2\right ],
\end{equation*}
where $\hat P_T=\sum_{t=1}^{T-1}\lVert \bm v_t-\bm v_{t+1}\rVert$, the path length of $\{\bm v_t\}_{t\in [T]}$, satisfies $\hat P_T\le P_T$. This gives our conclusion.
\end{proof}

\begin{lemma}[Guarantee of Projected SGD]\label{lem:SGD with adaptive LR}
Suppose that $\mathcal X$ is bounded by $[r,R]$, the $t$-th loss function $\ell_t$ is bounded by $[-C_t,C_t]$ and is $L$-Lipschitz. Further suppose that a stochastic gradient $\bm g_t$ can be calculated in round $t$ such that $\E[\bm g_t\mid \bm x_1,\ell_1,\ldots,\bm x_t,\ell_t]=\nabla \ell_t(\bm x_t)$ and $\lVert \bm g_t\rVert_2\le C_t$. The the iteration
\begin{equation*}
\bm x_{t+1}=\text{Proj}_{\mathcal X}\left [\bm x_t-\eta_t \bm g_t\right ]
\end{equation*}
ensures the following dynamic regret guarantee for any fixed $\bm u_1,\bm u_2,\ldots,\bm u_T\in \mathcal X$:
\begin{equation*}
\sum_{t=1}^T (\ell_t(\bm x_t)-\ell_t(\bm u_t))\le \frac{7R^2}{4\eta_T}+\frac{P_T R}{\eta_T}+\sum_{t=1}^T \frac{\eta_t}{2} C_t^2,
\end{equation*}
where $P_T=\sum_{t=1}^{T-1}\lVert \bm u_t-\bm u_{t+1}\rVert$ is the path length of $\{\bm u_t\}_{t=1}^T$.
\end{lemma}
\begin{proof}
We first consider the full-feedback model where the whole $\ell_t$, instead of the single-entry $\ell_t(\bm x_t)$, is available.
Then the Gradient Descent algorithm $\bm x_{t+1}=\text{Proj}_{\mathcal X}[\bm x_t-\eta_t \nabla \ell_t(\bm x_t)]$ enjoys the following dynamic regret guarantee (which follows the proof of \citet[Theorem 2]{zinkevich2003online}):
\begin{align}
&\quad \text{D-Regret}_T(\bm u_1,\bm u_1,\ldots,\bm u_T)=\sum_{t=1}^T (\ell_t(\bm x_t)-\ell_t(\bm u_t)) \nonumber\\
&\le \sum_{t=1}^T \left (\frac{1}{2\eta_t} \left (\lVert \bm x_t-\bm u_t\rVert^2-\lVert \bm x_{t+1}-\bm u_t\rVert^2\right )+\frac{\eta_t}{2}\lVert \nabla \ell_t(\bm x_t)\rVert^2\right ) \nonumber\\
&= \sum_{t=1}^T \frac{\lVert \bm x_t\rVert^2-\lVert \bm x_{t+1}\rVert^2}{2\eta_t}+\sum_{t=1}^T \frac{\langle \bm x_{t+1}-\bm x_t,\bm u_t\rangle}{\eta_t}+\sum_{t=1}^T \frac{\eta_t}{2} \lVert \nabla \ell_t(x_t)\rVert^2 \nonumber\\
&\overset{(a)}{\le}\frac{\lVert \bm x_1\rVert^2-\lVert \bm x_{T+1}\rVert^2}{2\eta_T}+\frac{\langle \bm x_{T+1},\bm u_T\rangle}{\eta_T}-\frac{\langle \bm x_{1},\bm u_1\rangle}{\eta_1}+\sum_{t=2}^T \frac{\langle \bm u_{t-1}-\bm u_t,\bm x_t\rangle}{\eta_t}+\sum_{t=1}^T \frac{\eta_t}{2} \lVert \nabla \ell_t(x_t)\rVert^2 \nonumber\\
&\le \frac{7R^2}{4\eta_T}+\frac{P_T R}{\eta_T}+\sum_{t=1}^T \frac{\eta_t}{2} \lVert \nabla \ell_t(x_t)\rVert^2,\label{eq:full information OGD D-Regret}
\end{align}
where (a) uses the property that $\eta_1\ge \eta_2\ge \cdots \ge \eta_T$.

Moving back to the bandit-feedback model, following the proof of \citet[Theorem 8]{zhao2021bandit}, we can consider a loss function defined as $\tilde \ell_t(x)\triangleq \ell_t(x)+\langle \bm x,\bm g_t-\nabla \ell_t(\bm x_t)\rangle$.
As $\nabla \tilde \ell_t(\bm x_t)=\bm g_t$ and $\E[\tilde \ell_t(\bm x)]=\E[\ell_t(\bm x)]$ (which is due to the fact that $\E[\bm g_t]=\nabla \ell_t(\bm x_t)$), applying \Cref{eq:full information OGD D-Regret} gives
\begin{align*}
&\quad \sum_{t=1}^T (\ell_t(\bm x_t)-\ell_t(\bm u_t))=\E\left [\sum_{t=1}^T (\tilde \ell_t(\bm x_t)-\tilde \ell_t(\bm u_t))\right ]\\
&\le \E\left [\frac{7R^2}{4\eta_T}+\frac{P_T R}{\eta_T}+\sum_{t=1}^T \frac{\eta_t}{2} \lVert \nabla \tilde \ell_t(\bm x_t)\rVert^2\right ]\overset{(b)}{\le} \frac{7R^2}{4\eta_T}+\frac{P_T R}{\eta_T}+\sum_{t=1}^T \frac{\eta_t}{2} C_t^2,
\end{align*}
where the expectation is taken \textit{w.r.t.} the randomness in the stochastic gradient $\{\bm g_t\}_{t\in [T]}$, and (b) makes use of the fact that $\nabla \tilde \ell_t(\bm x_t)=\bm g_t$.
\end{proof}

\subsection{Deciding $\bm \lambda(t)$ via \BCOAlg Algorithm (Proof of \Cref{thm:multi-hop utility decision making})}\label{sec:general bandit part appendix}
\begin{theorem}[Restatement of \Cref{thm:multi-hop utility decision making}; Deciding $\bm \lambda(t)$ via \BCOAlg Algorithm]
For the reference arrival rates $\{\mathring{\bm \lambda}(t)\}_{t\in [T]}$ defined in \Cref{def:multi-hop utility network stability}, suppose that its path length ensures
\begin{equation*}
P_t^\lambda\triangleq \sum_{t=1}^{T-1} \lVert \mathring{\bm \lambda}({t+1}) - \mathring{\bm \lambda}(t)) \rVert_1\le C^\lambda t^{1/2 - \delta_\lambda},\quad \forall t=1,2,\ldots,T,
\end{equation*}
where, similar to \Cref{thm:multi-hop stability main theorem}, $C^\lambda$ and $\delta_\lambda$ are assumed to be known constants but the precise $P_t^\lambda$ or $\{\mathring{\bm \lambda}(t)\}_{t\in [T]}$ both remain unknown.
Suppose that the action set $\Lambda$ is bounded by $[r,R]$ (\textit{i.e.}, $r\mathbb B\subseteq \Lambda\subseteq R\mathbb B$).
If we execute \BCOAlg (\Cref{alg:ABGD}) over $\Lambda$ with loss functions $\ell_t(\bm \lambda)=\langle \bm Q(t),\bm \lambda\rangle-V g_t(\bm \lambda)$ and parameters $\eta_t,\delta_t,\alpha_t$ defined as
\begin{align}
\eta_t&=\left (C^\lambda T^{1/2 - \delta_\lambda}\middle / \substack{\left (C^\lambda T^{1/2 - \delta_\lambda}\right )^{7/3} \left (4r^{-3} d^2 \right )^{28/9} \left (M+R\right )^{4/3}+\\
C^\lambda T^{1/2-\delta_\lambda} (r^{-3} d^2 VG^2 / L)^{4/3}+\\
\sum_{s=1}^t \left ((\lVert \bm q_s\rVert_\infty + VG)^2 (\lVert \bm q_s\rVert_2 + VL)^2\right )^{1/3}}\right )^{3/4}, \nonumber\\
\delta_t &= \left (\eta_t d^2 \frac{(\lVert \bm Q(t)\rVert_\infty + VG)^2}{(\lVert \bm Q(t)\rVert_2 + VL)}\right )^{1/3},\quad \alpha_t = \frac{\delta_t}{r},\label{eq:true learning rate in BCO appendix}
\end{align}
its outputs $\bm \lambda(1),\bm \lambda(2),\ldots,\bm \lambda(T) \in \Lambda$ ensure
\begin{align*}
&\scalemath{0.95}{\quad \E\left [\sum_{t=1}^T \left ((\langle \bm Q(t),\bm \lambda(t)\rangle - Vg_t(\bm \lambda(t))) - (\langle \bm Q(t),\mathring{\bm \lambda}(t)\rangle - Vg_t(\mathring{\bm \lambda}(t)))\right )\right ]}\\
&\scalemath{0.95}{=\O\left (\frac{R(2NM+R)}{r^7} d^{14/3} (C^\lambda T^{1/2-\delta_\lambda})^2\right )+\O\left (\E\left [\left (\frac{R}{r} d^{2/3}+R\right ) (C^\lambda T^{1/2 - \delta_\lambda})^{1/4} \left (\sum_{t=1}^T \left (\lVert \bm Q(t)\rVert_2 + V(L+G)\right )^{4/3}\right )^{3/4}\right ]\right ).}
\end{align*}
\end{theorem}
\begin{proof}
For loss function $\ell_t(\bm \lambda)=\langle \bm Q(t),\bm \lambda\rangle - Vg_t(\bm \lambda)$, it is bounded by $C_t\triangleq \lVert \bm Q(t)\rVert_\infty + V G$ and is $L_t\triangleq (\lVert \bm Q(t)\rVert_2 + VL)$-Lipschitz. As $\bm Q(t)$ is revealed after the $(t-1)$-th round, $G_t$ and $L_t$ are both $\mathcal F_{t-1}$-measurable.
As sketched in the main text, we first regard $\eta_t$ as a constant and tune $\delta_t$ to minimize the summation term in \Cref{lem:ABGD guarantee}.

Let $\delta_t=(\eta_t d^2 C_t^2 / L_t)^{1/3}$ and $\alpha_t=\delta_t / r$. Suppose that all conditions in \Cref{lem:ABGD guarantee} hold, then
\begin{align}
&\quad \E\left [\sum_{t=1}^T (\ell_t(\bm \lambda_t)-\ell_t(\bm \lambda_t^\ast))\right ]\le \E\left [\frac{7R^2 + 4 P_T R}{4\eta_T} + \sum_{t=1}^T \left (\eta_t \frac{d^2}{\delta_t^2} C_t^2 + 3L_t\delta_t + L_t \alpha_t R\right )\right ] \nonumber\\
&=\O\left (\E\left [\frac{R^2 + C^r T^{1/2 - \delta_r} R}{\eta_T}+\sum_{t=1}^T \left (\eta_t d^2 (\lVert \bm Q(t)\rVert_\infty + VG)^2 (\lVert \bm Q(t)\rVert_2 + VL)^2\right )^{1/3} \frac{R}{r}\right ]\right ). \label{eq:BCO regret decomposition}
\end{align}

We first only keep the last term in \Cref{eq:true learning rate in BCO appendix}, \textit{i.e.}, let
\begin{equation}\label{eq:learning rate in BCO}
\eta_t = \left (C^r T^{1/2 - \delta_r}\middle / \sum_{s=1}^t \left ((\lVert \bm q_s\rVert_\infty + VG)^2 (\lVert \bm q_s\rVert_2 + VL)^2\right )^{1/3}\right )^{3/4},
\end{equation}
then we have $\eta_1>\eta_2>\cdots >\eta_T$. Let's first pretend the other condition of $\alpha_t<1$ from \Cref{lem:ABGD guarantee} also holds at this moment.
\Cref{lem:3/4 summation lemma} reveals that if $x_1,x_2,\ldots,x_T\ge 0$, then $\left .\sum_{t=1}^T x_t \middle / (\sum_{s\le t} x_s)^{1/4} \right . =\O\left (\left (\sum_{t=1}^T x_t\right )^{3/4}\right )$. Plugging in $x_t=\left (\lVert \bm q_s\rVert_\infty + VG)^2 (\lVert \bm q_s\rVert_2 + VL)^2\right )^{1/3}$,
\begin{align*}
&\quad \E\left [\sum_{t=1}^T (\ell_t(\bm \lambda_t)-\ell_t(\bm \lambda_t^\ast))\right ] \\
&=\O\left (\E\left [(C^r T^{1/2 - \delta_r})^{1/4} R \left (\sum_{t=1}^T \left ((\lVert \bm Q(t)\rVert_\infty + VG)^2 (\lVert \bm Q(t)\rVert_2 + VL)^2\right )^{1/3}\right )^{3/4} \right ]\right . + \\
&\qquad \left .\E\left [\frac{R}{r} \left (C^r T^{1/2 - \delta_r}\right )^{1/4} d^{2/3} \left (\sum_{t=1}^T \left ((\lVert \bm Q(t)\rVert_\infty + VG)^2 (\lVert \bm Q(t)\rVert_2 + VL)^2\right )^{1/3}\right )^{3/4} \right ] \right )\\
&=\O\left (\E\left [\left (\frac{R}{r} d^{2/3}+R\right ) (C^r T^{1/2 - \delta_r})^{1/4} \left (\sum_{t=1}^T \left (\lVert \bm Q(t)\rVert_2 + V(L+G)\right )^{4/3}\right )^{3/4}\right ]\right ),
\end{align*}
where the last step utilizes $\lVert \bm Q(t)\rVert_\infty \le \lVert \bm Q(t)\rVert_2$.

This almost recovers our conclusion, so it only remains to ensure $\alpha_t=\delta_t/r<1$, which is equivalent to $\eta_t d^2 C_t^2 / L_t < r^3$, \textit{i.e.},
\begin{equation*}
\eta_t^{-1} > r^{-3} d^2 \frac{(\lVert \bm Q(t)\rVert_\infty + VG)^2}{\lVert \bm Q(t)\rVert_2 + VL}.
\end{equation*}

Consider adding a term $X$ into the denominator of \Cref{eq:learning rate in BCO}.
As $\lVert \bm Q(t)\rVert_\infty \le \lVert \bm Q(t)\rVert_2$, $(\lVert \bm Q(t)\rVert_\infty + VG)^2 / (\lVert \bm Q(t)\rVert_2 + VL)\le 2 (\lVert \bm Q(t)\rVert_\infty^2 / \lVert \bm Q(t)\rVert_2) + 2 ((VG)^2 / (VL))\le 2 \lVert \bm Q(t)\rVert_\infty + 2 VG^2 / L$. So we only need to show
\begin{align*}
\frac{\left (X + \sum_{s=1}^t \left ((\lVert \bm q_s\rVert_\infty + VG)^2 (\lVert \bm q_s\rVert_2 + VL)^2\right )^{1/3}\right )^{3/4}}{(C^r T^{1/2 - \delta_r})^{3/4}} > 2 r^{-3} d^2 \left (\lVert \bm Q(t)\rVert_\infty + V \frac{G^2}{L}\right ).
\end{align*}

We decompose $X$ into $X_1$ and $X_2$ and use them to cancel the two terms on the RHS, respectively. That is, we want
\begin{align*}
X_1 + \sum_{s=1}^t \left (\lVert \bm q_s\rVert_\infty^2 \lVert \bm q_s\rVert_2^2\right )^{1/3} & > C^r T^{1/2 - \delta_r} \left (r^{-3} d^2 \lVert \bm Q(t)\rVert_\infty\right )^{4/3},\\
X_2 + t^{3/4} (V^4 G^2 L^2)^{1/4} & > C^r T^{1/2 - \delta_r} \left (r^{-3} d^2 V G^2 / L\right )^{4/3}.
\end{align*}

We first craft $X_1$.
In \Cref{lem:4/3 lower bound}, we show that if $x_1=0$, $x_2,\ldots,x_T \ge 0$, and $\lvert x_{t+1}-x_t\rvert\le 1$, then $\sum_{t=1}^T x_t^{4/3} \ge (x_T/4)^{7/3}$. Thus, using the fact that $\lVert \bm Q(t)\rVert_\infty\le \lVert \bm Q(t)\rVert_2$ and the queue length increment bound $\lvert \lVert \bm Q(t+1)\rVert_\infty - \lVert \bm Q(t)\rVert_\infty \rvert\le (2NM+R)$ (\Cref{lem:queue length increment}), we can let $x_t=\lVert \bm Q(t)\rVert_\infty / (2NM+R)$ and lower bound the LHS by
\begin{equation*}
X_1 + \sum_{s=1}^t \left (\lVert \bm q_s\rVert_\infty^2 \lVert \bm q_s\rVert_2^2\right )^{1/3}\ge X_1+\frac{\left (\lVert \bm Q(t)\rVert_\infty / 4 \right )^{7/3}}{(2NM+R)}\ge X_1^{3/7} \left (\frac{\left (\lVert \bm Q(t)\rVert_\infty / 4 \right )^{7/3}}{(2NM+R)}\right )^{4/7},
\end{equation*}
where the inequality results from AM-GM inequality $\frac{ax + by}{a+b} \ge \sqrt[a+b]{x^a y^b}$. Therefore, we only need to ensure $X_1^{3/7}\ge C^r T^{1/2-\delta_r} (4 r^{-3} d^2)^{4/3} (2NM+R)^{4/7}$. Setting $X_1$ as following then suffices.
\begin{equation*}
X_1=\left (C^r T^{1/2 - \delta_r}\right )^{7/3} \left (4r^{-3} d^2 \right )^{28/9} \left ((2NM+R)\right )^{4/3}.
\end{equation*}

For $X_2$, we only need to set $X_2 = C^r T^{1/2-\delta_r} (r^{-3} d^2 VG^2 / L)^{4/3}$.
Plugging back $X=X_1+X_2$, we get the learning rate scheduling defined in \Cref{eq:true learning rate in BCO appendix}.
Now we verify that the two terms on the RHS of \Cref{eq:BCO regret decomposition} does not increase too much due to $X$. For the first term,
\begin{align*}
&\quad \E\left [\frac{R^2 + C^r T^{1/2 - \delta_r} R}{\eta_T}\right ]\\
&=\left [\frac{R^2 + C^r T^{1/2 - \delta_r} R}{(C^r T^{1/2 - \delta_r})^{3/4}}\left (X + \sum_{t=1}^T \left ((\lVert \bm Q(t)\rVert_\infty + VG)^2 (\lVert \bm Q(t)\rVert_2 + VL)^2\right )^{1/3}\right )^{3/4}\right ]\\
&=\O\left (\E\left [\left (C^r T^{1/2 - \delta_r} R\right ) \left (\frac{X}{C^r T^{1/2 - \delta_r}}\right )^{3/4}+(C^r T^{1/2 - \delta_r})^{1/4} R \left (\sum_{t=1}^T \left (\lVert \bm Q(t)\rVert_2 + V(L+G)\right )^{4/3} \right )^{3/4}\right ]\right )\\
&=\O \left (\E\left [\frac{R(2NM+R)}{r^7} d^{14/3} (C^r T^{1/2-\delta_r})^2 + R (C^r T^{1/2 - \delta_r})^{1/4} \left (\sum_{s=1}^t \left (\lVert \bm q_s\rVert_2 + V(L+G)\right )^{4/3} \right )^{3/4}\right ]\right ).
\end{align*}

For the second term, as $\eta_t$ is strictly smaller than that in \Cref{eq:learning rate in BCO}, we can again apply \Cref{lem:3/4 summation lemma} (if $x_1,x_2,\ldots,x_T\ge 0$, then $\left .\sum_{t=1}^T x_t \middle / (\sum_{s\le t} x_s)^{1/4} \right . =\O\left (\left (\sum_{t=1}^T x_t\right )^{3/4}\right )$) to conclude
\begin{align*}
&\quad \E\left [\sum_{t=1}^T \left (\eta_t d^2 (\lVert \bm Q(t)\rVert_\infty + VG)^2 (\lVert \bm Q(t)\rVert_2 + VL)^2\right )^{1/3} \frac{R}{r}\right ]\\
&=\E\left [\frac{R}{r} d^{2/3} (C^r T^{1/2 - \delta_r})^{1/4} \left (\sum_{t=1}^T \left (\lVert \bm Q(t)\rVert_2 + V(L+G)\right )^{4/3}\right )^{3/4}\right ].
\end{align*}

Summing up two parts gives our conclusion.
\end{proof}

\subsection{Main Theorem for Multi-Hop Utility Maximization (Proof of \Cref{thm:multi-hop utility main theorem})}\label{sec:multi-hop utility appendix}
\begin{theorem}[Restatement of \Cref{thm:multi-hop utility main theorem}; Main Theorem for Multi-Hop Utility Maximization]
Suppose that the feasible set of arrival rates vector $\Lambda$ is bounded by $[r,R]$.
Assume all (unknown) utility functions $g_t$ to be concave, $L$-Lipschitz, and $[-G,G]$-bounded.
Consider a reference action sequence $\{(\mathring{\bm a}(t),\mathring{\bm \lambda}(t))\}_{t\in [T]}$ satisfying \Cref{def:multi-hop utility network stability}, such that their path lengths satisfy
\begin{align*}
P_t^a\triangleq \sum_{s=1}^{t-1} \lVert \mathring{\bm a}(s)-\mathring{\bm a}(s+1)\rVert_1\le C^a t^{1/2-\delta_a},~P_t^\lambda\triangleq \sum_{s=1}^{t-1} \lVert \mathring{\bm \lambda}(s)-\mathring{\bm \lambda}(s+1)\rVert_1\le C^\lambda t^{1/2 - \delta_\lambda},\quad \forall t\in [T].
\end{align*}
Here, $M,R,r,L,G,C^a,\delta_a,C^\lambda,\delta_\lambda$ are assumed to be known constants, whereas the specific $\{(\mathring{\bm a}(t),\mathring{\bm \lambda}(t))\}_{t\in [T]}$ remains unknown.
If we execute the \UMAlg framework in \Cref{alg:utility} with the $\OLO$ sub-rountine given in \Cref{alg:OLO unbounded} and the $\BCO$ sub-routine given in \Cref{alg:ABGD}, when $T$ is large enough such that $V=o_T(\min\{T^{2\delta_a/3},T^{2\delta_\lambda/7}\})$, the following inequalities hold simultaneously:
\begin{align*}
&\scalemath{0.95}{\frac 1T \E\left [\sum_{t=1}^T \lVert \bm Q(t)\rVert_1\right ]=\O\left (\frac{(N^2 (2NM+R)^2 + \epsilon_W N^2 (2NM+R)) C_W + (N^4 M^2+N^2 R^2)}{\epsilon_W}\right ) + o_T(1),}\\
&\scalemath{0.95}{\frac 1T \E\left [\sum_{t=1}^T \left (g_{t}(\mathring{\bm \lambda}(t)) - g_{t}(\bm \lambda(t))\right )\right ]=\O\left (\frac{(N^2 (2NM+R)^2 + \epsilon_W N^2 (2NM+R)) C_W + (N^4 M^2+N^2 R^2)}{V}\right ) + o_T(V^{-1}).}
\end{align*}
That is, when $T\gg 0$, our algorithm not only stabilizes the system so that $\frac 1T \E\left [\sum_{t=1}^T \lVert \bm Q(t)\rVert_1\right ]=\O_T(1)$, but also enjoys an average utility approaching that of the reference policy polynomially fast, \textit{i.e.}, $\frac 1T \E\left [\sum_{t=1}^T \left (g_{t}(\mathring{\bm \lambda}(t)) - g_{t}(\bm \lambda(t))\right )\right ]=\O_T(V^{-1})$ -- the utility maximization objective \Cref{eq:average reward} is ensured.
\end{theorem}
\begin{proof}
As sketched in the main text, the first step is to \textit{i)} combine algorithmic guarantees for \OLOAlg (\Cref{thm:multi-hop stability decision making}) and \BCOAlg (\Cref{thm:multi-hop utility decision making}), \textit{ii)} plug in the network stability assumption \Cref{lem:multi-hop utility network stability}, and \textit{iii)} make use of the Lyapunov DPP analysis in \Cref{lem:multi-hop utility Lyapunov}. Deferring these calculations to \Cref{lem:multi-hop utility tedious calc}, we can get
\begin{align}
\E\left [\sum_{t=1}^T \lVert \bm Q(t)\rVert_1\right ]&\le -\frac{V}{\epsilon_W} \E\left [\sum_{t=1}^T \bigg (g_t(\mathring{\bm \lambda}(t)) - g_t(\bm \lambda(t))\bigg )\right ] + f(T) + \nonumber \\
&\quad g(T) \E\left [\sum_{t=1}^T \lVert \bm Q(t)\rVert_1\right ]^{3/4} \log \E\left [M \sum_{t=1}^T \lVert \bm Q(t)\rVert_1\right ] + \nonumber \\
&\quad h(T) \E\left [\sum_{t=1}^T \lVert \bm Q(t)\rVert_1\right ]^{7/8}.\label{eq:multi-hop utility self-bounding}
\end{align}
where
\begin{align*}
f(T)&= \epsilon_W^{-1} \O\left ((N^2 (2NM+R)^2 + \epsilon_W N^2 (2NM+R)) C_W T + \frac{R(2NM+R)}{r^7} d^{14/3} (C^\lambda T^{1/2 - \delta_\lambda})^2\right .+\\
&\qquad \qquad \left .\left (\frac{R}{r} d^{2/3}+R\right ) (C^r T^{1/2 - \delta_r})^{1/4} V(L+G) T^{3/4}+\frac 12 N^2 ((NM)^2+2(NM)^2+2R^2) T\right ),\\
g(T)&=\epsilon_W^{-1} \O\left ((2NM+R)^{1/4} M \sqrt{1+C^a T^{1/2 - \delta_a}} \log T \right ), \\
h(T)&=\epsilon_W^{-1} \O\left ((2NM+R)^{1/8} \left (\frac Rr d^{2/3} + R\right ) (C^\lambda T^{1/2 - \delta_\lambda})^{1/4}\right ).
\end{align*}

\paragraph{Step 1 (Develop a Coarse Average Queue Length Bound).}
Recall the assumption that $g_t$ is uniformly bounded by $[-G,G]$. Therefore, the first term on the RHS of \Cref{eq:multi-hop utility self-bounding} is bounded by $2\frac{V}{\epsilon_W} GT$ in absolute value.
In \Cref{lem:3/4 and 7/8 self-bounding}, we develop a self-bounding property that says, if $y\le f+y^{3/4} g \log y +y^{7/8}$, then $y=\O\left (f+g^4 \log^8\left (2(f^{1/8}+g^{1/2}+h)^2\right )+h^8\right )$.
Therefore, applying it to \Cref{eq:multi-hop utility self-bounding}, we have
\begin{align}
\E\left [\sum_{t=1}^T \lVert \bm Q(t)\rVert_1\right ]&\le \O\left (2\frac{V}{\epsilon_W} GT + f(T) + g(T)^4 \log^8\left (2(f(T)^{1/8}+g(T)^{1/2}+h(T))^2\right ) + h(T)^8\right ) \nonumber\\
&=\O\left (\frac{V}{\epsilon_W} GT + f(T)\right ) + \O_T\left (T^{1-2\delta_a} \log^8 \left (T^{1/4} + T^{1/4-\delta_a/2} + T^{1/4-\delta_\lambda/2}\right ) + T^{1-2\delta_\lambda}\right ) \nonumber\\
&=\O\left (\frac{V}{\epsilon_W} GT + f(T)\right )+o_T(1). \label{eq:coarse queue length bound}
\end{align}

As mentioned in the proof sketch of this theorem, this only gives a $\frac 1T\E[\sum_{t=1}^T \lVert \bm Q(t)\rVert_1]=\O_T(V)$ bound on the average queue length, which violates the system stability condition \Cref{eq:average queue}. However, this inequality can be used to derive the polynomial convergence result on the utility, which in turn refines the average queue length bound.

\paragraph{Step 2 (Yield Polynomial Convergence on the Utility).}
Moving the difference in the average utility in \Cref{eq:multi-hop utility self-bounding} to the LHS, we have
\begin{align*}
\frac{V}{\epsilon_W} \E\left [\sum_{t=1}^T \bigg (g_t(\mathring{\bm \lambda}(t)) - g_t(\bm \lambda(t))\bigg )\right ]&\le -\E\left [\sum_{t=1}^T \lVert \bm Q(t)\rVert_1\right ] + f(T) + \nonumber \\
&\quad g(T) \E\left [\sum_{t=1}^T \lVert \bm Q(t)\rVert_1\right ]^{3/4} \log \E\left [M \sum_{t=1}^T \lVert \bm Q(t)\rVert_1\right ] + \nonumber \\
&\quad h(T) \E\left [\sum_{t=1}^T \lVert \bm Q(t)\rVert_1\right ]^{7/8}.
\end{align*}

Plugging in the just-derived bound on average queue length, namely \Cref{eq:coarse queue length bound}, we have
\begin{align*}
&\quad \frac{V}{\epsilon_W} \E\left [\sum_{t=1}^T \bigg (g_t(\mathring{\bm \lambda}(t)) - g_t(\bm \lambda(t))\bigg )\right ]\\
&\le 0 + f(T) + g(T) \O\left (\left (\frac{V}{\epsilon_W} GT + f(T)\right )^{3/4} \log \left (\frac{V}{\epsilon_W} GT + f(T)\right ) + h(T) \left (\frac{V}{\epsilon_W} GT + f(T)\right )^{7/8}\right )\\
&=f(T) + \O_T \left (T^{1/4 - \delta_a/2} (VT)^{3/4} \log(VT) + T^{1/8 - \delta_\lambda / 4} (VT)^{7/8}\right ).
\end{align*}

According to the assumption that $V=o_T(\min\{T^{2\delta_a/3},T^{2\delta_\lambda/7}\})$, the second term on the RHS is of order $o_T(T)$. Therefore, we have
\begin{align}
&\quad \frac 1T \E\left [\sum_{t=1}^T \left (g_t(\mathring{\bm \lambda}(t)) - g_t(\bm \lambda(t)))\right )\right ]=\frac{\epsilon_W}{VT} f(T) + \frac{\epsilon_W}{VT} o_T(T) \nonumber \\
&=(VT)^{-1} \O\left ((N^2 (2NM+R)^2 + \epsilon_W N^2 (2NM+R)) C_W T + \frac{R(2NM+R)}{r^7} d^{14/3} (C^\lambda T^{1/2 - \delta_\lambda})^2\right .+ \nonumber\\
&\qquad \qquad \qquad \left .\left (\frac{R}{r} d^{2/3}+R\right ) (C^r T^{1/2 - \delta_r})^{1/4} V(L+G) T^{3/4}+\frac 12 N^2 ((NM)^2+2(NM)^2+2R^2) T\right )+o_T(V^{-1}) \nonumber\\
&=\O\left (\frac{(N^2 (2NM+R)^2 + \epsilon_W N^2 (2NM+R)) C_W + (N^4 M^2+N^2 R^2)}{V}\right ) + o_T(V^{-1}). \label{eq:refined utility bound}
\end{align}

The second conclusion of this theorem follows.

\paragraph{Step 3 (Refine the Average Queue Length Bound).}
Now we are ready to refine our average queue length bound using \Cref{eq:refined utility bound}. Instead of controlling the utility with the uniform boundedness assumption that $g_t\in [-G,G]$, we utilize the just-derived convergence result \Cref{eq:refined utility bound}.

Specifically, again applying the self-bounding property in \Cref{lem:3/4 and 7/8 self-bounding} to \Cref{eq:multi-hop utility self-bounding} but instead replacing the first term on the RHS with \Cref{eq:refined utility bound}, we get
\begin{align*}
\E\left [\sum_{t=1}^T \lVert \bm Q(t)\rVert_1\right ]&\le \O\left (f(T) + o_T(T) + f(T) + g(T)^4 \log^8\left (2(f(T)^{1/8}+g(T)^{1/2}+h(T))^2\right ) + h(T)^8\right )\\
&=\O(f(T)) + \O\left (g(T)^4 \log^8\left (2(f(T)^{1/8}+g(T)^{1/2}+h(T))^2\right ) + h(T)^8\right ) + o_T(T) \\
&=\O\left (\frac{(N^2 (2NM+R)^2 + \epsilon_W N^2 (2NM+R)) C_W + (N^4 M^2+N^2 R^2)}{\epsilon_W} T\right ) + o_T(T),
\end{align*}
which gives our first conclusion as well.
\end{proof}

\begin{lemma}[Calculations when Proving \Cref{thm:multi-hop utility main theorem}]\label{lem:multi-hop utility tedious calc}
Under the conditions of \Cref{thm:multi-hop utility main theorem}, we have
\begin{align*}
\E\left [\sum_{t=1}^T \lVert \bm Q(t)\rVert_1\right ]&\le -\frac{V}{\epsilon_W} \E\left [\sum_{t=1}^T \bigg (g_t(\mathring{\bm \lambda}(t)) - g_t(\bm \lambda(t))\bigg )\right ] + f(T) + \\
&\quad g(T) \E\left [\sum_{t=1}^T \lVert \bm Q(t)\rVert_1\right ]^{3/4} \log \E\left [M \sum_{t=1}^T \lVert \bm Q(t)\rVert_1\right ] + \\
&\quad h(T) \E\left [\sum_{t=1}^T \lVert \bm Q(t)\rVert_1\right ]^{7/8}.
\end{align*}
where
\begin{align*}
f(T)&= \epsilon_W^{-1} \O\left ((N^2 (2NM+R)^2 + \epsilon_W N^2 (2NM+R)) C_W T + \frac{R(2NM+R)}{r^7} d^{14/3} (C^\lambda T^{1/2 - \delta_\lambda})^2\right .+\\
&\qquad \qquad \left .\left (\frac{R}{r} d^{2/3}+R\right ) (C^r T^{1/2 - \delta_r})^{1/4} V(L+G) T^{3/4}+\frac 12 N^2 ((NM)^2+2(NM)^2+2R^2) T\right ),\\
g(T)&=\epsilon_W^{-1} \O\left ((2NM+R)^{1/4} M \sqrt{1+C^a T^{1/2 - \delta_a}} \log T \right ), \\
h(T)&=\epsilon_W^{-1} \O\left ((2NM+R)^{1/8} \left (\frac Rr d^{2/3} + R\right ) (C^\lambda T^{1/2 - \delta_\lambda})^{1/4}\right ).
\end{align*}
\end{lemma}
\begin{proof}
From the network stability assumption, we derived in \Cref{lem:multi-hop utility network stability} that
\begin{align*}
&\quad \epsilon_W \E\left [\sum_{t=1}^T \sum_{n\in \mN} \sum_{k\in \mN} Q_n^{(k)}(t)\right ] - (N^2 (2NM+R)^2 + \epsilon_W N^2 (2NM+R)) C_W T \nonumber\\
&\le -\E\left [\sum_{t=1}^T \sum_{(n,m)\in \mL}\sum_{k\in \mN} \mathring \mu_{n,m}^{(k)}(t) (Q_m^{(k)}(t) - Q_n^{(k)}(t))\right ] - \E\left [\sum_{t=1}^T \sum_{n\in \mN} \sum_{k\in \mN} Q_n^{(k)}(t) \mathring \lambda_n^{(k)}(t)\right ].
\end{align*}

Recall the \OLOAlg guarantee in \Cref{thm:multi-hop stability decision making} that
\begin{align*}
&\quad \E\left [\sum_{t=1}^T \sum_{(n,m)\in \mL} \sum_{k\in \mN} (\mu_{n,m}^{(k)}(t)-\mathring{\mu}_{n,m}^{(k)}(t))\left (Q_m^{(k)}(t)-Q_n^{(k)}(t)\right )\right ]\\
&=\O\left (M\sqrt{1+P_T^a} \E\left [\sqrt{\sum_{t=1}^T \lVert \bm Q(t)\rVert_2^2} \log T \log \left (\max_{t=1}^T \max_{(n,m)\in \mL} M \lVert \bm Q_m(t) - \bm Q_n(t)\rVert_\infty\right )\right ]\right )\text{,}
\end{align*}
and the Bandit Convex Optimization guarantee in \Cref{thm:multi-hop utility decision making} that
\begin{align*}
&\scalemath{0.95}{\quad \E\left [\sum_{t=1}^T \left ((\langle \bm Q(t),\bm \lambda(t)\rangle - Vg_t(\bm \lambda(t))) - (\langle \bm Q(t),\mathring{\bm \lambda}(t)\rangle - Vg_t(\mathring{\bm \lambda}(t)))\right )\right ]}\\
&\scalemath{0.95}{=\O\left (\frac{R(2NM+R)}{r^7} d^{14/3} (C^\lambda T^{1/2-\delta_\lambda})^2\right )+\O\left (\E\left [\left (\frac{R}{r} d^{2/3}+R\right ) (C^\lambda T^{1/2 - \delta_\lambda})^{1/4} \left (\sum_{t=1}^T \left (\lVert \bm Q(t)\rVert_2 + V(L+G)\right )^{4/3}\right )^{3/4}\right ]\right ),}
\end{align*}
we therefore have
\begin{align*}
&\quad \epsilon_W \E\left [\sum_{t=1}^T \sum_{n\in \mN} \sum_{k\in \mN} Q_n^{(k)}(t)\right ] - (N^2 (2NM+R)^2 + \epsilon_W N^2 (2NM+R)) C_W T \nonumber\\
&\le -\E\left [\sum_{t=1}^T \sum_{(n,m)\in \mL}\sum_{k\in \mN} \mu_{n,m}^{(k)}(t) (Q_m^{(k)}(t) - Q_n^{(k)}(t))\right ] + \\
&\quad \O\left (M \sqrt{1+P_T^a} \E\left [\sqrt{\sum_{t=1}^T \lVert \bm Q(t)\rVert_2^2} \log T \log \left (\max_{t=1}^T \max_{(n,m)\in \mL} M \lVert \bm Q_m(t) - \bm Q_n(t)\rVert_\infty\right )\right ]\right )+\\
&\quad - \E\left [\sum_{t=1}^T \sum_{n\in \mN} \sum_{k\in \mN} Q_n^{(k)}(t) \lambda_n^{(k)}(t)\right ] - V \E\left [\sum_{t=1}^T \bigg (g_t(\mathring{\bm \lambda}(t))-g_t(\bm \lambda(t))\bigg )\right ]+\\
&\quad \O\left (\frac{R(2NM+R)}{r^7} d^{14/3} (C^\lambda T^{1/2-\delta_\lambda})^2\right )+\\
&\quad \O\left (\E\left [\left (\frac{R}{r} d^{2/3}+R\right ) (C^\lambda T^{1/2 - \delta_\lambda})^{1/4} \left (\sum_{t=1}^T \left (\lVert \bm Q(t)\rVert_2 + V(L+G)\right )^{4/3}\right )^{3/4}\right ]\right ).
\end{align*}

Further plugging in the Lyapunov DPP calculation in \Cref{eq:multi-hop utility Lyapunov} (which controls the three $\E[\sum_{t=1}^T \cdots]$ terms outside $\O$ on the RHS), we have
\begin{align*}
&\quad \epsilon_W \E\left [\sum_{t=1}^T \sum_{n\in \mN} \sum_{k\in \mN} Q_n^{(k)}(t)\right ] - (N^2 (2NM+R)^2 + \epsilon_W N^2 (2NM+R)) C_W T \nonumber\\
&\le \O\left (M \sqrt{1+P_T^a} \E\left [\sqrt{\sum_{t=1}^T \lVert \bm Q(t)\rVert_2^2} \log T \log \left (\max_{t=1}^T \max_{(n,m)\in \mL} M \lVert \bm Q_m(t) - \bm Q_n(t)\rVert_\infty\right )\right ]\right )+\\
&\quad \O\left (\frac{R(2NM+R)}{r^7} d^{14/3} (C^\lambda T^{1/2-\delta_\lambda})^2\right )+\\
&\quad \O\left (\E\left [\left (\frac{R}{r} d^{2/3}+R\right ) (C^\lambda T^{1/2 - \delta_\lambda})^{1/4} \left (\sum_{t=1}^T \left (\lVert \bm Q(t)\rVert_2 + V(L+G)\right )^{4/3}\right )^{3/4}\right ]\right )+\\
&\quad \frac 12 N^2 ((NM)^2+2(NM)^2+2R^2)+V  \E\left [\sum_{t=1}^T (g_t(\bm \lambda(t))-g_t(\mathring{\bm \lambda}(t)))\right ].
\end{align*}

For notational simplicity, we can abbreviate this inequality as
\begin{align*}
&\quad \E\left [\sum_{t=1}^T \lVert \bm Q(t)\rVert_1\right ]\le -\frac{V}{\epsilon_W} \E\left [\sum_{t=1}^T \bigg (g_t(\mathring{\bm \lambda}(t)) - g_t(\bm \lambda(t))\bigg )\right ]+\tilde f(T)+\\
&\quad \tilde g(T) \sqrt{\E\left [\sum_{t=1}^T \lVert \bm Q(t)\rVert_2^2\right ]} \log \left (\max_{t=1}^T \max_{(n,m)\in \mL} M \lVert \bm Q_m(t)-\bm Q_n(t)\rVert_\infty\right ) + \tilde h(T) \left (\E\left [\sum_{t=1}^T \lVert \bm Q(t)\rVert_2^{4/3}\right ]\right )^{3/4},
\end{align*}
where
\begin{align*}
\tilde f(T)&= \epsilon_W^{-1} \O\left ((N^2 (2NM+R)^2 + \epsilon_W N^2 (2NM+R)) C_W T + \frac{R(2NM+R)}{r^7} d^{14/3} (C^\lambda T^{1/2 - \delta_\lambda})^2\right .+\\
&\qquad \qquad \left .\left (\frac{R}{r} d^{2/3}+R\right ) (C^r T^{1/2 - \delta_r})^{1/4} V(L+G) T^{3/4}+\frac 12 N^2 ((NM)^2+2(NM)^2+2R^2) T\right ),\\
\tilde g(T)&=\epsilon_W^{-1} \O\left (M \sqrt{1+C^a T^{1/2 - \delta_a}} \log T \right ), \\
\tilde h(T)&=\epsilon_W^{-1} \O\left (\left (\frac Rr d^{2/3} + R\right ) (C^\lambda T^{1/2 - \delta_\lambda})^{1/4}\right ).
\end{align*}

To handle the $\tilde g(T)$-related term, we use the argument same to that of \Cref{sec:multi-hop stability main theorem}:
\Cref{lem:2 upper bound} states that if $x_1=0$, $x_2,\ldots,x_T \ge 0$, and $\lvert x_{t+1}-x_t\rvert\le 1$, then $\sum_{t=1}^T x_t^2 = \O\left ((\sum_{t=1}^T x_t)^{3/2}\right )$. From \Cref{lem:queue length increment}, any single queue $Q_n^{(k)}(t)$ satisfies $\lvert Q_n^{(k)}(t+1)-Q_n^{(k)}(t)\rvert \le (2NM + R)$.
Hence, applying \Cref{lem:2 upper bound} to $\{Q_n^{(k)}(t) / (2NM+R)\}_{t\in [T]}$ to every $n\in \mN$ and $k\in \mN$, we have
\begin{equation*}
\sum_{t=1}^T \lVert \bm Q(t)\rVert_2^2 = (2NM+R)^2 \sum_{t=1}^T \sum_{n\in \mN} \sum_{k\in \mN} \left (\frac{Q_n^{(k)}(t)}{2NM+R}\right )^2 =\O\left (\sqrt{2NM+R} \left (\sum_{t=1}^T \lVert \bm Q(t)\rVert_1\right )^{1.5}\right ).
\end{equation*}

Further noticing that
\begin{equation*}
\max_{t=1}^T \max_{(n,m)\in \mL} M \lVert \bm Q_m(t) - \bm Q_n(t)\rVert_\infty\le \sum_{t=1}^T M \sum_{n\in \mN} \lVert \bm Q_n(t)\rVert_1\le M \sum_{t=1}^T \lVert \bm Q(t)\rVert_1,
\end{equation*}
the $\tilde g(T)$-related term then becomes
\begin{align*}
&\quad \tilde g(T) \sqrt{\E\left [\sum_{t=1}^T \lVert \bm Q(t)\rVert_2^2\right ]} \log \left (\max_{t=1}^T \max_{(n,m)\in \mL} M \lVert \bm Q_m(t)-\bm Q_n(t)\rVert_\infty\right )\\
&=\tilde g(T) \O\left ((2NM+R)^{1/4} \E\left [\left (\sum_{t=1}^T \lVert \bm Q(t)\rVert_1\right )^{3/4} \log \left (M \sum_{t=1}^T \lVert \bm Q(t)\rVert_1\right )\right ]\right ).
\end{align*}

Noticing that $x\mapsto x^{3/4} \log (M x)$ is concave when $x$ is large enough, Jensen inequality then gives
\begin{equation*}
\O\left (\E\left [\left (\sum_{t=1}^T \lVert \bm Q(t)\rVert_1\right )^{3/4} \log \left (M \sum_{t=1}^T \lVert \bm Q(t)\rVert_1\right ) \right ]\right )= \O\left (\E\left [\sum_{t=1}^T \lVert \bm Q(t)\rVert_1\right ]^{3/4} \log \E \left [ \sum_{t=1}^T \lVert \bm Q(t)\rVert_1\right ]\right ).
\end{equation*}

Moreover, we handle the $\tilde h(T)$-related term using \Cref{lem:4/3 upper bound}, a variant of \Cref{lem:2 upper bound} which states that if $x_1=0$, $x_2,\ldots,x_T \ge 0$, and $\lvert x_{t+1}-x_t\rvert\le 1$, then $\sum_{t=1}^T x_t^{4/3} = \O\left ((\sum_{t=1}^T x_t)^{7/6}\right )$. Hence, still applying it to $\{Q_n^{(k)}(t) / (2NM + R)\}_{t\in [T]}$ for every $n\in \mN$ and $k\in \mN$,
\begin{align*}
\sum_{t=1}^T \lVert \bm Q(t)\rVert_2^{4/3}&=\sum_{t=1}^T \left (\sum_{n\in \mN} \sum_{k\in \mN} Q_n^{(k)}(t)^2\right )^{2/3}\le \sum_{t=1}^T \sum_{n\in \mN} \sum_{k\in \mN} Q_n^{(k)}(t)^{4/3}\\
&=\O\left (\left (2NM+R\right )^{1/6} \left (\sum_{t=1}^T \lVert \bm Q(t)\rVert_1\right )^{7/6}\right ).
\end{align*}

Therefore, we have
\begin{align*}
\E\left [\sum_{t=1}^T \lVert \bm Q(t)\rVert_1\right ]&\le -\frac{V}{\epsilon_W} \E\left [\sum_{t=1}^T \bigg (g_t(\mathring{\bm \lambda}(t)) - g_t(\bm \lambda(t))\bigg )\right ] + \tilde f(T) + \\
&\quad \tilde g(T) \O\left ((2NM+R)^{1/4}\right ) \E\left [\sum_{t=1}^T \lVert \bm Q(t)\rVert_1\right ]^{3/4} \log \E\left [M \sum_{t=1}^T \lVert \bm Q(t)\rVert_1\right ] + \\
&\quad \tilde h(T) \O\left ((2NM+R)^{1/8}\right ) \E\left [\sum_{t=1}^T \lVert \bm Q(t)\rVert_1\right ]^{7/8}.
\end{align*}

Setting $f(T)=\tilde f(T)$, $g(T)=\tilde g(T) \O((2NM+R)^{1/4})$, and $h(T)=\tilde h(T) \O((2NM+R)^{1/8})$ gives our conclusion.
\end{proof}

\section{Auxiliary Lemmas}
The first lemma extends the famous summation lemma $\sum_{t=1}^T \frac{x_t}{\sqrt{\sum_{s=1}^t x_s}}=\O\left (\sqrt{\sum_{t=1}^T x_t}\right )$ \citep{auer2002nonstochastic}.
\begin{lemma}\label{lem:3/4 summation lemma}
For non-negative real numbers $x_1,x_2,\ldots,x_T\in \mathbb R$, we have
\begin{equation*}
\sum_{t=1}^T \frac{x_t}{\left (\sum_{s=1}^t x_s\right )^{1/4}} \le 2 \left (\sum_{t=1}^T x_t\right )^{3/4}.
\end{equation*}
\end{lemma}
\begin{proof}
Prove by induction. The case when $T=1$ is obvious. Suppose that the conclusion holds for $T-1$, then consider some $x_T$:
\begin{align*}
&\quad \sum_{t=1}^T \frac{x_t}{\left (\sum_{s=1}^t x_s\right )^{1/4}}\\
&= \sum_{t=1}^{T-1} \frac{x_t}{\left (\sum_{s=1}^t x_s\right )^{1/4}} + \frac{x_T}{\left (\sum_{t=1}^T x_t\right )^{1/4}}\\
&\le 2 \left (\sum_{t=1}^{T-1} x_t\right )^{3/4} + \frac{x_T}{\left (\sum_{t=1}^T x_t\right )^{1/4}},
\end{align*}
so it suffices to prove $\left . x_T \middle / \left (\sum_{t=1}^T x_t\right )^{1/4}\right .\le 2 \left (\sum_{t=1}^T x_t\right )^{3/4} - 2 \left (\sum_{t=1}^{T-1} x_t\right )^{3/4}$. Notice that
\begin{align*}
&\quad \left (\left (\sum_{t=1}^T x_t\right )^{3/4} - \left (\sum_{t=1}^{T-1} x_t\right )^{3/4}\right )\left (\left (\sum_{t=1}^T x_t\right )^{1/4}+\left (\sum_{t=1}^{T-1} x_t\right )^{1/4}\right )\\
&=\left (\sum_{t=1}^T x_t \right ) + \left (\sum_{t=1}^T x_t \right )^{3/4} \left (\sum_{t=1}^{T-1} x_t \right )^{1/4} - \left (\sum_{t=1}^T x_t \right )^{1/4} \left (\sum_{t=1}^{T-1} x_t \right )^{3/4} - \left (\sum_{t=1}^{T-1} x_t \right ) \\
&\ge x_T,
\end{align*}
where the last inequality uses $\sum_{t=1}^T x_t\ge \sum_{t=1}^{T-1} x_t$ (follows from $x_T\ge 0$). Hence, due to the fact that
\begin{equation*}
\left . x_T \middle / 2 \left (\sum_{t=1}^T x_t\right )^{1/4}\right .\le \left . x_T \middle / \left (\sum_{t=1}^T x_t\right )^{1/4} + \left (\sum_{t=1}^{T-1} x_t\right )^{1/4}\right .,
\end{equation*}
the conclusion holds for $T$ as well.
\end{proof}

\begin{lemma}\label{lem:4/3 lower bound}
Suppose that $x_1=0$, $x_2,x_3,\ldots,x_T\ge 0$, and $\lvert x_t-x_{t-1}\rvert\le 1$, $\forall t=2,3,\ldots,T$, then
\begin{equation*}
\sum_{t=1}^T x_t^{4/3} \ge 4^{-7/3} x_T^{7/3}.
\end{equation*}
\end{lemma}
\begin{proof}
As adjacent $x_t$'s differ by no more than $1$, $\lfloor x_T\rfloor < T$ and $x_{T-t}\ge x_T - t$. Therefore,
\begin{equation*}
\sum_{t=1}^T x_t^{4/3} \ge \sum_{t=0}^{\lfloor x_T\rfloor} x_{T-t}^{4/3}\ge \sum_{t=0}^{\lfloor x_T\rfloor} (x_T-t)^{4/3} \ge \sum_{t=0}^{\lfloor x_T\rfloor} \left (t^{4/3} + (x_T - \lfloor x_T\rfloor)^{4/3}\right )\ge \sum_{t=0}^{\lfloor x_T\rfloor} t^{4/3},
\end{equation*}
where the last step uses $(a+b)^{4/3} \ge a^{4/3} + b^{4/3}$.
As
\begin{equation*}
\sum_{i=0}^n i^{4/3} \ge \sum_{i=\lfloor \frac n2 \rfloor}^n i^{4/3} \ge \left (n - \left \lfloor \frac n2 \right \rfloor\right ) \left (\left \lfloor \frac n2 \right \rfloor\right )^{4/3} \ge \left (\left \lfloor \frac n2 \right \rfloor\right )^{7/3},
\end{equation*}
we have $\sum_{t=1}^T x_t^{4/3} \ge (\lfloor \frac{x_T}{2}\rfloor )^{7/3}$. If $x_T\ge 4$, then $\lfloor \frac{x_T}{2}\rfloor \ge \frac{x_T}{2}-1 \ge \frac{x_T}{4}$, giving the conclusion. Otherwise, \textit{i.e.}, $x_T<4$, then we naturally have $\sum_{t=1}^T x_t^{4/3} \ge (\frac{x_T}{4})^{4/3}\ge (\frac{x_T}{4})^{7/3}$, so our conclusion still follows.
\end{proof}

\begin{lemma}[{\citep[Lemma 4]{huang2023queue}}]\label{lem:2 upper bound}
If $x_1=0$, $x_2,x_3,\ldots,x_T\ge 0$, and $\lvert x_t-x_{t-1}\rvert\le 1$, $\forall t=2,3,\ldots,T$, then
\begin{equation*}
\sum_{t=1}^T x_t^2 \le 4 \left (\sum_{t=1}^T x_t\right )^{3/2}.
\end{equation*}
\end{lemma}

\begin{lemma}\label{lem:4/3 upper bound}
If $x_1=0$, $x_2,x_3,\ldots,x_T\ge 0$, and $\lvert x_t-x_{t-1}\rvert\le 1$, $\forall t=2,3,\ldots,T$, then
\begin{equation*}
\sum_{t=1}^T x_t^{4/3} \le 2^{1/6} \left (\sum_{t=1}^T x_t\right )^{7/6}.
\end{equation*}
\end{lemma}
\begin{proof}
Imitating the proof of \Cref{lem:2 upper bound} \citep[Lemma 4]{huang2023queue}, we short $x_1,x_2,\ldots,x_T$ as $y_1\le y_2\le \cdots \le y_T$. According to the original proof, $y_T=\max_{t\in [T]} x_t\le (2 \sum_{t=1}^T x_t)^{1/2}$. Hence,
\begin{equation*}
\sum_{t=1}^T x_t^{4/3} \le y_T^{1/3} \sum_{t=1}^T x_t \le \left (2 \sum_{t=1}^T x_t\right )^{1/6} \left (\sum_{t=1}^T x_t\right ) = 2^{1/6} \left (\sum_{t=1}^T x_t\right )^{7/6}.
\end{equation*}
\end{proof}

The following two lemmas are similar to Lemma 5 of \citet{huang2023queue}.
\begin{lemma}\label{lem:3/4 with log self-bounding}
If $y\le f + y^{3/4} g \log y$ and $f,g\ge 1$, then
\begin{equation*}
y^{1/4}\le f^{1/4} + g \log \left (2(f^{1/4} + g)^2\right ).
\end{equation*}
\end{lemma}
\begin{proof}
Let $i(z) = z^4 - z^3 g \log (z^4) - f$. Notice that
\begin{align*}
&\quad \left (f^{1/4} + g \log \left (2(f^{1/4} + g)^2\right )\right )^4\\
&\ge f + 4 \left (f^{1/4} + g \log \left (2(f^{1/4} + g)^2\right )\right )^3 g \log \left (2(f^{1/4} + g)^2\right ).
\end{align*}

As $f,g\ge 1$, we have $2 (f^{1/4} + g)^2\ge f^{1/4} + g \log (2 (f^{1/4} + g)^2)$. Applying this relationship to the second log on the RHS of the inequality above, we yield
\begin{equation*}
i\left (f^{1/4} + g \log \left (2(f^{1/4} + g)^2\right )\right )\ge 0.
\end{equation*}

On the other hand, from the conditions, we know $i(y^{1/4})\le 0$. Hence, if we can prove that $i(z)$ is monotone (at least for a range of $z$), we can conclude that $y^{1/4}\le f^{1/4} + g \log \left (2(f^{1/4} + g)^2\right )$, thus giving our conclusion.
To conclude the monotonity of $i(z)$, we calculate its derivative as
\begin{equation*}
i'(z) = 4z^3 - 3z^2 g \log (z^4) - 4z^2 g.
\end{equation*}

Thus, if we only consider the case where $z\ge 0$, $i'(z)\ge 0$ holds when $z\ge 3g\log z + g$. Denoting the larger root of $z=3g\log z+g$ as $z_0$ (in case it has no root, let $z_0=1$), we know $i(z)$ is increasing in $[z_0,+\infty)$.
Further observing that $f^{1/4} + g \log \left (2(f^{1/4} + g)^2\right )$ indeed satisfies $z\ge 3g\log z + g$, we know $f^{1/4} + g \log \left (2(f^{1/4} + g)^2\right )\ge z_0$.

On the other hand, from the assumption that $y\le f + y^{3/4} g \log y$, we know $i(y^{1/4})\le 0\le i(f^{1/4} + g \log \left (2(f^{1/4} + g)^2\right ))$. Our conclusion follows from discussing the relationship between $y^{1/4}$ and $z_0$: If $y^{1/4}\le z_0$, then we immediately have
\begin{equation*}
y^{1/4}\le z_0\le f^{1/4} + g \log \left (2(f^{1/4} + g)^2\right ).
\end{equation*}
Otherwise, according to the monotonity of $i(z)$ when $z\ge z_0$, we still have
\begin{equation*}
y^{1/4}\le f^{1/4} + g \log \left (2(f^{1/4} + g)^2\right ),
\end{equation*}
as claimed.
\end{proof}

\begin{lemma}\label{lem:3/4 and 7/8 self-bounding}
If $y \le f + y^{3/4} g \log y + y^{7/8} h$ and $f,g,h\ge 1$, then
\begin{equation*}
y^{1/8}\le f^{1/8} + g^{1/2} \log \left (2(f^{1/8}+g^{1/2}+h)^2\right ) + h.
\end{equation*}
\end{lemma}
\begin{proof}
Let $i(z) = z^8 - f - z^6 g \log (z^8) - z^7 h$. Notice that
\begin{align*}
&\quad \left (f^{1/8} + g^{1/2} \log \left (2(f^{1/8}+g^{1/2}+h)^2\right ) + h\right )^8 \\
&=\left (f^{1/8} + g^{1/2} \log \left (2(f^{1/8}+g^{1/2}+h)^2\right ) + h\right )^7 \left (f^{1/8} + g^{1/2} \log \left (2(f^{1/8}+g^{1/2}+h)^2\right )\right ) + \\
&\quad \left (f^{1/8} + g^{1/2} \log \left (2(f^{1/8}+g^{1/2}+h)^2\right ) + h\right )^7 h\\
&\ge \left (f^{1/8} + g^{1/2} \log \left (2(f^{1/8}+g^{1/2}+h)^2\right ) + h\right )^6 f^{1/4} + \\
&\quad \left (f^{1/8} + g^{1/2} \log \left (2(f^{1/8}+g^{1/2}+h)^2\right ) + h\right )^6 g \log \left (2(f^{1/8}+g^{1/2}+h)^2\right ) + \\
&\quad \left (f^{1/8} + g^{1/2} \log \left (2(f^{1/8}+g^{1/2}+h)^2\right ) + h\right )^7 h \\
&\ge f + \left (f^{1/8} + g^{1/2} \log \left (2(f^{1/8}+g^{1/2}+h)^2\right ) + h\right )^6 g \log \left (2(f^{1/8}+g^{1/2}+h)^2\right ) + \\
&\quad \left (f^{1/8} + g^{1/2} \log \left (2(f^{1/8}+g^{1/2}+h)^2\right ) + h\right )^7 h.
\end{align*}

As $f,g,h\ge 1$, we have $2(f^{1/8}+g^{1/2}+h)^2\ge f^{1/8} + g^{1/2} \log \left (2(f^{1/8}+g^{1/2}+h)^2\right ) + h$. Applying it to the $(\cdots)^6 g \log(2(f^{1/8}+g^{1/2}+h)^2)$ term on the RHS, we have
\begin{equation*}
i\left (f^{1/8} + g^{1/2} \log \left (2(f^{1/8}+g^{1/2}+h)^2\right )+h\right )\ge 0
\end{equation*}

On the other hand, from the conditions we know $i(y^{1/8})\le 0$. Hence, we again want to prove the monotony of $i(z)$, which gives the conclusion that $y^{1/8}\le f^{1/8} + g^{1/2} + h$.

Calculating the derivative of $i(z)$, we have
\begin{equation*}
i'(z) = 8z^7 - 6z^5 g \log(z^8) - 8 z^5 g - 7z^6 h.
\end{equation*}

Thus, if we only consider the case where $z\ge 1$, $i'(z)\ge 0$ holds when $8z^2 - 6 g \log(z^8) - 8 g - 7z h\ge 0$.
As $z^2$ is convex and $6g\log (z^8) + 8g + 7zh$ is concave, there are at most two intersections.
Hence, again denoting the larger root of $8z^2 - 6 g \log(z^8) - 8 g - 7z h = 0$ as $z_0$ (in case there is no root, let $z_0=1$), $i(z)$ is monotonic in $[z_0,+\infty)$.

As $f,g,h\ge 1$, $f^{1/8} + g^{1/2} \log \left (2(f^{1/8}+g^{1/2}+h)^2\right ) + h \ge z_0$ always holds. The conclusion follows by discussing the relationship of $y^{1/8}$ and $z_0$: If $y^{1/8}\le z_0$, we directly have
\begin{equation*}
y^{1/8}\le z_0\le f^{1/8} + g^{1/2} \log \left (2(f^{1/8}+g^{1/2}+h)^2\right ) + h.
\end{equation*}
Otherwise, we can still conclude
\begin{equation*}
y^{1/8}\le f^{1/8} + g^{1/2} \log \left (2(f^{1/8}+g^{1/2}+h)^2\right ) + h
\end{equation*}
from the monotonity of $i(z)$ when $z\ge z_0$.
\end{proof}

\end{document}